\documentclass[11pt,twoside]{amsart}

\usepackage[english]{babel}
\usepackage{csquotes}
\usepackage{amsfonts}
\usepackage{hyperref}
\usepackage{etoolbox}
\usepackage{amsmath}
\usepackage{amssymb}
\usepackage{amsthm}
\usepackage{stmaryrd}
\usepackage{dsfont}
\usepackage{pifont}
\usepackage{caption}
\usepackage{subcaption}
\usepackage{mathrsfs}
\usepackage{graphicx}
\usepackage{enumerate}
\usepackage[all]{xy}
\usepackage{enumitem}
\usepackage{geometry}
\usepackage{amsfonts}
\usepackage{fancyhdr}
\usepackage{calc}
\usepackage{cancel}
\usepackage{accents}
\usepackage{abraces}
\usepackage{amsaddr}
\usepackage[new]{old-arrows}
\usepackage{tikz,tikz-3dplot}
\usetikzlibrary{decorations.markings}
\usetikzlibrary{shapes,arrows,plotmarks}
\usepackage[backend=biber,style=alphabetic,sorting=nty,maxnames=50]{biblatex}

\usepackage{color}
\definecolor{linkColor}{rgb}{0.0,0.0,0.554}
\definecolor{citeColor}{rgb}{0.0,0.0,0.554}
\definecolor{fileColor}{rgb}{0.0,0.0,0.554}
\definecolor{urlColor}{rgb}{0.0,0.0,0.554}
\definecolor{promptColor}{rgb}{0.0,0.0,0.589}
\definecolor{brkpromptColor}{rgb}{0.589,0.0,0.0}
\definecolor{gapinputColor}{rgb}{0.589,0.0,0.0}
\definecolor{gapoutputColor}{rgb}{0.0,0.0,0.0}
\usepackage{fancyvrb}
\usepackage{adjustbox}
\usepackage{helvet}

\definecolor{cof}{RGB}{219,144,71}
\definecolor{pur}{RGB}{186,146,162}
\definecolor{greeo}{RGB}{91,173,69}
\definecolor{greet}{RGB}{52,111,72}

\tdplotsetmaincoords{70}{165}

\newcommand{\changefont}{%
    \fontsize{8}{8}\selectfont
}
\newcommand{\actleft}{\rotatebox[origin=c]{-90}{$\circlearrowright$}}

\mathchardef\mhyphen="2D 

\title[Equivariant triangulations of tori and extension to Coxeter groups]{Equivariant triangulations of tori of compact Lie groups and hyperbolic extension to non-crystallographic Coxeter groups}

\author{Arthur Garnier}
\address{LAMFA, Universit\'e de Picardie Jules Verne, CNRS UMR 7352, \\33, rue Saint-Leu, 80000, Amiens, France.}
\email{arthur.garnier@math.cnrs.fr}

\theoremstyle{plain}
\newtheorem{prop}{Proposition}[subsection]
\newtheorem{prop-def}[prop]{Proposition-Definition}

\newtheorem{lem}[prop]{Lemma}
\newtheorem{theo}[prop]{Theorem}
\newtheorem{cor}[prop]{Corollary}
\newtheorem{rem}[prop]{Remark}
\newtheorem{definition}[prop]{Definition}
\newtheorem{exemple}[prop]{Example}

\newtheorem*{prop*}{Proposition}
\newtheorem*{prop-def*}{Proposition-Definition}
\newtheorem*{propri*}{Property}
\newtheorem*{lem*}{Lemma}
\newtheorem*{theo*}{Theorem}
\newtheorem*{cor*}{Corollary}
\newtheorem*{rem*}{Remark}
\newtheorem*{definition*}{Definition}
\newtheorem*{exemple*}{Example}
\newtheorem*{notation*}{Notation}

\newcommand{\lra}{\longrightarrow}

\newcommand{\ra}{\rightarrow}
\newcommand{\sdp}{\times\kern-.2em\vrule height1.1ex depth-.05ex}
\newcommand{\epi}{\lra \kern-.8em\ra}
\newcommand{\C}{{\mathbb C}}

\newcommand{\Q}{{\mathbb Q}}
\newcommand{\N}{{\mathbb N}}
\newcommand{\R}{{\mathbb R}}

\newcommand{\Z}{{\mathbb Z}}

\DeclareMathOperator{\ho}{Hom}

\DeclareMathOperator{\ima}{im}

\DeclareMathOperator{\codim}{codim}

\DeclareMathOperator\conv{conv}
\DeclareMathOperator\vertices{vert}

\DeclareMathOperator{\tr}{tr}

\newcommand{\longto}{\longrightarrow}
\newcommand{\Sph}{\mathbb{S}}

\newcommand{\Sym}{\mathfrak{S}}

\DeclareRobustCommand\longtwoheadrightarrow
     {\relbar\joinrel\twoheadrightarrow}

 \normalbaselines
\makeatletter
\newlength\@SizeOfCirc%
\newcommand{\CircleArrowRight}[1]{%
    \setlength{\@SizeOfCirc}{\maxof{\widthof{#1}}{\heightof{#1}}}%
    \tikz [x=1.0ex,y=1.0ex,line width=.12ex]%
        \draw [->,anchor=center]%
            node (0,0) {#1}%
            (0,0.8\@SizeOfCirc) arc (85:-240:0.8\@SizeOfCirc);%
}%
\newcommand{\CircleArrowLeft}[1]{%
    \setlength{\@SizeOfCirc}{\maxof{\widthof{#1}}{\heightof{#1}}}%
    \tikz [x=1.0ex,y=1.0ex,line width=.12ex]%
        \draw [<-,anchor=center]%
            node (0,0) {#1}%
            (0,0.8\@SizeOfCirc) arc (85:-240:0.8\@SizeOfCirc);%
}%
\makeatother


\date{\today}

\begin{document}

\begin{abstract}
Given a simple connected compact Lie group $K$ and a maximal torus $T$ of $K$, the Weyl group $W=N_K(T)/T$ naturally acts on $T$.

First, we use the combinatorics of the (extended) affine Weyl group to provide an explicit $W$-equivariant triangulation of $T$. We describe the associated $W$-dg-ring.

For a non-crystallographic Coxeter group $W$, using compact hyperbolic extensions rather than affine ones, we construct a compact $W$-manifold $\mathbf{T}(W)$, which is an analogue of a torus for $W$. We exhibit a $W$-equivariant triangulation of $\mathbf{T}(W)$ and compute the associated $W$-dg-ring. Also, we derive its homology representation.
\end{abstract}

\maketitle

\setcounter{section}{-1}
\section{Introduction}
 
\indent Let $K$ be a simple compact Lie group, $T<K$ a maximal torus and $W:=N_K(T)/T$ be the Weyl group. Our first aim is to provide an explicit $W$-equivariant triangulation of $T$ and to describe the associated cellular homology cochain complex, as a $W$-dg-ring. This study is part of a program of construction of $W$-equivariant cellular structures in Lie theory, and more precisely for the flag manifold $K/T$ (see \cite{chirivi-garnier-spreafico} for the case of $SL_3(\R)/B$) and the classifying space $B_T$ of $T$. Then, we will look at a generalization to non-crystallographic Coxeter groups: more precisely, to a non-crystallographic Coxeter group $W$, we associate a compact $W$-manifold $\mathbf{T}(W)$, which is an analogue of a torus for $W$.

Let $\mathfrak{t}$ denote the Lie algebra of $T$. The exponential map $\mathfrak{t} \to T$ induces a $W$-equivariant isomorphism of Lie groups $\mathfrak{t}/\Lambda \to T$, where $\Z\Phi^\vee\subset\Lambda\subset\mathfrak{t}$ is a $W$-lattice. The situation is simpler when $\pi_1(K)=1$ (i.e. $\Lambda=\Z\Phi^\vee$) and we summarize our first result as follows:

\begin{theo*}[\ref{cupproductforT}]
Let $K$ be a simply-connected simple compact Lie group, $T<K$ be a maximal torus and $W=N_K(T)/T$ be the associated Weyl group. If $W_{\rm{a}}$ denotes the affine Weyl group, then the fundamental alcove induces a $W_{\mathrm{a}}$-equivariant triangulation of the Lie algebra $\mathfrak{t}$ of $T$, whose $W_{\mathrm{a}}$-dg-ring $C^*_{\rm{cell}}(\mathfrak{t},W_{\mathrm{a}};\Z)$ is described in terms of parabolic cosets. This induces a $W$-equivariant triangulation of $T$ and the associated $W$-dg-ring is given by
\[C^*_{\rm{cell}}(T,W;\Z)=\mathrm{Def}^{W_{\mathrm{a}}}_W(C^*_{\rm{cell}}(\mathfrak{t},W_{\mathrm{a}};\Z)),\]
where $\mathrm{Def}^{W_{\mathrm{a}}}_W : \Z[W_{\mathrm{a}}]\mhyphen\mathbf{dgRing}\to\Z[W]\mhyphen\mathbf{dgRing}$ is the deflation functor.
\end{theo*}

In particular, if $P$ denotes the weight lattice associated to $(K,T)$, then we retrieve that
\[H^*(C^*_{\rm{cell}}(T,W;\Z))=H^*(T,\Z)=\Lambda^*(P).\]
In the general case, the \emph{extended affine Weyl group} $W_\Lambda:=\Lambda\rtimes W$ is no longer a Coxeter group and the above combinatorics does not hold anymore. This comes from the non-trivial symmetries of the fundamental alcove in the group. However, it is enough to consider the barycentric subdivision of the fundamental alcove (see Theorem \ref{complexforY}). Though heavy in computations, this has the advantage of giving a general statement for all cases at once. Moreover, this construction applied to the simply-connected case gives the same complex as the first one, up to $\Z[W]$-homotopy equivalence.

It is a natural question to ask if there still is any geometric information behind the above complex in the non-crystallographic Coxeter groups. Our main result is the following one:

\begin{theo*}[\ref{defsteinbergtorus}, \ref{T(I2(m))hyperbolicsurface} and \ref{definableoverQbar}]
Let $(W,S)$ be a finite irreducible Coxeter system of rank $n$. For any fixed reflection $r\in W$, we may consider the extended Coxeter system $(\widehat{W},\widehat{S})$, where $\widehat{S}:=S\cup\{\widehat{s}\}$ and $m_{\widehat{s},s}=m_{s,\widehat{s}}$ is the order of $rs\in W$ for $s\in S$. 

\begin{enumerate}
\item If $W$ is crystallographic, fix a semisimple simply-connected root datum $\mathcal{R}$ with Weyl group $W$, consider the reflection $r_W:=r_{\widetilde{\alpha}}$ with $\widetilde{\alpha}$ the highest (long) root of $\mathcal{R}$ and let $G$ be the (simply-connected) compact Lie group with root datum $\mathcal{R}$. The corresponding extension $\widehat{W}$ is affine.
\item Otherwise, choose a reflection $r_W\in W$ such that the extension $\widehat{W}$ is compact hyperbolic (see \ref{cpcthypextsofnoncryst}).
\end{enumerate}

In either case, we let $\pi : \widehat{W}\twoheadrightarrow W$ be the projection map, sending $s\in S$ to itself, and the additional simple reflection $\widehat{s}$ to $r_W$. Let also $\widehat{\Sigma}$ be the Coxeter complex of $\widehat{W}$ and $Q:=\ker\pi$. Then $\mathbf{T}(W):=\widehat{\Sigma}/Q$ is a connected, orientable, compact, $W$-triangulated Riemannian $W$-manifold of dimension $n$. Moreover, in the first case, $\mathbf{T}(W)$ is $W$-isometric to a maximal torus of $G$ and in the second case, the manifold $\mathbf{T}(W)$ is hyperbolic.

In the dihedral case, the surfaces $\mathbf{T}(I_2(2g+1))$, $\mathbf{T}(I_2(4g)$ and $\mathbf{T}(I_2(4g+2))$ are naturally Riemann surfaces of genus $g$, definable over $\overline{\Q}$ and rational elliptic curves if $g=1$.
\end{theo*}

\begin{rem*}\label{unicity}
The extension $\widehat{W}$ thus comes with a torsion-free normal subgroup $Q\unlhd\widehat{W}$ such that $\widehat{W}=Q\rtimes W$. We have $Q=\Z\Phi^\vee$ if $W$ is Weyl and a non-commutative analogue otherwise. A key fact is that the action of $Q$ (under the dual geometric representation of $W$) on the \emph{Coxeter complex} $\Sigma(\widehat{W})$ (seen as a quotient of the Tits cone) is a covering space action and we naturally define $\mathbf{T}(W):=\Sigma(\widehat{W})/Q$.

Regarding the unicity of the possible extensions and suitable associated reflections, we can say the following. First, if $W$ is crystallographic and simply-laced, then there is a unique reflection for which the corresponding extension is affine. In the other crystallographic cases, there is a reflection yielding an affine extension for each choice of root lengths in a root system for $W$, i.e. for each choice of a simply-connected root datum of the same type as $W$. We get two different extensions for $W=B_n=C_n$ with a unique suitable reflection for each and a unique extension for $F_4$ and $G_2$, each one of which being realized by two reflections.

In types $H_3$ and $H_4$, there is a unique reflection giving a compact hyperbolic extension. For the dihedral groups, the situation is more delicate; see the Remark \ref{unicity_dihedrals} for more details.
\end{rem*}

It should be mentioned that the manifolds $\mathbf{T}(H_3)$ and $\mathbf{T}(H_4)$ were already constructed respectively by Zimmermann (\cite{zimmermann_hyperbolic}) and by Davis (\cite{davis_hyperbolic}), using a different method.

It is clear from the construction that $\mathbf{T}(W)$ is equipped with a $W$-triangulation (which yields a \emph{dessin d'enfant} when $W=I_2(m)$), whose associated $W$-dg-ring has the same combinatorics as in the Weyl group case (see Theorem \ref{cochaincomplexforT(W)}). We use the \emph{Hopf trace formula} (see lemma \ref{hopfforW}) to describe the homology of $\mathbf{T}(W)$ as a representation of $W$.

\begin{theo*}[\ref{H0andHn}, \ref{PD} and \ref{euler+poincare=love}]
If $W$ is a finite Coxeter group of rank $n$ and $\mathbf{T}(W)$ is the $W$-manifold from the previous theorem, then we have
\[H_0(\mathbf{T}(W),\Z)=\mathds{1}~~\text{and}~~H_n(\mathbf{T}(W),\Z)=\varepsilon,\]
where $\varepsilon$ is the signature representation of $W$. 

The homology $H_*(\mathbf{T}(W),\Z)$ is torsion-free and the Betti numbers of $\mathbf{T}(W)$ are palindromic, meaning that $b_i=b_{n-i}$ for all $i$.

Moreover, the geometric representation of $W$ is a direct summand of $H_1(\mathbf{T}(W),\Bbbk)$, where $\Bbbk$ is a splitting field for $W$ and $H_1(\mathbf{T}(W),\Bbbk)$ is irreducible if and only if $W$ is crystallographic.

Finally, if $W(q)$ (resp. $\widehat{W}(q)$) is the Poincar\'e series of $W$ (resp. of $\widehat{W}$), then
\[\chi(\mathbf{T}(W))=\left.\frac{W(q)}{\widehat{W}(q)}\right|_{q=1}.\]
\end{theo*}

Finally, a GAP4 package\footnote{\url{https://github.com/arthur-garnier/Salvetti-and-tori-complexes}} is provided to compute these complexes, along with the DeConcini-Salvetti complex of a finite Coxeter group (see \cite{salvetti}).

\part{Weyl-equivariant triangulations of tori of compact Lie groups and related $W$-dg-rings}

\section{Prerequisites and notation}

Let $K$ be a simple compact Lie group, $T$ a maximal torus of $K$, we denote by $\mathfrak{k}$ and $\mathfrak{t}$ their respective Lie algebras. The complexification of $\mathfrak{k}$ is the reductive complex Lie algebra $\mathfrak{g}:=\mathfrak{k}+i\mathfrak{k}$ and $\mathfrak{h}:=\mathfrak{t}+i\mathfrak{t}$ is a Cartan subalgebra of $\mathfrak{g}$. Let $\Phi\subset\mathfrak{h}^*$ be the root system of $(\mathfrak{g},\mathfrak{h})$, with simple system $\Pi$, root lattice $Q=\Z\Phi$ and weight lattice $P\supset Q$. Since $\mathfrak{t}=\mathrm{span}_\R(i\alpha^\vee)_{\alpha\in\Pi}$, we have $\Phi\subset i\mathfrak{t}^*=:V$ (see \cite[\S 3.2]{kirillovandjunior} or \cite[\S 103]{zelobenko}) and we consider the \emph{character lattice} of $T$ given by
\[X(T)=\{d\lambda : \mathfrak{t}\to i\R~;~\lambda\in\ho(T,\Sph^1)\}\subset i\mathfrak{t}^*=V\]
and its \emph{cocharacter lattice} is $Y(T):=X(T)^\wedge\subset V^*$, so that $(X(T),\Phi,Y(T),\Phi^\vee)$ is a root datum. Since $T$ is abelian, the elements of the Weyl group $W=W(\Phi)\simeq N_K(T)/T$ act on $T$ by conjugation by a representative element in $N_K(T)$. By \cite[Lemma 1]{kirillovandjunior}, the normalized exponential map defines a $W$-isomorphism of Lie groups
\begin{equation}\label{T=V/Y}\tag{$\dagger$}
V^*/Y(T) \stackrel{\tiny{\sim}}\longrightarrow T.
\end{equation}
Moreover, we have the following isomorphisms
\[P/X(T)\simeq\pi_1(K)~~\text{and}~~X(T)/Q\simeq Z(K).\]

\begin{notation*}
Throughout the first part of the paper we fix an irreducible root datum $(X,\Phi,Y,\Phi^\vee)$ and rank $n$, with ambient space $V=\Z\Phi\otimes \R$, simple roots $\Pi\subset\Phi^+$, Weyl group $W=\left<s_\alpha,~\alpha\in\Pi\right>$, fundamental (co)weights $(\varpi_\alpha)_{\alpha\in\Pi}$ and $(\varpi^\vee_\alpha)_{\alpha\in\Pi}$, (co)root lattices $Q$ and $Q^\vee$ and (co)weight lattices $P$ and $P^\vee$. We index the set $\Pi$ of simple roots by $\{\alpha_1,\dotsc,\alpha_n\}$ and the sets of fundamental (co)weights accordingly. Let also $\alpha_0=\sum_{i=1}^nn_i\alpha_i\in\Phi^+$ be the highest (long) root of $\Phi$.
\end{notation*}

Consider the affine transformation
\[s_0:=\mathrm{t}_{\alpha_0^\vee}s_{\alpha_0} : \lambda \longmapsto s_{\alpha_0}(\lambda)+\alpha_0^\vee=\lambda-(\left<\lambda,\alpha_0\right>-1)\alpha_0^\vee,\]
where $\mathrm{t}_{\alpha_0^\vee} : V^*\to V^*$ is the translation by $\alpha_0^\vee$. The group $W_\mathrm{a}:=\left<s_0,s_1,\dotsc,s_n\right>\le\mathrm{Aff}(V^*)$ is then a Coxeter group, called the \emph{affine Weyl group}. It splits as $W_\mathrm{a}=Q^\vee\rtimes W$. For $\alpha\in\Phi$ and $k\in\Z$, we consider the affine hyperplanes $H_{\alpha,k}:=\{\lambda\in V^*~;~\left<\lambda,\alpha\right>=k\}$ and we call \emph{alcove} any connected component of $V^*\setminus\bigcup_{\alpha,k}H_{\alpha,k}$. The \emph{fundamental alcove} is
\[\mathcal{A}_0:=\{\lambda\in V^*~;~\forall \alpha\in\Phi^+,~0<\left<\lambda,\alpha\right><1\}=\{\lambda\in V^*~;~\forall 1\le i \le n,~\left<\lambda,\alpha_i\right>>0,~\left<\lambda,\alpha_0\right><1\}.\]
Then, by \cite[V, \S 2.2, Corollary]{bourbaki456}, its closure is a standard simplex
\[\overline{\mathcal{A}_0}=\mathrm{conv}\left(\{0\}\cup\left\{{\varpi_i^\vee}/{n_i}\right\}_{1\le i \le n}\right)\simeq\Delta^n\]
and by \cite[\S 4.5 and 4.8]{humphreys-reflectiongroups}, $\overline{\mathcal{A}_0}$ is a fundamental domain for $W_\mathrm{a}$ in $V^*$ and moreover, $W_\mathrm{a}$ acts simply transitively on the set of open alcoves.

Before going any further into our study, we give some reminders and notation on equivariant CW-complexes. A detailed treatment can be found in \cite[II, \S 1]{tomdieck_transgroups}.

Recall that, for a discrete group $G$, a $G$-space $X$ is said to be a $G$\emph{-CW-complex} if it has a CW-complex structure such that $G$ acts on the $k$-cells of $X$ for all $k$ and, for any cell $e$ of $X$ and any $g\in G$, if $ge=e$ then $gx=x$ for every $x\in e$.

Given a CW-complex $X$, we can consider its \emph{cellular homology chain complex} $C^{\text{cell}}_*(X,\Z)$, where $C^{\text{cell}}_n(X,Z)=\bigoplus_{i\in I}\Z e_i$ with $e_i,~i\in I,$ the $n$-cells of $X$. If $X$ is a $G$-CW-complex, then $C^{\text{cell}}_*(X,\Z)$ is a chain complex of $\Z[G]$-modules, which we denote by $C^\mathrm{cell}_*(X,G;\Z)$, or simply $C^\mathrm{cell}_*(X,\Z)$ if the group $G$ is clear by the context. Moreover, if $\mathcal{E}_n$ is the (possibly infinite) set of $n$-cells of $X$, then $G$ acts on $\mathcal{E}_n$ and the $\Z$-module $C_n^\mathrm{cell}(X,\Z)$ is free with basis $\mathcal{E}_n$, so that $C_n^\mathrm{cell}(X,G;\Z)$ is a \emph{permutation module} and decomposing $\mathcal{E}_n=\bigsqcup_i G/H_i$ into orbits yields
\[C_n^\mathrm{cell}(X,G;\Z)\simeq\bigoplus_i \Z[G/H_i],\]
where $H_i$ runs through a representative set of stabilizers of $n$-cells of $X$.

We may describe the dual complex $C^*_{\rm{cell}}(X,G;\Z)$ in a similar way. For an arbitrary set $S$, we denote by $\Z[[S]]$ the set of families $x=(x_s)_{s\in S}$ of integers, indexed by $S$; it is just the $\Z$-module $\Z^S$. It will be convenient to prefer the formal notation $x=\sum_{s\in S}x_ss$. Notice that, for an arbitrary group $G$ and $H\le G$, we have a canonical isomorphism of right $\Z[G]$-modules
\[\begin{array}{ccc}
\Z[G/H]^\vee\stackrel{\tiny{\text{df}}}=\ho(\Z[G/H],\Z) & \longrightarrow & \Z[[H\backslash G]] \\ (gH)^* & \longmapsto & Hg^{-1}\end{array}\]
and this yields an isomorphism $\Z[G/H]^\vee\to\Z[H\backslash G]$ in case $H$ is of finite index. We get
\[C^n_{\rm{cell}}(X,G;\Z)=\prod_i\Z[[H_i\backslash G]]\]
where the $H_i$'s are as above.

If $G$ acts on a set $X$ and if $N\unlhd G$, then we may consider the \emph{deflation} $\mathrm{Def}^G_{G/N}(X):=X/N$ of $X$, with the induced action of $G/N$. On another hand, if $\pi : G\twoheadrightarrow G/N$ is the projection, then we have a canonical isomorphism of $G/N$-sets $\mathrm{Def}^G_{G/N}(G/H)\simeq \pi(G)/\pi(H)$, for every subgroup $H$ of $G$. This gives a functor $\mathrm{Def}^G_{G/N} : G\mhyphen\mathbf{Set}\to G/N\mhyphen\mathbf{Set}$ and linearizing it gives the usual linear deflation
\[\begin{array}{ccccc}
\mathrm{Def}_{G/N}^G & : & \Z[G]\mhyphen\mathbf{Mod} & \longrightarrow & \Z[G/N]\mhyphen\mathbf{Mod} \\ & & U & \longmapsto & U_N\end{array}\]
where $U_N:=U/\left<nu-u~|~n\in N,~u\in U\right>$, which we extend to (co)chain complex categories. We have the following easy result:

\begin{lem}\label{GCWandSDP}
Let $G$ be a discrete group, written as a semi-direct product $G=N\rtimes H$, $X$ be a $G$-CW-complex and let $p : X\twoheadrightarrow X/N$ be the natural projection. If the quotient space $X/N$ is Hausdorff, then it is an $H$-CW-complex such that $\mathcal{E}_k(X/N)=\{p(e),~e\in\mathcal{E}_k(X)\}$ for all $k\in\N$ and the projection $G\twoheadrightarrow H$ induces a natural isomorphism
\[C^\mathrm{cell}_*(X/N,H;\Z)\simeq \mathrm{Def}^G_H\left(C^\mathrm{cell}_*(X,G;\Z)\right).\]
\end{lem}

\section{The simply-connected case}

\subsection{The $W_\mathrm{a}$-triangulation of $V^*$ associated to the fundamental alcove}
\hfill

In order to obtain a $W_{\rm{a}}$-triangulation of $V^*$, it suffices to have a triangulation of $\overline{\mathcal{A}_0}$, which is compatible with the action of $W_{\mathrm{a}}$ in the sense that if a face is fixed globally by some $w\in W_{\mathrm{a}}$, then $w$ induces the identity on this face.

Given a polytope $\mathcal{P}\subset\R^n$ and an integer $-1\le k\le \dim(\mathcal{P})$, we denote by $F_k(\mathcal{P})$ the set of $k$-dimensional faces of $\mathcal{P}$, in particular, $F_{-1}(\mathcal{P})=\{\emptyset\}$ and $F_{\dim(\mathcal{P})}(\mathcal{P})=\{\mathcal{P}\}$. We let $F(\mathcal{P}):=\bigcup_kF_k(\mathcal{P})$ be the \emph{face lattice} of $\mathcal{P}$. It is indeed a lattice for the inclusion relation.

Resuming to root data, for each $i\in S:=\{1,\dotsc,n\}$ we consider the hyperplane
\[H_i:=H_{\alpha_i,0}=\{\lambda\in V^*~;~\left<\lambda,\alpha_i\right>=0\}~~\text{and}~~H_0:=H_{\alpha_0,1}=\{\lambda\in V^*~;~\left<\lambda,\alpha_0\right>=1\}.\]
We also take the following notation for the vertices of $\overline{\mathcal{A}_0}$, 
\[v_i:=\frac{\varpi_i^\vee}{n_i}~~\text{and}~~v_0:=0~~\text{so that}~~\vertices(\overline{\mathcal{A}_0})=\{v_0,v_1,\dotsc,v_n\}.\]
The hyperplanes $H_i$ for $i\in S_0:=S\cup\{0\}$ form a complete set of bounding hyperplanes for the $n$-simplex $\overline{\mathcal{A}_0}$ and for every face $f\in F_k(\overline{\mathcal{A}_0})$ there is a subset $I\subseteq S_0$ of cardinality $|I|=\codim_{\overline{\mathcal{A}_0}}(f)=n-k$ such that 
\[f=f_I:=\overline{\mathcal{A}_0}\cap\bigcap_{i\in I}H_i\]
and we readily have $\vertices(f_I)=\{v_i~;~i\in S_0\setminus I\}$.

For $I\subseteq S_0$, we may consider the \emph{(standard) parabolic subgroup} $(W_\mathrm{a})_I$ of $W_\mathrm{a}$ generated by the subset $\{s_i,~i\in I\}$. If $0\notin I$, then $(W_\mathrm{a})_I$ is viewed as a parabolic subgroup of $W$.

\begin{lem}\label{stabfaceinA0}
Let $0\le k \le n$ and $I\subseteq S_0$ with $|I|=n-k$. Then the stabilizer $(W_\mathrm{a})_{f_I}$ of the face $f_I\in F_k(\overline{\mathcal{A}_0})$ is the parabolic subgroup of $W_\mathrm{a}$ associated to $I$. In other words,
\[(W_\mathrm{a})_{f_I}=(W_\mathrm{a})_I.\]
\end{lem}

\begin{proof}
As $\vertices(f_I)=\{v_i,~i\notin I\}$ is $(W_\mathrm{a})_{f_I}$-stable, the Theorem from \cite[\S 4.8]{humphreys-reflectiongroups} yields
\[(W_\mathrm{a})_{f_I}=\bigcap_{i\in S_0\setminus I}(W_\mathrm{a})_{v_i}=\bigcap_{i\in S_0\setminus I}(W_\mathrm{a})_{S_0\setminus\{i\}}=(W_\mathrm{a})_{\bigcap_{i\notin I}S_0\setminus\{i\}}=(W_\mathrm{a})_I.\]
\end{proof}

Therefore, we have a triangulation
\[V^*=\coprod_{\substack{f\in F(\overline{\mathcal{A}_0}) \\ \widetilde{w}\in W_\mathrm{a}/(W_{\mathrm{a}})_f}}\widetilde{w}\cdot f\]
which is $W_\mathrm{a}$-equivariant and following the notation from the first section, we have $\mathcal{E}_k(V^*)/W_\mathrm{a}=F_k(\overline{\mathcal{A}_0})$ for all $k$. Therefore, we get isomorphisms of $\Z[W_\mathrm{a}]$-modules
\[C^\mathrm{cell}_k(V^*,W_\mathrm{a};\Z)\simeq\bigoplus_{f\in F_k(\overline{\mathcal{A}_0})}\Z[W_\mathrm{a}/(W_\mathrm{a})_f]=\bigoplus_{\substack{I\subset S_0 \\ |I|=n-k}}\Z[W_\mathrm{a}/(W_\mathrm{a})_I].\]

We have to fix an orientation of the cells in $V^*$ and determine their boundary. But each one of them is a simplex, so its orientation is determined by an orientation on its vertices. We choose to orient them as the index set $(S_0,\le)$. For $I\subseteq S_0$ with corresponding $k$-face $f_I=\conv(\{v_i~;~i\in S_0\setminus I\})$, we write
\[f_I=[v_{j_1},\dotsc,v_{j_{k+1}}]~~\text{with}~~\{j_1<j_2<\dotsc<j_{k+1}\}=S_0\setminus I\]
to make its orientation explicit. The oriented boundary of $f_I$ is then simply given by
\[\partial_k(f_I)=\sum_{u=1}^{k+1}(-1)^u\underbrace{[v_{j_1},\dotsc,\widehat{v_{j_u}},\dotsc,v_{j_{k+1}}]}_{=\conv(\{v_j~;~j_u\ne j\in S_0\setminus I\})}=\sum_{u=1}^{k+1}(-1)^uf_{I\cup\{j_u\}}\]

We have thus obtained the following result:

\begin{theo}\label{complexforWa}
The face lattice of the $n$-simplex $\overline{\mathcal{A}_0}$ induces a $W_\mathrm{a}$-equivariant triangulation of $V^*$, whose cellular complex $C^\mathrm{cell}_*(V^*,W_\mathrm{a};\Z)$ is given (in degrees $k$ and $k-1$) by
\[\xymatrix{\cdots \ar[r] & \displaystyle{\bigoplus_{\substack{I\subset S_0 \\ |I|=n-k}}\Z\left[W_\mathrm{a}^I\right]} \ar^<<<<<{\partial_k}[r] & \displaystyle{\bigoplus_{\substack{I\subset S_0 \\ |I|=n-k+1}}\Z\left[W_\mathrm{a}^I\right]} \ar[r] & \cdots}\]
where $W_\mathrm{a}^I\approx W_\mathrm{a}/(W_\mathrm{a})_I$ is the $W_\mathrm{a}$-set of minimal length left coset representatives and the boundaries are defined as follows: for $k\in\N$ and $I\subset S_0$, letting $\{j_1<\cdots <j_{k+1}\}:=S_0\setminus I$,
\[{(\partial_k)}_{\mid\Z\left[W_\mathrm{a}^I\right]}=\sum_{u=1}^{k+1}(-1)^u p^I_{I\cup\{j_u\}},\]
where, for $I\subset J$, $p^{I}_J$ denotes the projection
\[p^{I}_J : W_\mathrm{a}^{I}=W_\mathrm{a}/(W_\mathrm{a})_{I} \longtwoheadrightarrow W_\mathrm{a}/(W_\mathrm{a})_J=W_\mathrm{a}^J.\]
\end{theo}

\begin{exemple}\label{ex_with_A2}
We look at the case of the group $SU(3)$ in type $A_2$. We denote by $\Phi=\{\pm\alpha,\pm\beta,\pm(\alpha+\beta)\}$ a root system of type $A_2$, with simple system $\Pi=\{\alpha,\beta\}$. The chain complex $C^{\rm{cell}}_*(V^*,W_{\mathrm{a}};\Z)$ is readily given by
\[\footnotesize{\xymatrix{\Z[W_{\mathrm{a}}] \ar^>>>>>{\partial_2}[r] & \Z[W_{\mathrm{a}}/\left<s_\beta\right>]\oplus\Z[W_{\mathrm{a}}/\left<s_0\right>]\oplus\Z[W_{\mathrm{a}}/\left<s_\alpha\right>]\ar^>>>>>{\partial_1}[r] & \Z[W_{\mathrm{a}}/\left<s_\alpha,s_\beta\right>]\oplus\Z[W_{\mathrm{a}}/\left<s_\beta,s_0\right>]\oplus\Z[W_{\mathrm{a}}/\left<s_\alpha,s_0\right>]}},\]
where the boundaries are
\[\partial_2=\left(\begin{smallmatrix}1 & 1 & -1\end{smallmatrix}\right),~\partial_1=\left(\begin{smallmatrix}-1 & 1 & 0 \\ 0 & -1 & 1 \\ -1 & 0 & 1\end{smallmatrix}\right).\]
Applying the functor $\mathrm{Def}_W^{W_{\mathrm{a}}}$, we obtain the complex $C^\mathrm{cell}_*(T,W;\Z)$ for $T=S(U(1)^3)$ (see Figure \ref{unsubdivided_alcoveA2}) as
\[\xymatrix{\Z[W] \ar^>>>>>{\partial_2}[r] & \Z[W/\left<s_\beta\right>]\oplus\Z[W/\left<s_\alpha s_\beta s_\alpha\right>]\oplus\Z[W/\left<s_\alpha\right>]\ar^<<<<<{\partial_1}[r] & \Z^3}.\]

\begin{center}
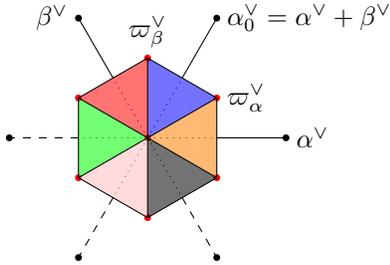
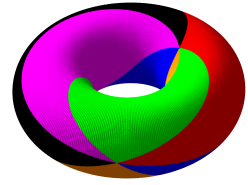
\begin{figure}[h!]
\begin{subfigure}[c]{0.3\textwidth}
\begin{tikzpicture}[scale=1.3,rotate=45]  
  \coordinate (z) at (0,0);
  \coordinate (a) at (1,-1);
  \coordinate (b) at (0.3660254040,1.366025404);
  \coordinate (ab) at (1+0.3660254040,-1+1.366025404);
  \coordinate (ma) at (-1,1);
  \coordinate (mb) at (-0.3660254040,-1.366025404);
  \coordinate (mab) at (-1-0.3660254040,1-1.366025404);
  
  \coordinate (la) at (2/3+1/3*0.3660254040,-2/3+1/3*1.366025404);
  \coordinate (lb) at (1/3+2/3*0.3660254040,-1/3+2/3*1.366025404);
  
  \coordinate (lad) at (1/3+1/6*0.3660254040,-1/3+1/6*1.3660254040);
  \coordinate (lbd) at (1/6+1/3*0.3660254040,-1/6+1/3*1.3660254040);
  \coordinate (barlab) at (1/3+1/3*0.3660254040,-1/3+1/3*1.3660254040);
  \coordinate (labd) at ($(la)!0.5!(lb)$);
  
  \coordinate (lamx) at (-2/3+1/3*1.3660254040,2/3+1/3*0.366025404);
  \coordinate (lamy) at (2/3-1/3*1.3660254040,-2/3-1/3*0.366025404);
  \coordinate (mla) at (-2/3-1/3*0.3660254040,2/3-1/3*1.366025404);
  \coordinate (mlb) at (1/3-2/3*1.3660254040,-1/3-2/3*0.366025404);
  \coordinate (lblamxd) at ($(lb)!0.5!(lamx)$);
  \coordinate (lamxmlad) at ($(lamx)!0.5!(mla)$);
  \coordinate (mlamlbd) at ($(mla)!0.5!(mlb)$);
  \coordinate (mlblamyd) at ($(mlb)!0.5!(lamy)$);
  \coordinate (lamylad) at ($(lamy)!0.5!(la)$);
  
  \fill[fill=black] (a) circle (1pt);
  \fill[fill=black] (b) circle (1pt);
  \fill[fill=black] (ab) circle (1pt);
  \fill[fill=black] (ma) circle (1pt);
  \fill[fill=black] (mb) circle (1pt);
  \fill[fill=black] (mab) circle (1pt);
  
  \fill[fill=red] (la) circle (1pt);
  \fill[fill=red] (lb) circle (1pt);
  \fill[fill=red] (z) circle (1pt);
  \fill[fill=red] (lamx) circle (1pt);
  \fill[fill=red] (lamy) circle (1pt);
  \fill[fill=red] (mla) circle (1pt);
  \fill[fill=red] (mlb) circle (1pt);
  
  \draw[dotted] (z)--(labd) (z)--(lblamxd) (z)--(lamxmlad) (z)--(mlamlbd) (z)--(mlblamyd) (z)--(lamylad);
  \draw (labd)--(ab) (lblamxd)--(b) (lamylad)--(a);
  \draw[dashed] (lamxmlad)--(ma) (mlblamyd)--(mb) (mlamlbd)--(mab);
  
  \draw (a) node[right]{\small{$\alpha^\vee$}};
  \draw (b) node[left]{\small{$\beta^\vee$}};
  \draw (ab) node[right]{\small{$\alpha_0^\vee=\alpha^\vee+\beta^\vee$}};
  \draw (la) node[right]{\small{$\varpi_\alpha^\vee$}};
  \draw (lb) node[above]{\small{$\varpi_\beta^\vee$}};
  
  \fill[fill=blue,opacity=0.55] (z)--(la)--(lb);
  \fill[fill=red,opacity=0.55] (z)--(lb)--(lamx);
  \fill[fill=green,opacity=0.55] (z)--(lamx)--(mla);
  \fill[fill=pink,opacity=0.55] (z)--(mla)--(mlb);
  \fill[fill=black,opacity=0.55] (z)--(mlb)--(lamy);
  \fill[fill=orange,opacity=0.55] (z)--(lamy)--(la);
  
  \draw (z)--(la) (z)--(lb) (z)--(lamx) (z)--(lamy) (z)--(mla) (z)--(mlb);
  \draw (lb)--(la) (la)--(lamy) (lamy)--(mlb) (mlb)--(mla) (mla)--(lamx) (lamx)--(lb);
  
  \coordinate (x) at (2+0.3660254040,-2+1.3660254040);
  \coordinate (y) at (1+2*0.3660254040,-1+2*1.3660254040);

\end{tikzpicture}
\caption{The fundamental alcove $\overline{\mathcal{A}_0}$ (in blue) in type $A_2$, and its $\Sym_3$-translates.}
\end{subfigure}
\hfill
\begin{subfigure}[c]{0.3\textwidth}
\includegraphics[scale=0.08]{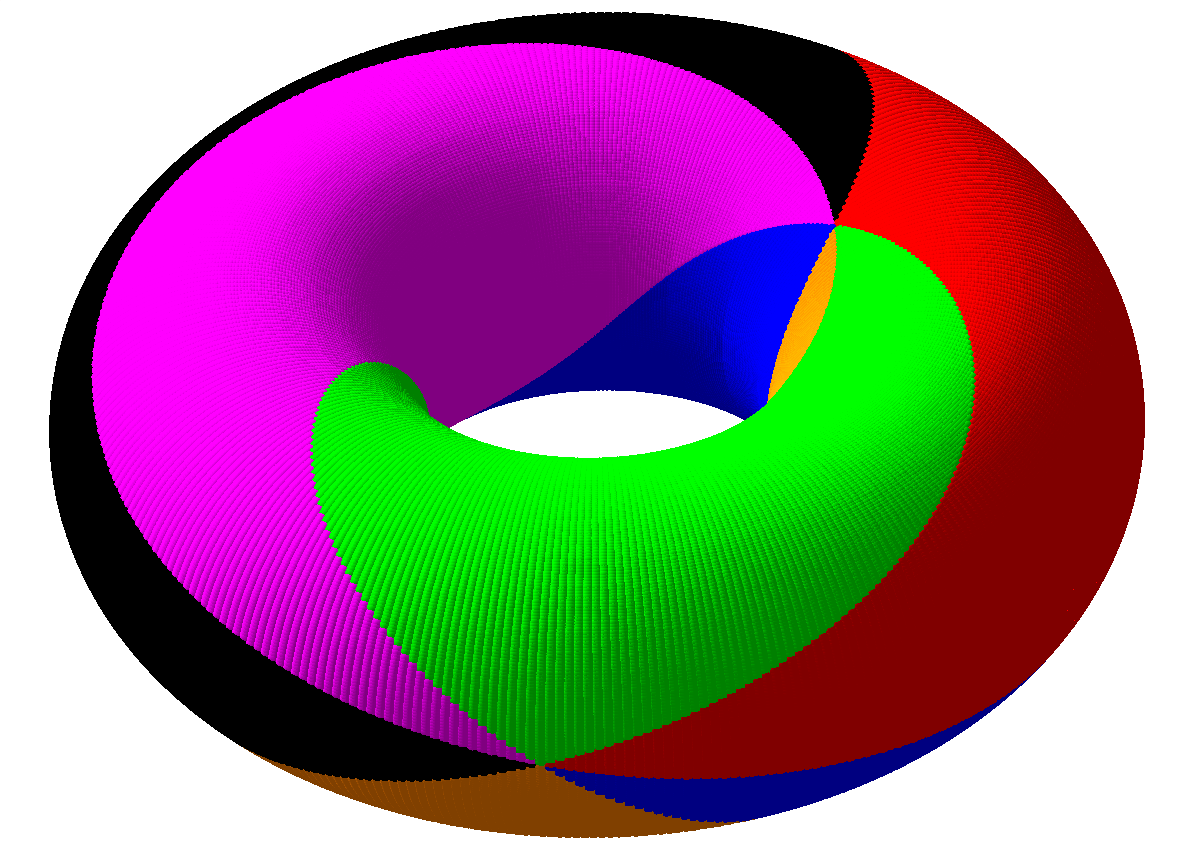}
\caption{The resulting $\Sym_3$-equivariant triangulation of $S(U(1)^3)\simeq (\Sph^1)^2$.}
\end{subfigure}
\caption{Triangulation of the torus $S(U(1)^3)$ of $SU(3)$.}
\label{unsubdivided_alcoveA2}
\end{figure}
\end{center}
\end{exemple}

\subsection{The $W$-dg-ring structure}
\hfill

We now make the cup product on $C^\mathrm{cell}_*(V^*,W_\mathrm{a};\Z)$ explicit. We write
\[C_\mathrm{cell}^k(V^*,W_\mathrm{a};\Z)=\prod_{\substack{I\subset S_0 \\ |I|=n-k}}\Z\left[W_\mathrm{a}/(W_\mathrm{a})_I\right]^\vee\simeq\bigoplus_{\substack{I\subset S_0 \\ |I|=n-k}}\Z\left[\left[{}^I{W_\mathrm{a}}\right]\right],\]
where 
\[{}^I{W_\mathrm{a}}\stackrel{\tiny{\text{df}}}=\{w\in W_\mathrm{a}~;~\ell(s_iw)>\ell(w),~\forall i\in I\}\approx (W_\mathrm{a})_I\backslash W_\mathrm{a}.\]

\begin{lem}[{\cite[\S 3, Proposition 2.7 and Corollary 2.8]{billey2016}}]\label{DCD}
Let $(W,S)$ be a Coxeter system and $I,J\subset S$. Denote as usual
\[W^I:=\{w\in W~;~\ell(ws)>\ell(w),~\forall s\in I\}\approx W/W_I,\]
\[{}^IW:=\{w\in W~;~\ell(sw)>\ell(w),~\forall s\in I\}\approx W_I\backslash W\]
and for $J\subset I$,
\[W^J_I:=\{w\in W_I~;~\ell(ws)>\ell(w),~\forall s\in J\}\approx W_I/W_J.\]
\begin{enumerate}
\item Each double coset in $W_I\backslash W/W_J$ has a unique element of minimal length.
\item An element $w\in W$ is of minimal length in its double coset if and only if $w\in {}^IW\cap W^J$. In particular, we have a bijection 
\[W_I\backslash W/W_J\approx {}^IW\cap W^J.\]
\item As a consequence, if $w\in {}^IW\cap W^J$ and $x\in W_I$, then $xw\in W^J$ if and only if $x\in W^{I\cap {}^{w\!}{J}}_I$, where we let ${}^{w\!}{J}:=wJw^{-1}$. Hence, we have the following property:
\[\forall x\in W_IwW_J,~\exists!(u,v)\in W^{I\cap{}^{w\!}{J}}_I\times W_J~;~\left\{\begin{array}{ll}x=uwv, \\ 
\ell(x)=\ell(u)+\ell(w)+\ell(v). \end{array}\right.\]
\end{enumerate}
\end{lem}

We can now formulate the main result:

\begin{theo}\label{cupprodparabolics}
The $\Z[W_\mathrm{a}]$-cochain complex $C_\mathrm{cell}^*(V^*,W_\mathrm{a};\Z)$ associated to the $W_\mathrm{a}$-triangulation of $V^*$ is a $W_\mathrm{a}$-dg-ring with homogeneous components given, for $0\le k \le n$, by
\[C^k_\mathrm{cell}(V^*,W_\mathrm{a};\Z)=\bigoplus_{\substack{I\subset S_0 \\ |I|=n-k}}\Z\left[\left[(W_\mathrm{a})_I\backslash W_\mathrm{a}\right]\right]\simeq\bigoplus_{\substack{I\subset S_0 \\ |I|=n-k}}\Z\left[\left[{}^I{W_\mathrm{a}}\right]\right]\]
and differentials defined, for any $I\subset S_0$ and $w\in W_{\mathrm{a}}$, by
\[d^k({}^I{w})=\sum_{\substack{0\le u \le k+1 \\ j_{u-1}<j<j_u}}(-1)^u\epsilon^{I\setminus\{j\}}_Iw\]
where $\{j_0<\cdots<j_k\}:=S_0\setminus I$ and, by convention, $j_{-1}=-1$, $j_{k+1}=n+1$ and for $J\subset I$,
\[{}^J_I{W_\mathrm{a}}:=\{w\in (W_\mathrm{a})_I~;~\ell(s_jw)>\ell(w),~\forall j\in J\}~~\text{and}~~{\epsilon}^J_I:=\sum_{x\in{}^J_I{W_\mathrm{a}}}x\in\Z\left[{}^J{W_\mathrm{a}}\right].\]

Moreover, the cup product
\[C^p_{\mathrm{cell}}(V^*,W_\mathrm{a};\Z) \otimes C^q_{\mathrm{cell}}(V^*,W_\mathrm{a};\Z) \stackrel{\tiny{\cup}}\longrightarrow C^{p+q}_{\mathrm{cell}}(V^*,W_\mathrm{a};\Z)\]
is induced by the unique map
\[\Z\left[\left[{}^I{W_\mathrm{a}}\right]\right] \otimes \Z\left[\left[{}^J{W_\mathrm{a}}\right]\right] \longrightarrow \Z\left[\left[{}^{I\cap J}{W_\mathrm{a}}\right]\right].\]
satisfying the formula
\[{}^Ix\cup{}^Jy=\delta_{\max(I^\complement),\min(J^\complement)}\times\left\{\begin{array}{cc}
{}^{I\cap J}{((xy^{-1})_Jy)} & \text{if }xy^{-1}\in (W_\mathrm{a})_I(W_\mathrm{a})_J, \\ 0 & \text{otherwise},\end{array}\right.\]
where $\complement$ denotes the complementary in $S_0$ and, given $w\in W_\mathrm{a}$, we denote by ${}^{I\cap J}w\in{}^{I\cap J}{W_\mathrm{a}}$ its minimal length right coset representative and if $(W_\mathrm{a})_Iw(W_\mathrm{a})_J=(W_\mathrm{a})_I(W_\mathrm{a})_J$, we let $w_J$ be the unique element $v\in (W_\mathrm{a})_J$ such that $w=uv$, with $u\in (W_\mathrm{a})_I^{I\cap J}$ and $\ell(w)=\ell(u)+\ell(v)$.
\end{theo}

\begin{proof}
Take a $k$-simplex $\sigma=[j_0,\dotsc,j_k]\subset\overline{\mathcal{A}_0}$ with $j_u\in S_0$ and let $I:=S_0\setminus\{j_0,\dotsc,j_k\}$ so that $(W_{\rm{a}})_{\sigma^*}=(W_{\rm{a}})_I$ and set $j_{-1}:=-1$ and $j_{k+1}:=n+1$. By definition of $d^k$, we have
\[d^k((W_{\rm{a}})_I\cdot1)_{|\overline{\mathcal{A}_0}}=d^k(\sigma^*)_{|\overline{\mathcal{A}_0}}=\sum_{\substack{0\le u \le k+1 \\ j_{u-1}<j<j_u}}(-1)^u[j_0,\dotsc,j_{u-1},j,j_u,\dotsc,j_k]^*=\sum_{\substack{0\le u \le n+1-|I| \\ j_{u-1}<j<j_u}}(-1)^u((W_{\rm{a}})_{I\setminus\{j\}}\cdot 1)\]
Thus, as $\overline{\mathcal{A}_0}$ is a fundamental domain for $W_{\rm{a}}$ in $V^*$, this yields
\[d^k((W_{\rm{a}})_I\cdot 1)=\sum_{\substack{0\le u\le n+1-|I| \\ j_{u-1}<j<j_u}}\left(\sum_{w\in{}^{I\setminus\{j\}}_I{W_{\rm{a}}}}(-1)^u((W_{\rm{a}})_{I\setminus\{j\}}\cdot w)\right).\]

Next, we establish the cup product formula. Let $\sigma\in F_a(\overline{\mathcal{A}_0})$, $\tau\in F_b(\overline{\mathcal{A}_0})$ be two simplices written as $\sigma=[i_0,\dotsc,i_a]$ and $\tau=[j_0,\dotsc,j_b]$ and let $I:=S_0\setminus\{i_0,\dotsc,i_a\}$ as well as $J:=S_0\setminus\{j_0,\dotsc,j_b\}$. Let also $x,y\in W_{\rm{a}}$. As the action of $W_{\rm{a}}$ on $\overline{\mathcal{A}_0}$ is simplicial, we have 
\[(W_{\rm{a}})_Ix\cup(W_{\rm{a}})_Jy=(\sigma^*x)\cup(\tau^*y)=((\sigma^*xy^{-1})\cup\tau^*)y\]
so we may assume that $y=1$. If $\sigma^*x\cup\tau^*\ne0$, then $x^{-1}\sigma$ and $\tau$ are included in some simplex of $V^*$ and we may choose $w_\tau\in(W_{\rm{a}})_\tau=(W_{\rm{a}})_J$ such that $x^{-1}\sigma\in w_\tau(\overline{\mathcal{A}_0})$, so that $\sigma\subset\overline{\mathcal{A}_0}\cap xw_\tau(\overline{\mathcal{A}_0})$ and thus $w_\tau\sigma=x^{-1}\sigma$. This implies that $\sigma^*x\cup\tau^*=(\sigma^*\cup\tau^*)w_\tau^{-1}$ and
\[0\ne\sigma^*\cup\tau^*=\delta_{i_a,j_0}[i_0,\dotsc,i_a,j_1,\dotsc,j_b]^*\]
and hence $\max(I^\complement)=i_a=j_0=\min(J^\complement)$. In this case, we get $\sigma^*x\cup\tau^*=(W_{\rm{a}})_{I\cap J}\cdot w_\tau^{-1}$ and, because $\sigma^*x\cup\tau^*\ne0$, we also have $xw_\tau\in(W_{\rm{a}})_I$, so $x\in(W_{\rm{a}})_I(W_{\rm{a}})_J$. The double coset decomposition from Lemma \ref{DCD} applied to the trivial double coset $(W_{\rm{a}})_Ix(W_{\rm{a}})_J$ allows us to write uniquely $x$ as $x=ux_J$ with $u\in(W_{\rm{a}})_I^{I\cap J}$ and $x_J\in(W_{\rm{a}})_J$ such that $\ell(x)=\ell(u)+\ell(x_J)$. We have $(W_{\rm{a}})_J\ni x_Jw_\tau=u^{-1}xw_\tau\in(W_{\rm{a}})_I$ so $x_Jw_\tau\in(W_{\rm{a}})_I\cap(W_{\rm{a}})_J=(W_{\rm{a}})_{I\cap J}$ and hence
\[\sigma^*x\cup\tau^*=(W_{\rm{a}})_{I\cap J}\cdot w_\tau^{-1}=(W_{\rm{a}})_{I\cap J}\cdot x_J.\]

It remains to prove that the stated formula indeed induces a well-defined map 
\[\Z\left[\left[{}^I{W_{\mathrm{a}}}\right]\right]\otimes\Z\left[\left[{}^J{W_{\mathrm{a}}}\right]\right]\to\Z\left[\left[{}^{I\cap J}{W_{\mathrm{a}}}\right]\right].\]
To see this, fix $x,y\in W_{\rm{a}}$. If $x',y'\in W_{\rm{a}}$ are such that ${}^Ix\cup{}^Jy={}^Ix'\cup{}^Jy'$, then $(xy^{-1})_Jy=(x'y^{'-1})_Jy'$ are in the same coset modulo $(W_{\rm{a}})_{I\cap J}$ and in particular in the same coset modulo $(W_{\rm{a}})_J$ and therefore ${}^Jy={}^Jy'$. As $(W_{\rm{a}})_J$ is finite, there are only finitely many such elements $y'$ and the same goes for $x'\in(W_{\rm{a}})_I(W_{\rm{a}})_Jy'$. This proves that for any $z\in{}^{I\cap J}W_{\rm{a}}$, there are only finitely many pairs $({}^Ix,{}^Jy)$ such that $z={}^Ix\cup{}^Jy$. Thus, if $a=\sum_{x\in{}^IW_{\rm{a}}}a_xx\in \Z\left[\left[{}^IW_{\rm{a}}\right]\right]$ and $b=\sum_{y\in{}^JW_{\rm{a}}}b_yy\in\Z\left[\left[{}^JW_{\rm{a}}\right]\right]$ with $a_x,b_y\in\Z$, then we can define
\[a\cup b:=\sum_{z\in{}^{I\cap J}W_{\rm{a}}}\left(\sum_{\substack{(x,y)\in{}^IW_{\rm{a}}\times{}^JW_{\rm{a}} \\ x\cup y=z}}a_xb_y\right)z\in\Z\left[\left[{}^{I\cap J}W_{\rm{a}}\right]\right]\]
and this is clearly the only way to define a map satisfying the stated formula.
\end{proof}

\begin{cor}\label{cupproductforT}
The $\Z[W]$-cochain complex $C^*_{\mathrm{cell}}(T,W;\Z)$ associated to the $W$-triangulation of $T=V^*/Q^\vee$ induced by the $W_\mathrm{a}$-triangulation of $V^*$ is given by
\[C^*_\mathrm{cell}(T,W;\Z)=\mathrm{Def}^{W_\mathrm{a}}_W\left(C^*_\mathrm{cell}(V^*,W_\mathrm{a};\Z)\right).\]
In other words, if $\pi : W_\mathrm{a} \twoheadrightarrow W$ is the projection, then
\[C^k_\mathrm{cell}(T,W;\Z)=\bigoplus_{\substack{I\subset S_0 \\ |I|=n-k}}\Z\left[\pi\left({}^I{W_\mathrm{a}}\right)\right]\simeq\bigoplus_{\substack{I\subset S_0 \\|I|=n-k}}\Z\left[\pi((W_\mathrm{a})_I)\backslash W\right],\]
with differentials given, for any $I\subset S_0$ and $w\in W_{\mathrm{a}}$, by
\[d^k(\pi({}^I{w}))=\sum_{\substack{0\le u \le k+1 \\ j_{u-1}<j<j_u}}(-1)^u\pi\left(\epsilon^{I\setminus\{j\}}_Iw\right),~\epsilon_I^J=\sum_{x\in{}^J_I{W_\mathrm{a}}}x\]
where $\{j_0<\cdots j_k\}:=S_0\setminus I$. The cup product in cohomology is induced by the formula
\[\pi({}^Ix)\cup\pi({}^Jy)=\delta_{\max(I^\complement),\min(J^\complement)}\times\left\{\begin{array}{cc}
\pi\left({}^{I\cap J}((xy^{-1})_Jy)\right) & \text{if }xy^{-1}\in(W_\mathrm{a})_I(W_\mathrm{a})_J \\ 0 & \text{otherwise}.\end{array}\right.\]
\end{cor}

\begin{rem}\label{lambda_homo}
There are $W$-equivariant isomorphisms of graded algebras
\[H^*(C^*_{\rm{cell}}(T,W;\Z))\simeq H^*(T;\Z)\simeq \Lambda^*(P).\]
The first isomorphism just results from the fact that $C^*_{\rm{cell}}(T,W;\Z)$ is the dg-ring associated to a $W$-equivariant simplicial structure on $T$. Next, the K\"{u}nneth formula provides a natural isomorphism of graded algebras 
\[\Lambda^*(H^1(T;\Z))\stackrel{\tiny{\sim}}\longto H^*(T;\Z).\]
More precisely, let $\mathbf{Tori}$ be the category of tori (the full subcategory of the category of Lie groups consisting of compact connected abelian Lie groups)  and $\mathbf{C}\subset\mathbf{Tori}$ be the full subcategory consisting of powers of the circle $\Sph^1$. Let $n\in\N^*$ and $S_n:=(\Sph^1)^n\in\mathbf{C}$.  If $\alpha\in H^1(\Sph^1)$ denotes the Poincar\'e dual of the fundamental class of a point in $\Sph^1$, then we have an isomorphism $H^*(\Sph^1)\simeq \Z[\alpha]/(\alpha^2)$. For $1\le i \le n$, let $\pi_i : S_n\twoheadrightarrow\Sph^1$ denote the canonical projection and $\alpha_i:=\pi_i^*(\alpha)\in H^1(S_n)$. Then $H^1(S_n)\simeq\bigoplus_i \Z\alpha_i\simeq\Z^n$ and letting $\iota_i : \Sph^1\hookrightarrow S_n$ denote the natural inclusion, we have
\[\Lambda^*(H^1(S_n))=\Z[\alpha_1,\dotsc,\alpha_n]/(\alpha_i\alpha_j+\alpha_j\alpha_i)\simeq\bigotimes_{i=1}^n\Z[\alpha_i]/(\alpha_i^2)\simeq\bigotimes_i H^*(\iota_i(\Sph^1)).\]
Now, post-composing this with the K\"{u}nneth isomorphism $\bigotimes_i H^*(\iota_i(\Sph^1))\stackrel{\tiny{\sim}}\to H^*(S_n)$ yields an isomorphism
\[\gamma_{S_n} : \Lambda^*(H^1(S_n))\stackrel{\tiny{\sim}}\longto H^*(S_n).\]
Since the cross product is a multi-functor, the collection $\gamma:=(\gamma_{S_n})_{n\in\N}$ yields a natural isomorphism 
\[\gamma : (\Lambda^*\circ H^1)_{|\mathbf{C}}\stackrel{\tiny{\sim}}\longto H^*_{|\mathbf{C}}\]
of functors from $\mathbf{C}$ to the category $\mathbf{gr\mhyphen Rng}$ of graded rings. Now, the inclusion $j : \mathbf{C}\subset\mathbf{Tori}$ is an equivalence. If $r : \mathbf{Tori}\to\mathbf{C}$ denotes its inverse and if $\tau : {\rm Id}_{\mathbf{Tori}}\stackrel{\tiny{\sim}}\longto j\circ r$, then for a torus $S$, we let
\[\eta_S:=\tau_S^*\circ \gamma_{r(S)}\circ\Lambda^*(\tau_S^1)^{-1}.\]
It is straightforward to check that the collection $\eta:=(\eta_S)_S$ defines a natural isomorphism
\[\eta : \Lambda^*\circ H^1\stackrel{\tiny{\sim}}\longto H^*\]
of functors $\mathbf{Tori}\longto\mathbf{gr\mhyphen Rng}$.

The first homology group $H_1(T;\Z)$ is naturally isomorphic to the abelianization $\pi_1(T)^{\rm{ab}}$ of the fundamental group $\pi_1(T)$, which in turn is naturally isomorphic to $Y(T)=Q^\vee$, invoking the homotopy long exact sequence associated to the covering map $V^*\twoheadrightarrow V^*/Y(T)\simeq T$. By the universal coefficient theorem, we finally obtain natural isomorphisms
\[H^1(T;\Z)\simeq\ho(H_1(T;\Z),\Z)\simeq\ho(Q^\vee,\Z)\simeq P.\]

Applying the preceding argument to the case where the Lie group $K$ is no longer simply-connected, yields a natural (in particular, $W$-equivariant) isomorphism of graded algebras
\[\Lambda^*(X(T))\stackrel{\tiny{\sim}}\longto H^*(T;\Z).\]
\end{rem}

\section{The general case}

\subsection{The fundamental group as symmetries of an alcove}
\hfill

The \emph{extended affine Weyl group} $\widehat{W_\mathrm{a}}:=P^\vee\rtimes W$ acts on alcoves (transitively since $W_\mathrm{a}\unlhd\widehat{W_\mathrm{a}}$ does) but not simply-transitively. We introduce the stabilizer
\[\Omega:=\{\widehat{w}\in\widehat{W_\mathrm{a}}~;~\widehat{w}(\mathcal{A}_0)=\mathcal{A}_0\}\simeq \widehat{W_\mathrm{a}}/W_\mathrm{a}\simeq P^\vee/Q^\vee\simeq P/Q,\]
a finite abelian group and we see that we have a decomposition $\widehat{W}_{\rm{a}}\simeq W_{\rm{a}}\rtimes\Omega$. The Table \ref{extendeddynkindiagrams} depicts the fundamental groups of irreducible root systems. Recall that a fundamental weight $\varpi_i$ is called \emph{minuscule} if $n_i=1$ and that minuscule weights form a set of representatives of the non-trivial classes in $P/Q$ (see \cite[VI, \S 2.3, Corollary]{bourbaki456}). Dually, we have the same notion and result for \emph{minuscule coweights}.

\begin{prop-def}[\emph{\cite[VI, \S 2.3, Proposition 6]{bourbaki456}}]\label{descriptionOmega}
Define $M:=\{i \in S~;~n_i=1\}$ and let $w_0\in W$ be the longest element. For $i\in M$, let $w_0^i$ be the longest element of the Weyl group of the subsystem of $\Phi$ generated by $\{\alpha_j\}_{j\ne i}\subset\Pi$ and let $w_i:=w^i_0w_0$. Then $\mathrm{t}_{\varpi_i^\vee}w_i\in\Omega$ and the following map is a bijection:
\[\begin{array}{ccc}
M & \longrightarrow & \Omega\setminus\{1\} \\ i & \longmapsto & \omega_i:=\mathrm{t}_{\varpi_i^\vee}w_i \end{array}\]
\end{prop-def}

\indent We now have to see what happens if the cocharacter lattice $Y=X(T)^\wedge$ is such that $Q^\vee \subsetneq Y \subsetneq P^\vee$. To simplify the notation, we identify a lattice $\Lambda\subset V^*$ with its translation group $\mathrm{t}(\Lambda)\subset \mathrm{Aff}(V^*)$ and we define the \emph{intermediate affine Weyl group} $W_\Lambda:=\Lambda\rtimes W$. There is a correspondence between $W$-lattices $Q^\vee\subseteq\Lambda\subseteq P^\vee$ and the subgroups of $\Omega$.

\begin{lem}\label{identlatticessubgroupsofOmega}
Recall that $\widehat{W_\mathrm{a}}\simeq W_\mathrm{a}\rtimes\Omega$ and denote by 
\[\pi : \widehat{W_\mathrm{a}}\longtwoheadrightarrow \Omega\]
the natural projection. We have a bijective correspondence
\[\begin{array}{ccc}
\left\{\Lambda~;~Q^\vee\subseteq\Lambda\subseteq P^\vee~\text{is a } W\text{-lattice}\right\} & \stackrel{\tiny{\mathrm{1-1}}}\longleftrightarrow & \left\{H\le \Omega\right\} \\
\Lambda & \longmapsto & \pi(W_\Lambda) \\ \pi^{-1}(H)\cap P^\vee & \longmapsfrom & H \end{array}\]
\end{lem}

\subsection{A $\widehat{W_\mathrm{a}}$-triangulation of $V^*$ from the barycentric subdivision of an alcove}
\hfill

In order to obtain a $W_Y$-triangulation of the torus $V^*/Y$, we just have to exhibit an $\Omega_Y$-triangulation of the alcove $\overline{\mathcal{A}_0}$. As the group $\Omega_Y$ acts by affine automorphisms of $\overline{\mathcal{A}_0}$, the construction follows from the next easy result about simplicial subdivisions. 

Recall that, given a polytope $\mathcal{P}$, its \emph{barycentric subdivision} is the simplicial complex $\mathrm{Sd}(\mathcal{P})$ whose $k$-simplices are increasing chains of non-empty faces of $\mathcal{P}$ of length $k+1$. A $k$-simplex $(f_0,f_1,\dotsc,f_k)$ of $\mathrm{Sd}(\mathcal{P})$ may be geometrically realized as $\mathrm{conv}(\mathrm{bar}(f_0),\dotsc,\mathrm{bar}(f_k))$, where $\mathrm{bar}(f_i)$ stands for the barycenter of the face $f$.

\begin{lem}\label{barysubdivisequiv}
If $\mathcal{P}$ is a polytope, then $\mathrm{Sd}(\mathcal{P})$ is an $\mathrm{Aut}(\mathcal{P})$-triangulation of $\mathcal{P}$.
\end{lem}

\begin{proof}
It is well-known that $\mathrm{Sd}(\mathcal{P})$ triangulates $\mathcal{P}$ and it is clear that $\Gamma:=\mathrm{Aut}(\mathcal{P})$ permutes the simplices of $\mathrm{Sd}(\mathcal{P})$. We have to prove that, for a simplex $\sigma=(f_0,\dotsc,f_k)$ of $\mathrm{Sd}(\mathcal{P})$ and $\gamma\in\Gamma$, if $\gamma\sigma=\sigma$, then $\gamma x=x$ for each $x\in|\sigma|$ in the realization $|\sigma|$ of $\sigma$. Take $0\le i \le k$. The point $\mathrm{bar}(f_i)$ is taken by $\gamma$ to some $\mathrm{bar}(f_j)$ and since the barycenter of a polytope lies in its relative interior, we have $\gamma(\mathring{f_i})\cap\mathring{f_j}\ne\emptyset$ and as $\gamma$ acts as an automorphism of $\mathcal{P}$, this forces $\gamma(f_i)=f_j$ and $\dim(f_i)=\dim(\gamma(f_i))=\dim(f_j)$. But the sequence $(\dim f_0,\dotsc,\dim f_k)$ is increasing, so $f_i=f_j$ and $\mathrm{bar}(f_i)=\mathrm{bar}(f_j)=\gamma(\mathrm{bar}(f_i))$. The conclusion now follows from the equality $|\sigma|=\mathrm{conv}(\mathrm{bar}(f_0),\dotsc,\mathrm{bar}(f_k))$.
\end{proof}

From this we deduce that $W_\mathrm{a}\cdot\mathrm{Sd}(\overline{\mathcal{A}_0})$ is a $W_Y$-triangulation of $V^*$ for all $Q^\vee\subset Y\subset P^\vee$ at once. We have $\mathrm{vert}(\overline{\mathcal{A}_0})\approx S_0=\{0,\dotsc,n\}$ and $\overline{\mathcal{A}_0}\simeq\Delta^n$, so that the face lattice of $\overline{\mathcal{A}_0}$ is $F(\overline{\mathcal{A}_0})\simeq(\mathscr{P}(S_0),\subset)$. This gives a description of $F(\mathrm{Sd}(\overline{\mathcal{A}_0}))$: for $0\le d \le n$, we have
\[F_d(\mathrm{Sd}(\overline{\mathcal{A}_0}))=\{Z_\bullet=(Z_0,Z_1,\dotsc,Z_d)~;~\forall i,~\emptyset\ne Z_i\subset S_0,~Z_i\subsetneq Z_{i+1}\}\]
and $Z_\bullet\subset Z_\bullet'$ if $Z_\bullet$ is a subsequence of $Z_\bullet'$.

\begin{lem}\label{decofstabs}
The group $\Omega_Y$ acts on $\overline{\mathcal{A}_0}$ and this induces an action on $S_0$. The resulting action on $F(\mathrm{Sd}(\overline{\mathcal{A}_0}))$ corresponds to the action of $\Omega_Y$ on $|\mathrm{Sd}(\overline{\mathcal{A}_0})|=\overline{\mathcal{A}_0}$. Moreover, for $Z_\bullet\in F_d(\mathrm{Sd}(\overline{\mathcal{A}_0}))$, the stabilizer of $Z_\bullet$ in $W_Y$ decomposes as
\[(W_Y)_{Z_\bullet}=(W_\mathrm{a})_{Z_\bullet}\rtimes (\Omega_Y)_{Z_\bullet}=(W_\mathrm{a})_{S_0\setminus Z_d}\rtimes (\Omega_Y)_{Z_\bullet}~~\text{and}~~(\Omega_Y)_{Z_\bullet}=\bigcap_{i=0}^d\Omega_{Z_i}.\]
\end{lem}

\begin{proof}
The first statement is obvious. Write $Z_\bullet=(Z_0\subsetneq\cdots\subsetneq Z_d)$ and let $\widehat{w}:=w\omega_j\in(W_Y)_{Z_\bullet}$ with $w\in W_\mathrm{a}$ and $\omega_j\in\Omega_Y$. Then, for every $x\in|Z_{\bullet}|$, we have $\widehat{w}(x)=w(\omega_j(x))=x$ and $\omega_j(x)\in\overline{\mathcal{A}_0}$ so $x=\omega_j(x)$ and $\omega_j\in(\Omega_Y)_{Z_\bullet}$. On another hand we get $w(x)=x$ so $w\in(W_\mathrm{a})_{Z_\bullet}$. Now, an element $w\in W_\mathrm{a}$ fixes $Z_\bullet$ if and only if it fixes the maximal face $Z_d$ of $Z_\bullet$. This is indeed the parabolic subgroup $(W_\mathrm{a})_{S_0\setminus Z_d}$.
\end{proof}

\begin{center}
\resizebox{0.8\textwidth}{!}{
\begin{tabular}{|c|c|c|c|}
\hline
Type & Extended Dynkin diagram & Fundamental group $\Omega\simeq P/Q$ & Non-trivial elements of $\Omega\le\mathrm{Aut}({\mathcal{D}ynkin_0})$ \\
\hline
\hline
$\widetilde{A_1}$ & \begin{minipage}{0.5\textwidth}\centering \begin{tikzpicture}
	\coordinate (a) at (0,0);
	\coordinate (b) at (1,0);
	
	\draw (a) node[below]{$1$};
	\draw (b) node[below]{$0$};
	
	\draw (a)--(b);
	
	\fill[fill=white] (a) circle (2.5pt);
	\fill[fill=white] (b) circle (2.5pt);
	\node[mark size=2.5pt] at (b) {\pgfuseplotmark{otimes}};
	\draw (a) circle (2.5pt);
	
	\draw (1/2,0) node[above]{$\infty$};
\end{tikzpicture}\end{minipage} & $\Z/2\Z$ & $\omega_1=(0,1)$ \\
\hline
 $\widetilde{A_n}~(n\ge2)$ & \begin{minipage}{0.5\textwidth}\centering \begin{tikzpicture}[scale=1]
	\coordinate (a) at (-2,0);
	\coordinate (b) at (-1,0);
	\coordinate (c) at (0,0);
	\coordinate (d) at (1,0);
	\coordinate (e) at (2,0);
	\coordinate (f) at (0,1);
	\coordinate (bc) at (-1/2,0);
	\coordinate (cd) at (1/2,0);
	
	\draw (a) node[below]{$1$};
	\draw (b) node[below]{$2$};
	\draw (c) node{$\cdots$};
	\draw (d) node[below]{${n-1}$};
	\draw (e) node[below]{$n$};
	\draw (f) node[above]{$0$};
	
	\draw (a)--(b)
	(b)--(bc)
	(cd)--(d)
	(d)--(e)
	(e)--(f)
	(f)--(a);
	
	\fill[fill=white] (a) circle (2.5pt);
	\fill[fill=white] (b) circle (2.5pt);
	\fill[fill=white] (d) circle (2.5pt);
	\fill[fill=white] (e) circle (2.5pt);
	\fill[fill=white] (f) circle (2.5pt);
	
	\draw (a) circle (2.5pt);
	\draw (b) circle (2.5pt);
	\draw (d) circle (2.5pt);
	\draw (e) circle (2.5pt);
	\node[mark size=2.5pt] at (f) {\pgfuseplotmark{otimes}};
\end{tikzpicture}\end{minipage} & $\Z/(n+1)\Z$ & $\begin{array}{ll} \omega_{1}=(0,1,2,\cdots,n) \\ \\ \omega_{i}=(\omega_{1})^i,~1\le i \le n\end{array}$ \\
\hline
 & & & \\
$\widetilde{B_2}=\widetilde{C_2}$ & \begin{minipage}{0.5\textwidth}\centering \begin{tikzpicture}[scale=1]
	\coordinate (a) at (0,0);
	\coordinate (e) at (1,0);
	\coordinate (f) at (2,0);
	
	\draw (a) node[below]{$0$};
	\draw (e) node[below]{$1$};
	\draw (f) node[below]{$2$};
	
	\draw[decoration={markings, mark=at position 0.75 with {\arrow{angle 60}}},postaction={decorate},double distance=2.5pt] (a)--(e);
	\draw[decoration={markings, mark=at position 0.75 with {\arrow{angle 60}}},postaction={decorate},double distance=2.5pt] (f)--(e);
	
	\fill[fill=white] (a) circle (2.5pt);
	\node[mark size=2.5pt] at (a) {\pgfuseplotmark{otimes}};
	\fill[fill=black] (e) circle (2.5pt);
	\fill[fill=white] (f) circle (2.5pt);
    \draw (f) circle (2.5pt);
\end{tikzpicture}\end{minipage} & $\Z/2\Z$ & $\omega_{1}=(0,2)$ \\
 & & & \\
\hline
$\widetilde{B_n}~(n\ge3)$ & \begin{minipage}{0.5\textwidth}\centering \begin{tikzpicture}[scale=1]
	\coordinate (a) at (-1/2,1.732/2);
	\coordinate (b) at (-1/2,-1.732/2);
	\coordinate (c) at (1/2,0);
	\coordinate (d) at (1+1/2,0);
	\coordinate (e) at (2+1/2,0);
	\coordinate (f) at (3+1/2,0);
	\coordinate (g) at (4+1/2,0);
	
	\draw (a) node[left]{$1$};
	\draw (b) node[left]{$0$};
	\draw (c) node[below]{$2$};
	\draw (d) node[below]{$3$};
	\draw (e) node{$\cdots$};
	\draw (f) node[below]{${n-1}$};
	\draw (g) node[below]{$n$};
	
	\draw (a)--(c)
	(b)--(c)
	(c)--(d)
	(d)--(2,0)
	(3,0)--(f);
	\draw[decoration={markings, mark=at position 0.75 with {\arrow{angle 60}}},postaction={decorate},double distance=2.5pt] (f)--(g);
	
	\fill[fill=white] (a) circle (2.5pt);
	\fill[fill=black] (c) circle (2.5pt);
	\fill[fill=black] (d) circle (2.5pt);
	\fill[fill=black] (f) circle (2.5pt);
	\fill[fill=black] (g) circle (2.5pt);
	\fill[fill=white] (b) circle (2.5pt);
	\node[mark size=2.5pt] at (b) {\pgfuseplotmark{otimes}};
	\draw (a) circle (2.5pt);
\end{tikzpicture}\end{minipage} & $\Z/2\Z$ & $\omega_{1}=(0,1)$ \\
\hline
 & & & \\
$\widetilde{C_n}~(n\ge3)$ & \begin{minipage}{0.5\textwidth}\centering \begin{tikzpicture}[scale=1]
	\coordinate (a) at (-3,0);
	\coordinate (b) at (-2,0);
	\coordinate (c) at (-1,0);
	\coordinate (d) at (0,0);
	\coordinate (e) at (1,0);
	\coordinate (f) at (2,0);
	
	\draw (a) node[below]{$0$};
	\draw (b) node[below]{$1$};
	\draw (c) node[below]{$2$};
	\draw (d) node{$\cdots$};
	\draw (e) node[below]{${n-1}$};
	\draw (f) node[below]{$n$};
	
	\draw[decoration={markings, mark=at position 0.75 with {\arrow{angle 60}}},postaction={decorate},double distance=2.5pt] (a)--(b);
	\draw[decoration={markings, mark=at position 0.75 with {\arrow{angle 60}}},postaction={decorate},double distance=2.5pt] (f)--(e);
	\draw (b)--(c)
	(c)--(-1/2,0)
	(1/2,0)--(e);
	
	\fill[fill=white] (a) circle (2.5pt);
	\node[mark size=2.5pt] at (a) {\pgfuseplotmark{otimes}};
	\fill[fill=black] (b) circle (2.5pt);
	\fill[fill=black] (c) circle (2.5pt);
	\fill[fill=black] (e) circle (2.5pt);
	\fill[fill=white] (f) circle (2.5pt);
	
    \draw (f) circle (2.5pt);
\end{tikzpicture}\end{minipage} & $\Z/2\Z$ & $\displaystyle{\omega_{n}=(0,n)\prod_{i=1}^{\left\lfloor \frac{n-1}{2}\right\rfloor}(i,{n-i})}$ \\
 & & & \\
\hline
 & & & \\
$\widetilde{D_{2n}}~(n\ge2)$ & \begin{minipage}{0.5\textwidth}\centering \begin{tikzpicture}[scale=1]
	\coordinate (a) at (-1/2,1.732/2);
	\coordinate (b) at (-1/2,-1.732/2);
	\coordinate (c) at (1/2,0);
	\coordinate (d) at (1+1/2,0);
	\coordinate (e) at (2+1/2,0);
	\coordinate (f) at (3+1/2,0);
	\coordinate (g) at (4+1/2,1.732/2);
	\coordinate (h) at (4+1/2,-1.732/2);
	
	\draw (a) node[left]{$1$};
	\draw (b) node[left]{$0$};
	\draw (c) node[below]{$2$};
	\draw (d) node[below]{$3$};
	\draw (e) node{$\cdots$};
	\draw (f) node[right]{${2n-2}$};
	\draw (g) node[right]{${2n}$};
	\draw (h) node[right]{${2n-1}$};
	
	\draw (a)--(c)
	(b)--(c)
	(c)--(d)
	(d)--(2,0)
	(3,0)--(f)
	(f)--(g)
	(f)--(h);
	
	\fill[fill=white] (a) circle (2.5pt);
	\fill[fill=black] (c) circle (2.5pt);
	\fill[fill=black] (d) circle (2.5pt);
	\fill[fill=black] (f) circle (2.5pt);
	\fill[fill=white] (g) circle (2.5pt);
	\fill[fill=white] (b) circle (2.5pt);
	\fill[fill=white] (h) circle (2.5pt);
	\node[mark size=2.5pt] at (b) {\pgfuseplotmark{otimes}};
	\draw (a) circle (2.5pt);
	\draw (g) circle (2.5pt);
	\draw (h) circle (2.5pt);
\end{tikzpicture}\end{minipage} & $\Z/2\Z\oplus\Z/2\Z$ & $\displaystyle{\begin{array}{llll} \omega_{1}=(0,1)({2n-1},{2n}) \\ \\ \omega_{{2n-1}}=(0,{2n-1})(1,{2n})\prod_{i=2}^{n-1}(i,{2n-i}) \\ \\ \omega_{{2n}}=(0,{2n})(1,{2n-1})\prod_{i=2}^{n-1}(i,{2n-i})=\omega_{1}\omega_{{2n-1}}\end{array}}$ \\
 & & & \\
\hline
 & & & \\
$\widetilde{D_{2n+1}}~(n\ge2)$ & \begin{minipage}{0.5\textwidth}\centering \begin{tikzpicture}[scale=1]
	\coordinate (a) at (-1/2,1.732/2);
	\coordinate (b) at (-1/2,-1.732/2);
	\coordinate (c) at (1/2,0);
	\coordinate (d) at (1+1/2,0);
	\coordinate (e) at (2+1/2,0);
	\coordinate (f) at (3+1/2,0);
	\coordinate (g) at (4+1/2,1.732/2);
	\coordinate (h) at (4+1/2,-1.732/2);
	
	\draw (a) node[left]{$1$};
	\draw (b) node[left]{$0$};
	\draw (c) node[below]{$2$};
	\draw (d) node[below]{$3$};
	\draw (e) node{$\cdots$};
	\draw (f) node[right]{${2n-1}$};
	\draw (g) node[right]{${2n+1}$};
	\draw (h) node[right]{${2n}$};
	
	\draw (a)--(c)
	(b)--(c)
	(c)--(d)
	(d)--(2,0)
	(3,0)--(f)
	(f)--(g)
	(f)--(h);
	
	\fill[fill=white] (a) circle (2.5pt);
	\fill[fill=black] (c) circle (2.5pt);
	\fill[fill=black] (d) circle (2.5pt);
	\fill[fill=black] (f) circle (2.5pt);
	\fill[fill=white] (g) circle (2.5pt);
	\fill[fill=white] (b) circle (2.5pt);
	\fill[fill=white] (h) circle (2.5pt);
	\node[mark size=2.5pt] at (b) {\pgfuseplotmark{otimes}};
	\draw (a) circle (2.5pt);
	\draw (g) circle (2.5pt);
	\draw (h) circle (2.5pt);
\end{tikzpicture}\end{minipage} & $\Z/4\Z$ & $\displaystyle{\begin{array}{llll} \omega_{1}=(0,1)({2n},{2n+1}) \\ \\ \omega_{{2n}}=(0,{2n},1,{2n+1})\prod_{i=2}^{n}(i,{2n+1-i}) \\ \\ \omega_{{2n+1}}=(0,{2n+1},1,{2n})\prod_{i=2}^{n}(i,{2n+1-i})\end{array}}$ \\
 & & & \\
\hline
$\widetilde{E_6}$ & \begin{minipage}{0.5\textwidth}\centering \begin{tikzpicture}
	\coordinate (a) at (-2,0);
	\coordinate (b) at (-1,0);
	\coordinate (c) at (0,0);
	\coordinate (d) at (1,0);
	\coordinate (e) at (2,0);
	\coordinate (f) at (0,1);
	\coordinate (g) at (0,2);
	
	\draw (a) node[below]{$1$};
	\draw (b) node[below]{$3$};
	\draw (c) node[below]{$4$};
	\draw (d) node[below]{$5$};
	\draw (e) node[below]{$6$};
	\draw (f) node[right]{$2$};
	\draw (g) node[right]{$0$};
	
	\draw (a)--(b)
	(b)--(c)
	(c)--(d)
	(d)--(e)
	(c)--(f)
	(f)--(g);
	
	\fill[fill=white] (a) circle (2.5pt);
	\fill[fill=white] (e) circle (2.5pt);
	\fill[fill=white] (g) circle (2.5pt);
	\fill[fill=black] (b) circle (2.5pt);
	\fill[fill=black] (c) circle (2.5pt);
	\fill[fill=black] (d) circle (2.5pt);
	\fill[fill=black] (f) circle (2.5pt);
	\draw (a) circle (2.5pt);
	\draw (e) circle (2.5pt);
	\node[mark size=2.5pt] at (g) {\pgfuseplotmark{otimes}};
\end{tikzpicture}\end{minipage} & $\Z/3\Z$ & $\begin{array}{ll} \omega_{1}=(0,1,6)(2,3,5) \\ \\ \omega_{6}=(1,0,6)(3,2,5)=\omega_{1}^{-1}\end{array}$ \\
\hline
$\widetilde{E_7}$ & \begin{minipage}{0.5\textwidth}\centering \begin{tikzpicture}
	\coordinate (a) at (-2,0);
	\coordinate (b) at (-1,0);
	\coordinate (c) at (0,0);
	\coordinate (d) at (1,0);
	\coordinate (e) at (2,0);
	\coordinate (f) at (0,1);
	\coordinate (g) at (3,0);
	\coordinate (h) at (-3,0);
	
	\draw (a) node[below]{$1$};
	\draw (b) node[below]{$3$};
	\draw (c) node[below]{$4$};
	\draw (d) node[below]{$5$};
	\draw (e) node[below]{$6$};
	\draw (f) node[right]{$2$};
	\draw (g) node[below]{$7$};
	\draw (h) node[below]{$0$};
	
	\draw (a)--(b)
	(b)--(c)
	(c)--(d)
	(d)--(e)
	(c)--(f)
	(e)--(g)
	(h)--(a);
	
	\fill[fill=black] (a) circle (2.5pt);
	\fill[fill=black] (e) circle (2.5pt);
	\fill[fill=white] (g) circle (2.5pt);
	\fill[fill=white] (h) circle (2.5pt);
	\fill[fill=black] (b) circle (2.5pt);
	\fill[fill=black] (c) circle (2.5pt);
	\fill[fill=black] (d) circle (2.5pt);
	\fill[fill=black] (f) circle (2.5pt);
	\draw (g) circle (2.5pt);
	\node[mark size=2.5pt] at (h) {\pgfuseplotmark{otimes}};
\end{tikzpicture}\end{minipage} & $\Z/2\Z$ & $\omega_{7}=(0,7)(1,6)(3,5)$ \\
\hline
$\widetilde{E_8}$ & \begin{minipage}{0.5\textwidth}\centering \begin{tikzpicture}
	\coordinate (a) at (-2,0);
	\coordinate (b) at (-1,0);
	\coordinate (c) at (0,0);
	\coordinate (d) at (1,0);
	\coordinate (e) at (2,0);
	\coordinate (f) at (0,1);
	\coordinate (g) at (3,0);
	\coordinate (h) at (5,0);
	\coordinate (i) at (4,0);
	
	\draw (a) node[below]{$1$};
	\draw (b) node[below]{$3$};
	\draw (c) node[below]{$4$};
	\draw (d) node[below]{$5$};
	\draw (e) node[below]{$6$};
	\draw (f) node[right]{$2$};
	\draw (g) node[below]{$7$};
	\draw (h) node[below]{$0$};
	\draw (i) node[below]{$8$};
	
	\draw (a)--(b)
	(b)--(c)
	(c)--(d)
	(d)--(e)
	(c)--(f)
	(e)--(g)
	(h)--(i)
	(g)--(i);
	
	\fill[fill=black] (a) circle (2.5pt);
	\fill[fill=black] (e) circle (2.5pt);
	\fill[fill=black] (g) circle (2.5pt);
	\fill[fill=white] (h) circle (2.5pt);
	\fill[fill=black] (b) circle (2.5pt);
	\fill[fill=black] (c) circle (2.5pt);
	\fill[fill=black] (d) circle (2.5pt);
	\fill[fill=black] (f) circle (2.5pt);
	\fill[fill=black] (i) circle (2.5pt);
	\node[mark size=2.5pt] at (h) {\pgfuseplotmark{otimes}};
\end{tikzpicture}\end{minipage} & $1$ & $\varnothing$ \\
\hline
 & & & \\
\centering{$\widetilde{F_4}$} & \begin{minipage}{0.5\textwidth}\centering \begin{tikzpicture}
	\coordinate (a) at (-2,0);
	\coordinate (b) at (-1,0);
	\coordinate (c) at (0,0);
	\coordinate (d) at (1,0);
	\coordinate (e) at (2,0);
	
	\draw (a) node[below]{$0$};
	\draw (b) node[below]{$1$};
	\draw (c) node[below]{$2$};
	\draw (d) node[below]{$3$};
	\draw (e) node[below]{$4$};
	
	\draw (a)--(b)
	(b)--(c)
	(d)--(e);	
	\draw[decoration={markings, mark=at position 0.75 with {\arrow{angle 60}}},postaction={decorate},double distance=2.5pt] (c)--(d);
	
	\fill[fill=white] (a) circle (2.5pt);
	\node[mark size=2.5pt] at (a) {\pgfuseplotmark{otimes}};
	\fill[fill=black] (b) circle (2.5pt);
	\fill[fill=black] (c) circle (2.5pt);
	\fill[fill=black] (d) circle (2.5pt);
	\fill[fill=black] (e) circle (2.5pt);
\end{tikzpicture}\end{minipage} & $1$ & $\varnothing$ \\
\hline
 & & & \\
$\widetilde{G_2}$ & \begin{minipage}{0.5\textwidth}\centering \begin{tikzpicture}
	\coordinate (a) at (-1,0);
	\coordinate (b) at (0,0);
	\coordinate (c) at (1,0);
	
	\draw (a) node[below]{$1$};
	\draw (b) node[below]{$2$};
	\draw (c) node[below]{$0$};
	
	\draw (b)--(c);
	\draw[decoration={markings, mark=at position 0.75 with {\arrow{angle 60}}},postaction={decorate},double distance=2.5pt] (b)--(a);
	\draw (b)--(a);
	
	\fill[fill=white] (c) circle (2.5pt);
	\node[mark size=2.5pt] at (c) {\pgfuseplotmark{otimes}};
	\fill[fill=black] (a) circle (2.5pt);
	\fill[fill=black] (b) circle (2.5pt);
\end{tikzpicture}\end{minipage} & $1$ & $\varnothing$ \\
\hline
\end{tabular}}
\captionof{table}{Extended Dynkin diagrams and fundamental groups elements, represented as permutations of the nodes.}
\label{extendeddynkindiagrams}
\end{center}

To avoid too many choices, we fix a total ordering $\prec$ on $F(\mathrm{Sd}(\overline{\mathcal{A}_0}))$. For instance, the lexicographical order $<_{lex}$ induced by the order on $\mathscr{P}(S_0)=2^{S_0}$ inherited from the natural order on $S_0$. As the barycentric subdivision of $\overline{\mathcal{A}_0}$ is simplicial, the boundaries of the complex and the cup product are easily determined and lead to the following result:

\begin{theo}\label{complexforY}
For $0\le d\le n$, decompose the $\Omega_Y$-set $F_d(\mathrm{Sd}(\overline{\mathcal{A}_0}))$ into orbits
\[F_d(\mathrm{Sd}(\overline{\mathcal{A}_0}))/\Omega_Y\approx\{Z_{d,1}\prec\cdots\prec Z_{d,k_d}\},~~\text{where}~~Z_{d,i}=\min_{\prec}(\Omega_Y\cdot Z_{d,i}).\]
Denote further, for $0\le p \le d$ and $1\le i \le k_d$,
\[Z_{d,i}^{(p)}:=((Z_{d,i})_0,\dotsc,\widehat{(Z_{d,i})_p},\dotsc,(Z_{d,i})_d).\]
Then the complex $C^\mathrm{cell}_*(V^*,W_Y;\Z)$ is given by
\[C^\mathrm{cell}_d(V^*,W_Y;\Z)=\bigoplus_{i=1}^{k_d}\Z\left[W_Y/(W_Y)_{Z_{d,i}}\right],~~\text{with}~~(W_Y)_{Z_{d,i}}=(W_\mathrm{a})_{(Z_{d,i})_d^\complement}\rtimes \bigcap_{j=0}^d(\Omega_Y)_{(Z_{d,i})_j}.\]
The boundaries are given by
\[\partial_d(Z_{d,i})=\sum_{p=0}^d(-1)^p\omega_{p,i}(Z_{d-1,u_i}),~\text{where }u_i\in S_0~;~Z_{d-1,u_i}=\min_\prec(\Omega_Y\cdot Z_{d,i}^{(p)})~\text{and}~\omega_{p,i}(Z_{d-1,u_i})=Z_{d,i}^{(p)}.\]

Moreover, the dual complex $C^*_\mathrm{cell}(V^*,\widehat{W_\mathrm{a}};\Z)$ is a $W_Y$-dg-ring with product
\[Z_{d,i}^*\cup Z_{e,j}^*=\delta_{(Z_{d,i})_d,(Z_{e,j})_0}\omega(Z_{d+e,k})^*,\]
where
\[Z_{d+e,k}=\min_\prec(\Omega_Y\cdot((Z_{d,i})_0,\dotsc,(Z_{d,i})_d,(Z_{e,j})_0,\dotsc,(Z_{e,j})_e))~~\text{and}~~\omega(Z_{d+e,k})=((Z_{d,i})_0,\dotsc,(Z_{e,j})_e).\]

Finally, the complex for the torus $V^*/Y$ is given by
\[C^\mathrm{cell}_*(V^*/Y,W;\Z)=\mathrm{Def}_W^{W_Y}(C_*^\mathrm{cell}(V^*,W_Y;\Z)).\]
\end{theo}

\begin{exemple}\label{A2example}
Continuing the Example \ref{ex_with_A2}, we treat the extended type $A_2$, which is fairly computable by hand. We have $S_0=\{0,1,2\}=J$ and 
\[\Omega=\Omega_{P^\vee}=\{1,\underbrace{t_{\varpi_\alpha^\vee}s_\alpha s_\beta}_{\omega_\alpha},\underbrace{t_{\varpi_\beta^\vee}s_\beta s_\alpha}_{\omega_\beta}\}\simeq \Z/3\Z.\]
In this case, $W_{P^\vee}=\widehat{W_\mathrm{a}}$ is the classical extended affine Weyl group. Geometrically, the element $\omega_\alpha$ acts as the rotation with angle $2\pi/3$ around the barycenter of $\overline{\mathcal{A}_0}=\mathrm{conv}(0,\varpi_\alpha^\vee,\varpi_\beta^\vee)=:[0,1,2]\simeq\Delta^2$. The situation can be visualized in Figure \ref{alcoveA2}.

\begin{figure}[h!]
\centering
\begin{tikzpicture}[scale=2.2,rotate=45]  
  \coordinate (z) at (0,0);
  \coordinate (a) at (1,-1);
  \coordinate (b) at (0.3660254040,1.366025404);
  \coordinate (ab) at (1+0.3660254040,-1+1.366025404);
  \coordinate (ma) at (-1,1);
  \coordinate (mb) at (-0.3660254040,-1.366025404);
  \coordinate (mab) at (-1-0.3660254040,1-1.366025404);
  
  \coordinate (la) at (2/3+1/3*0.3660254040,-2/3+1/3*1.366025404);
  \coordinate (lb) at (1/3+2/3*0.3660254040,-1/3+2/3*1.366025404);
  
  \coordinate (lad) at (1/3+1/6*0.3660254040,-1/3+1/6*1.3660254040);
  \coordinate (lbd) at (1/6+1/3*0.3660254040,-1/6+1/3*1.3660254040);
  \coordinate (barlab) at (1/3+1/3*0.3660254040,-1/3+1/3*1.3660254040);
  
  \fill[fill=red] (la) circle (1pt);
  \fill[fill=red] (lb) circle (1pt);
  \fill[fill=red] (z) circle (1pt);
  
  \draw (la) node[right]{\scriptsize{$\varpi_\alpha^\vee$}};
  \draw (lb) node[left]{\scriptsize{$\varpi_\beta^\vee$}};
  \draw (0,0.05) node[left]{\scriptsize{$0$}};
  
  \fill[fill=blue,opacity=0.4] (z)--(la)--(lb);
  
  \draw (z)--(la);
  \draw (z)--(lb);
  \draw (la)--(lb);
  \draw (la)--(lbd);
  \draw (lb)--(lad);
  \draw (z)--($(z)!0.5!(ab)$);
  
  \coordinate (x) at (2+0.3660254040,-2+1.3660254040);
  \coordinate (y) at (1+2*0.3660254040,-1+2*1.3660254040);
  
  \draw ($(z)!0.82!(barlab)$) node[left]{\scriptsize{$e^2_2$}};
  \draw ($(z)!0.57!(barlab)$) node[right]{\scriptsize{$e^2_1$}};

\end{tikzpicture}
\caption{Barycentric subdivision $|\mathrm{Sd}(\overline{\mathcal{A}_0})|$ of the fundamental alcove $\overline{\mathcal{A}_0}$.}
\label{alcoveA2}
\end{figure}
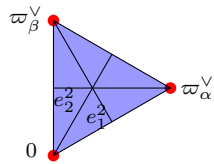

Therefore, the complex $C^\mathrm{cell}_*(V^*,\widehat{W_\mathrm{a}};\Z)$ is given by
\[\xymatrix{\Z[\widehat{W_\mathrm{a}}]^2 \ar^<<<<<{\partial_2}[r] & \Z[\widehat{W_\mathrm{a}}/\left<s_\beta\right>]\oplus\Z[\widehat{W_\mathrm{a}}/\left<s_\alpha\right>]\oplus\Z[\widehat{W_\mathrm{a}}]^2 \ar^{\partial_1}[r] & \Z[\widehat{W_\mathrm{a}}/W]\oplus\Z[\widehat{W_\mathrm{a}}/\left<s_\beta\right>]\oplus\Z[\widehat{W_\mathrm{a}}/\Omega]},\]
with
\[\partial_2=\left(\begin{smallmatrix}1 & 0 & -1 & 1 \\ 0 & -1 & 1 & -\omega_\beta\end{smallmatrix}\right),~~\partial_1=\left(\begin{smallmatrix}-1 & 1 & 0 \\ -1 & \omega_\beta & 0 \\ -1 & 0 & 1 \\ 0 & -1 & 1\end{smallmatrix}\right).\]
Moreover, the root datum $(P,\Phi,P^\vee,\Phi^\vee)$ may be realized by the Lie group $PSU(3)=SU(3)/\mu_3$ with torus $T=T_0/\mu_3\simeq V^*/P^\vee$, where $T_0=S(U(1)^3)$ is the standard torus of diagonal matrices of $SU(3)$. The complex $C^\mathrm{cell}_*(T,W;\Z)=\mathrm{Def}^{\widehat{W_\mathrm{a}}}_W(C^\mathrm{cell}_*(V^*,\widehat{W_\mathrm{a}};\Z))$ then becomes
\[\xymatrix{\Z[W]^2 \ar^<<<<<{\overline{\partial_2}}[r] & \Z[W/\left<s_\beta\right>]\oplus\Z[W/\left<s_\alpha\right>]\oplus\Z[W]^2 \ar^<<<<<{\overline{\partial_1}}[r] & \Z\oplus\Z[W/\left<s_\beta\right>]\oplus\Z[W/\left<s_\alpha s_\beta\right>]},\] 
with
\[\overline{\partial_2}=\left(\begin{smallmatrix}1 & 0 & -1 & 1 \\ 0 & -1 & 1 & -s_\beta s_\alpha \end{smallmatrix}\right),~~\overline{\partial_1}=\left(\begin{smallmatrix}-1 & 1 & 0 \\ -1 & s_\beta s_\alpha & 0 \\ -1 & 0 & 1 \\ 0 & -1 & 1\end{smallmatrix}\right).\]

The complexes $C^\mathrm{cell}_*(T_0,W;\Z)$ and $C^\mathrm{cell}_*(T,W;\Z)$ may be obtained using the commands 
\newline
\emph{\texttt{ComplexForFiniteCoxeterGroup("A",2)}}and \emph{\texttt{CellularComplexT("A",2,[0,1,2])}} provided by the package \emph{\texttt{Salvetti-and-tori-complexes}}\footnote{\url{https://github.com/arthur-garnier/Salvetti-and-tori-complexes}}.
\end{exemple}

\part{Hyperbolic tori for non-crystallographic Coxeter groups}

The goal of this part is to construct a smooth manifold affording a dg-algebra with a similar combinatorics as the one in Theorem \ref{cupprodparabolics} and playing the role of a torus for non-crystallographic Coxeter groups. First, we will define \emph{compact hyperbolic extensions} of non-crystallographic finite Coxeter groups and the desired manifold will then be constructed as an orbit space of the \emph{Coxeter complex} of the hyperbolic extension.

\section{Construction of the hyperbolic extensions and the hyperbolic torus}

Throughout the following three subsections we fix, once and for all, a finite irreducible Coxeter system $(W,S)$ of rank $n>1$.

\subsection{Affine and compact hyperbolic one-node extensions of Coxeter groups}
\hfill

Let us first recall some basic terminology concerning Coxeter groups; see \cite{bourbaki456} and \cite{humphreys-reflectiongroups}. We write
\[W=\left<s_1,\dotsc,s_n~|~(s_is_j)^{m_{i,j}}=1\right>,\]
with $M=(m_{i,j})_{1\le i,j \le n}$ the Coxeter matrix of $(W,S)$. Recall that we may define a symmetric bilinear form $B$ on the formal vector space $V:=\mathrm{span}_\R(\alpha_i,~1\le i \le n)$ by
\[B(\alpha_i,\alpha_j):=-\cos\left({\pi}/{m_{i,j}}\right)\]
as well as linear mappings
\[\forall 1 \le i \le n,~\sigma_i:=(v\mapsto v-2B(\alpha_i,v)\alpha_i).\]
Then the assignment $s_i \mapsto \sigma_i$ extends uniquely to a faithful irreducible representation
\[\sigma : W \longrightarrow O(V,B)\le GL(V),\]
known as the \emph{geometric representation} of $W$. Moreover, $W$ is finite (resp. affine) if and only if the form $B$ is positive definite (resp. positive semidefinite).

\begin{prop-def}[{\cite[\S 6.8]{humphreys-reflectiongroups}}]\label{hyperbolics}
The followings are equivalent
\begin{enumerate}[label=(\roman*)]
\item The form $B$ has signature $(n-1,1)$ and $B(\lambda,\lambda)<0$ for $\lambda\in \widehat{V}^*$ in the \emph{fundamental chamber} $C$ (i.e. such that $\left<\lambda,\alpha_s\right>>0$ for each $s\in\widehat{S}$),
\item The form $B$ is non-degenerate but not positive and the graph obtained by removing any vertex from the graph of $W$ is of non-negative type (i.e. is finite or affine).
\end{enumerate}

If these conditions occur, then $W$ is said to be \emph{hyperbolic}. If the second condition is enhanced by requiring that any such sub-graph is of positive definite type, then $W$ is said to be \emph{compact hyperbolic}.
\end{prop-def}

\begin{prop-def}\label{cpcthypextsofnoncryst}
If $W$ is a Weyl group, we denote by $r_W$ the reflection associated to the highest root of the root system of $W$. In the other cases we choose $r_W\in W$ to be the following reflection in $W$:
\[r_W:=\left\{\begin{array}{ccc}
(s_1s_2)^{\left\lfloor\frac{m-1}{2}\right\rfloor}s_1 & \text{if} & W=I_2(m),~m\ge3, \\[.5em]
s_3^{(s_2s_1)^2} & \text{if} & W=H_3, \\[.5em]
s_4^{\left(s_1^{s_2s_3}(s_1s_2)^2s_3s_4\right)^2} & \text{if} & W=H_4. \end{array}\right.\]

Define
\[\widehat{W}:=\left<\widehat{s}_0,\widehat{s}_1,\dotsc,\widehat{s}_r~\left|~\forall i,j\ge1,~(\widehat{s}_i\widehat{s}_j)^{m_{i,j}}=(\widehat{s}_0\widehat{s}_i)^{o(r_Ws_i)}={\widehat{s}_0}^2=1\right.\right>,\]
where $o(x)$ is the order of the element $x$ and $\widehat{S}:=\{\widehat{s}_0,\dotsc,\widehat{s}_n\}$. If $W$ is crystallographic, then $(\widehat{W},\widehat{S})$ is the usual Coxeter system corresponding to the affine Weyl group $W_{\rm a}$.

In the other cases, the pair $(\widehat{W},\widehat{S})$ is a compact hyperbolic Coxeter system, whose Coxeter graph is as in the following table:

\begin{table}[h!]
\resizebox{1.0\textwidth}{!}{
\begin{tabular}{|c||c|c|c|c|c|}
\hline
\emph{Extension} & $\widehat{I_2(m)}~(m\equiv 1[2])$ & $\widehat{I_2(m)}~(m\equiv 0[4])$ & $\widehat{I_2(m)}~(m\equiv 2[4])$ & $\widehat{H_3}$ & $\widehat{H_4}$ \\
\hline
\hline
\emph{Coxeter graph} & \begin{minipage}{0.18\textwidth}\centering \begin{tikzpicture}
	\coordinate (a) at (-1,0);
	\coordinate (b) at (0,0);
	\coordinate (c) at (-1/2,1);
	
	\draw (a) node[below]{$1$};
	\draw (b) node[below]{$2$};
	\draw (c) node[above]{$0$};
	\draw (-1/2,0) node[below]{$m$};
	\draw (-1/4,1/2) node[right]{$m$};
	\draw (-3/4,1/2) node[left]{$m$};
	
	\draw (a)--(b);
	\draw (b)--(c);
	\draw (a)--(c);
	
	\fill[fill=white] (c) circle (2.5pt);
	\node[mark size=2.5pt] at (c) {\pgfuseplotmark{otimes}};
	\fill[fill=black] (a) circle (2.5pt);
	\fill[fill=black] (b) circle (2.5pt);
\end{tikzpicture}\end{minipage} & \begin{minipage}{0.18\textwidth}\centering \begin{tikzpicture}
	\coordinate (a) at (-1,0);
	\coordinate (b) at (0,0);
	\coordinate (c) at (-2,0);
	
	\draw (a) node[below]{$1$};
	\draw (b) node[below]{$2$};
	\draw (c) node[below]{$0$};
	\draw (-3/2,0) node[above]{$m$};
	\draw (-1/2,0) node[above]{$m$};
	
	\draw (a)--(b);
	\draw (a)--(c);
	
	\fill[fill=white] (c) circle (2.5pt);
	\node[mark size=2.5pt] at (c) {\pgfuseplotmark{otimes}};
	\fill[fill=black] (a) circle (2.5pt);
	\fill[fill=black] (b) circle (2.5pt);
\end{tikzpicture}\end{minipage} & \begin{minipage}{0.18\textwidth}\centering \begin{tikzpicture}
	\coordinate (a) at (-1,0);
	\coordinate (b) at (0,0);
	\coordinate (c) at (-2,0);
	
	\draw (a) node[below]{$1$};
	\draw (b) node[below]{$2$};
	\draw (c) node[below]{$0$};
	\draw (-3/2,0) node[above]{${m}/{2}$};
	\draw (-1/2,0) node[above]{$m$};
	
	\draw (a)--(b);
	\draw (a)--(c);
	
	\fill[fill=white] (c) circle (2.5pt);
	\node[mark size=2.5pt] at (c) {\pgfuseplotmark{otimes}};
	\fill[fill=black] (a) circle (2.5pt);
	\fill[fill=black] (b) circle (2.5pt);
\end{tikzpicture}\end{minipage} & \begin{minipage}{0.13\textwidth}\centering \begin{tikzpicture}
	\coordinate (a) at (-1,0);
	\coordinate (b) at (0,0);
	\coordinate (c) at (0,1);
	\coordinate (d) at (-1,1);
	
	\draw (a) node[left]{$1$};
	\draw (b) node[right]{$2$};
	\draw (c) node[right]{$3$};
	\draw (d) node[left]{$0$};
	\draw (-1/2,0) node[below]{$5$};
	\draw (-1/2,1) node[above]{$5$};
	
	\draw (a)--(b);
	\draw (b)--(c);
	\draw (c)--(d);
	\draw (d)--(a);
	
	\fill[fill=white] (d) circle (2.5pt);
	\node[mark size=2.5pt] at (d) {\pgfuseplotmark{otimes}};
	\fill[fill=black] (c) circle (2.5pt);
	\fill[fill=black] (a) circle (2.5pt);
	\fill[fill=black] (b) circle (2.5pt);
\end{tikzpicture}\end{minipage} & \begin{minipage}{0.3\textwidth}\centering \begin{tikzpicture}
	\coordinate (a) at (-1,0);
	\coordinate (b) at (0,0);
	\coordinate (c) at (1,0);
	\coordinate (d) at (2,0);
	\coordinate (e) at (3,0);
	
	\draw (a) node[below]{$1$};
	\draw (b) node[below]{$2$};
	\draw (c) node[below]{$3$};
	\draw (d) node[below]{$4$};
	\draw (e) node[below]{$0$};
	\draw (-1/2,0) node[above]{$5$};
	\draw (5/2,0) node[above]{$5$};
	
	\draw (a)--(b);
	\draw (b)--(c);
	\draw (c)--(d);
	\draw (d)--(e);
	
	\fill[fill=white] (e) circle (2.5pt);
	\node[mark size=2.5pt] at (e) {\pgfuseplotmark{otimes}};
	\fill[fill=black] (d) circle (2.5pt);
	\fill[fill=black] (c) circle (2.5pt);
	\fill[fill=black] (a) circle (2.5pt);
	\fill[fill=black] (b) circle (2.5pt);
\end{tikzpicture}\end{minipage} \\
\hline
\end{tabular}}
\captionof{table}{Compact hyperbolic extensions of $I_2(m)$, $H_3$ and $H_4$.}
\label{hypextdiags}
\end{table}

In types $H_3$ and $H_4$, the reflection $r_W$ is the only one for which the group $\widehat{W}$ is compact hyperbolic.
\end{prop-def}

\begin{proof}
The first statement on the case where $W$ is a Weyl group is standard, see for instant \cite[\S 4.6]{humphreys-reflectiongroups}. Observe that for a crystallographic dihedral group $W=I_2(m)=\left<s,t\right>$ (i.e. when $m\in\{3,4,6\}$), the above expression $r_W=(st)^{\left\lfloor\frac{m-1}{2}\right\rfloor}s=s(ts)^{\left\lfloor\frac{m-1}{2}\right\rfloor}$ is a reduced expression of the reflection associated to the highest root, so that we obtain a single expression for all irreducible dihedral groups at once.

In the non-crystallographic case, the fact that $(\widehat{W},\widehat{S})$ is compact hyperbolic is straightforward, as the compact hyperbolic Coxeter graphs are well-known, see \cite[\S 6.9]{humphreys-reflectiongroups} or the original work \cite[Appendices]{chein}. The statement concerning the unicity comes from a tedious, but elementary verification on the 15 (resp. 60) reflections of $H_3$ (resp. $H_4$): only the reflection $r_W$ from the statement gives a graph which appears in the table of \cite{humphreys-reflectiongroups}.
\end{proof}

The very definitions of $W$ and $\widehat{W}$ as finitely presented groups lead to the following result:

\begin{cor}\label{projfromWhattoW}
The assignment
\[\widehat{s}_0 \mapsto r_W,~\widehat{s}_i \mapsto s_i~(i\ge1)\]
extends (uniquely) to a surjective reflection-preserving group homomorphism
\[\begin{array}{ccc}\widehat{W}\stackrel{\tiny{\pi}}\longtwoheadrightarrow W.\end{array}\]

Moreover, if $r_W=s_{i_1}\cdots s_{i_k}$ is a reduced expression of $r_W$, then the element $\widehat{r_W}=\widehat{s}_{i_1}\cdots\widehat{s}_{i_k}\in\widehat{W}$ is well-defined and we have
\[\ker\pi=\left<(\widehat{s}_0\widehat{r_W})^{\widehat{W}}\right>,\]
that is, $\ker(\pi)$ is the normal closure of $\widehat{s}_0\widehat{r_W}$ in $\widehat{W}$.
\end{cor}

\begin{proof}
For every expression of $r_W$ as in the statement, we have $\widehat{r_W}\in\widehat{W}_{\{1,\dotsc,n\}}\simeq W$ so that $\widehat{r_W}$ doesn't depend on the chosen reduced expression. The kernel of $\pi$ certainly contains the subgroup $N:=\left<(\widehat{s}_0\widehat{r_W})^{\widehat{W}}\right>$ and we easily find a presentation of $\widehat{W}/N$ by adding the relation $\widehat{s}_0=\widehat{s}_{i_1}\cdots \widehat{s}_{i_k}$ to the already known relations for $\widehat{W}$. The composite
\[\left<\overline{s}_0,\overline{s}_1,\dotsc,\overline{s}_n~|~\forall i,j\ge1,~(\overline{s}_i\overline{s}_j)^{m_{i,j}}=1,~\overline{s}_0=\overline{s}_{i_1}\cdots\overline{s}_{i_k}\right>\simeq\widehat{W}/N \longtwoheadrightarrow \widehat{W}/\ker\pi=W\]
maps $\overline{s}_i$ to $s_i$ and is an isomorphism. In particular, this yields an isomorphism of $\widehat{W}$-sets
\[\widehat{W}/N \simeq \widehat{W}/{\ker\pi},\]
forcing $\ker(\pi)$ and $N$ to be conjugate in $\widehat{W}$, hence equal.
\end{proof}

\begin{definition}\label{defofQ}
We denote the kernel of the projection from the previous Corollary by
\[Q:=\ker\pi=\left<(\widehat{s}_0\widehat{r_W})^{\widehat{W}}\right>.\]
\end{definition}

\begin{cor}
With the notation above, we have
\[\widehat{W}=Q\rtimes W.\]
\end{cor}

\begin{rem}\label{unicity_dihedrals}
In the dihedral case, there is no unicity on the reflection, nor on the diagram. Our choice here is motivated by the fact that, just like any group except for $H_3$, the reflection $r_W$ is of highest length among the reflections of $W$. First, by \cite[\S 6.9]{humphreys-reflectiongroups}, every one-node extension of a non-crystallographic $I_2(m)=\left<s,t\right>$ is compact hyperbolic. Next, take a reflection $r\in I_2(m)$ with $o(sr)=:p$ and $o(tr)=:q$, so that we obtain the diagram of Figure \ref{dih_diag}.

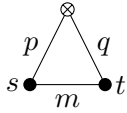
\begin{figure}[h!]
\begin{tikzpicture}
	\coordinate (a) at (-1,0);
	\coordinate (b) at (0,0);
	\coordinate (c) at (-1/2,1);
	
	\draw (a) node[left]{$s$};
	\draw (b) node[right]{$t$};
	\draw (-1/2,0) node[below]{$m$};
	\draw (-1/4,1/2) node[right]{$q$};
	\draw (-3/4,1/2) node[left]{$p$};
	
	\draw (a)--(b);
	\draw (b)--(c);
	\draw (a)--(c);
	
	\fill[fill=white] (c) circle (2.5pt);
	\node[mark size=2.5pt] at (c) {\pgfuseplotmark{otimes}};
	\fill[fill=black] (a) circle (2.5pt);
	\fill[fill=black] (b) circle (2.5pt);
\end{tikzpicture}
\caption{Extended dihedral diagram (compact hyperbolic for $5\le m\ne 6$).}\label{dih_diag}
\end{figure}
Then either $p$ or $q$ is equal to $2$ or both $p$ and $q$ divide $m$. Moreover, for any divisor $p\ne m$ of $m$, there is a diagram of this form, with reflection $r=s(st)^{m/p}$. The extreme case $r=s$ corresponds to $(p,q)=(2,m)$.

The number of reflections giving a diagram with $(p,q)=(m,m)$ is the number of $0<i<m$ such that both $(st)^i$ and $(st)^{i-1}$ have order $m$, i.e. such that $m$ is relatively prime to both $i$ and $i-1$. Thus, the reflection is far from unique: if for instance $m$ is prime, then any reflection different from $s$ and $t$ yields such a diagram.

However, we can classify the reflections giving an ``L'' diagram, i.e. one for which $q=2$ for instance. Write $r=:t(st)^j$ with $0\le j<m$. We have $1=(tr)^2=(st)^{2j}$ so that $m$ divides $2j$: say $2j=km$. But we have $0\le km=2j<2m$ so $k=0,1$. If $k=0$ then $r=t$ and so $(p,q)=(m,2)$ and if $k=1$, then $m=2j$ is even and $r=t(st)^{m/2}$. In particular, there is only one reflection for each diagram. If now $m=2e$ is even and $r=t(st)^{m/2}$, then the order of $sr=(st)^{m/2+1}$ is given by
\[o(sr)=\frac{m}{\gcd\!\left(m,\frac{m}{2}+1\right)}.\]
If $e=2\ell+1$ is odd, then $2\ell+1$ and $\ell+1$ are relatively prime so $o(sr)=e=m/2$ and if $e=2\ell$ is even, then $2\ell+1$ and $4\ell$ are relatively prime and thus $o(sr)=2e=m$.

Hence, for $m>2$, the only possible ``L'' extensions of $I_2(m)$ and associated reflections are as follows: 
\begin{enumerate}
\item If $m\ne 2\pmod 4$ then the unique extension with $q=2$ has $p=m$. In the case $m=0\pmod 4$, there are two suitable reflections, otherwise the reflection is unique.
\item If $m=2\pmod 4$, then there are two possible diagrams with $q=2$: they have $p=m$ and $p=m/2$ and there is exactly one reflection realizing each one of them.
\end{enumerate}
Moreover, if $6\ne m\ge 5$, then these extensions are compact hyperbolic.

Finally, for the crystallographic dihedral cases, if $m=3,4$ then the only infinite one-node extensions are affine, and any one-node extension of $I_2(6)$ either has $\{p,q\}=\{3,2\}$ and is the affine extension of $G_2$, or has $\{p,q\}=\{3,6\}$ and is compact hyperbolic.
\end{rem}

A (non-necessarily crystallographic) root system $\Phi$ may be associated to $W$. More precisely, the finite set $\Phi:=\bigcup_i\sigma(W)(\alpha_i)\subset V$ forms a root system in $V$, Euclidean with respect to the Tits form $B$, except that it is no longer assumed to satisfy the condition $\left<\alpha^\vee,\beta\right>\in\Z$, see \cite[\S 5.4]{humphreys-reflectiongroups}. We can mimic the case of the affine extension of a Weyl group to produce an affine representation of $\widehat{W}$ on $V$ in general. However, this representation is faithful only when $W$ is a Weyl group. More precisely, we have the following result:
\begin{prop}\label{Qisnotabelian}
Choose a root $\widetilde{\alpha}\in\Phi$ whose associated reflection is $\sigma(r_W)$ and let $\mathrm{t}_{\widetilde{\alpha}} : V\to V$ be the translation with vector $\widetilde{\alpha}$. The assignment $\widehat{s}_i\mapsto\sigma(s_i)$ for $i\ge1$ and $\widehat{s}_0\mapsto\mathrm{t}_{\widetilde{\alpha}}\sigma(r_W)$ defines an affine representation
\[\mathbf{a} : \widehat{W}\longto\mathrm{Aff}(V),\]
sending $Q$ to a free abelian group of finite rank, isomorphic to the subgroup $\left<\sigma(W)(\widetilde{\alpha})\right>$ of $V$ generated by the $\sigma(W)$-orbit of $\widetilde{\alpha}$. Moreover, the following statements are equivalent:
\begin{enumerate}[label=(\roman*)]
\item The group $W$ is a Weyl group.
\item The representation $\mathbf{a}$ is faithful.
\item The subgroup $Q\unlhd\widehat{W}$ is abelian.
\item The Coxeter group $\widehat{W}$ is affine.
\item The subgroup $\left<\sigma(W)(\widetilde{\alpha})\right>$ is a lattice in $V$.
\item The image of $\mathbf{a}$ is a discrete subgroup of $\mathrm{Aff}(V)$.
\end{enumerate}
\end{prop}
\begin{proof}
To see that $\mathbf{a}$ is well-defined, we just have to check that for $i\ge1$ and $n_i:=o(r_Ws_i)$, we have the relation $(\mathbf{a}(\widehat{s}_0)\mathbf{a}(\widehat{s}_i))^{n_i}=1$. We compute
\[(\mathbf{a}(\widehat{s}_0)\mathbf{a}(\widehat{s}_i))^{n_i}=(\mathrm{t}_{\widetilde{\alpha}}\sigma(r_Ws_i))^{n_i}=\mathrm{t}_{\widetilde{\alpha}+\sigma(r_Ws_i)(\widetilde{\alpha})+\cdots+\sigma(r_Ws_i)^{n_i-1}(\widetilde{\alpha})}=\mathrm{t}_v,\]
where $v:=\sum_{j=1}^{n_i}\sigma(r_Ws_i)^j(\widetilde{\alpha})$. Observe now that
\[\sigma(r_W)(v)=\sum_j\sigma(r_W)\sigma(r_Ws_i)^j(\widetilde{\alpha})=\sum_j\sigma(r_W)\sigma(s_ir_W)^j(\widetilde{\alpha})=\sum_j\sigma(r_Ws_i)^j\sigma(r_W)(\widetilde{\alpha})=-v\]
and since $v=\sigma(s_ir_W)(v)=-\sigma(s_i)(v)$, we also get $\sigma(s_i)(v)=-v$ and thus, $v\in\R\widetilde{\alpha}\cap\R\alpha_i=\{0\}$, as required.

Now, we have the following description of $\mathbf{a}(Q)$:
\[\mathbf{a}(Q)=\left<\mathrm{t}_x~;~x\in\sigma(W)(\widetilde{\alpha})\right>.\]
Indeed, using the Corollary \ref{projfromWhattoW}, we have
\[\mathbf{a}(Q)=\mathbf{a}\left(\left<(\widehat{s}_0r_W)^{\widehat{W}}\right>\right)=\left<\mathbf{a}(\widehat{s}_0r_W)^{\mathbf{a}(\widehat{W})}\right>=\left<\mathrm{t}_{\widetilde{\alpha}}^{\mathbf{a}(\widehat{W})}\right>=\left<\mathrm{t}_x~;~x\in\mathbf{a}(\widehat{W})(\widetilde{\alpha})\right>\]
and it thus suffices to see that $\left<\mathbf{a}(\widehat{W})(\widetilde{\alpha})\right>\subset\left<\sigma(W)(\widetilde{\alpha})\right>$. Take $x=\mathbf{a}(w)(\widetilde{\alpha})\in\mathbf{a}(\widehat{W})(\widetilde{\alpha})$ and choose a reduced expression $w=\widehat{s}_{i_1}\cdots\widehat{s}_{i_k}$ in $\widehat{W}$. We proceed by induction on the length $k=\ell(w)$ of $w$, the result being trivial if $k=0$. Since $\left<\sigma(W)(\widetilde{\alpha})\right>$ is $\sigma(W)$-stable, we may assume that $i_1=0$. By induction hypothesis, the element $y:=\mathbf{a}(\widehat{s}_{i_2}\cdots\widehat{s}_{i_k})(\widetilde{\alpha})$ is in $\left<\sigma(W)(\widetilde{\alpha})\right>$ and thus, $x=\mathbf{a}(\widehat{s}_0)(y)=\widetilde{\alpha}+\sigma(r_W)(y)\in\left<\sigma(W)(\widetilde{\alpha})\right>$, as expected. It follows that $\mathbf{a}(Q)\simeq\left<\sigma(W)(\widetilde{\alpha})\right>\le (V,+)$ is abelian and of finite rank, hence also free.

Let us move on to the equivalence and prove $(i)\Rightarrow (ii)\Rightarrow (iii)\Rightarrow (iv)\Rightarrow (i)$ first.
\begin{itemize}
\item If $W$ is a Weyl group, then the presentation of $\widehat{W}$ in the Proposition-Definition \ref{cpcthypextsofnoncryst} is precisely the same as the one for $W_{\rm a}$ (see \cite[\S 4.6]{humphreys-reflectiongroups}) and since $\mathbf{a}(\widehat{s}_0)=\mathrm{t}_{\alpha_0}s_{\alpha_0}$, where $\alpha_0\in\Phi^+$ is the highest root, the map $\mathbf{a}$ sends the generators of $\widehat{W}$ on those of $W_{\rm a}$, hence is injective and identifies $\widehat{W}$ with $W_{\rm a}\le\mathrm{Aff}(V)$.
\item If $\mathbf{a}$ is injective, then $Q\simeq\mathbf{a}(Q)\simeq\left<\sigma(W)(\widetilde{\alpha})\right>$ is abelian.
\item By definition, if $Q$ is abelian then $\widehat{W}$ is virtually abelian. This ensures that $\widehat{W}$ is affine, by \cite[Proposition 3.10]{davis-okun} (see also \cite[Corollary 1.5]{qi_thesis}).
\item If $W$ is non-crystallographic, then $\widehat{W}$ is compact hyperbolic and in particular, not affine.
\end{itemize}

We now prove that $(i)\Rightarrow (v)\Rightarrow (vi)\Rightarrow (i)$.
\begin{itemize}
\item If $W$ is a Weyl group, then $\widetilde{\alpha}=\alpha_0$ is the highest root and $\left<\sigma(W)(\alpha_0)\right>=\left<\Phi\right>=\Z\Phi$ is the root lattice of the root system $\Phi$ of $W$.
\item If $\left<\sigma(W)(\widetilde{\alpha})\right>$ is a lattice, then it is in particular discrete and thus
\[\ima(\mathbf{a})=\left\{\mathrm{t}_x\sigma(w)~;~x\in\left<\sigma(W)(\widetilde{\alpha})\right>,~w\in W\right\}=\bigcup_{w\in W}\sigma(w)(\mathbf{a}(Q))\]
is discrete, as a finite union of discrete sets.
\item We finally prove that if $W$ is non-crystallographic, then $\ima(\mathbf{a})$ is not discrete. In particular, it suffices to prove that $\mathbf{a}(Q)\simeq\left<\sigma(W)(\widetilde{\alpha})\right>\le(V,+)$ is not discrete. However, any discrete subgroup $D\le V$ is free of finite rank over $\Z$, so that for $\lambda\in\R\setminus\Q$, there is no $x\in D\setminus\{0\}$ such that $\lambda x\in D$. Therefore, to conclude that $\left<\sigma(W)(\widetilde{\alpha})\right>$ is not discrete, it suffices to find $x\in\sigma(W)(\widetilde{\alpha})$ such that $\lambda x\in\left<\sigma(W)(\widetilde{\alpha})\right>$ for some irrational $\lambda\in\R$. If $W$ is of type $H_3$ of $H_4$, then the reflections form a single conjugacy class and thus, there is only one orbit of roots and $\sigma(W)(\widetilde{\alpha})=\Phi$. Labelling the simple roots as in the diagrams in the Table \ref{hypextdiags}, we have in particular $\alpha_1\in\Phi$ and $2\cos(\pi/5)\alpha_1=s_2(\alpha_1)-\alpha_1\in\left<\Phi\right>$, allowing us to conclude because $2\cos(\pi/5)\notin\Q$. Assume finally that $W=I_2(m)=\left<s,t~|~s^2=t^2=(st)^m=1\right>$ with $m>2$. Since $W$ contains two conjugacy classes of reflections and up to re-labelling, we may assume that $r_W$ is conjugate to $s$, so that $\alpha_s\in\sigma(W)(\widetilde{\alpha})$. Since $W$ is assumed not to be crystallographic, we have $m\notin\{3,4,6\}$, so $\cos(2\pi/m)\notin\Q$ and hence, the real $c:=2\cos(\pi/m)$ has an irrational square. But we have
\begin{align*}
\sigma(st)(\alpha_s)&=\sigma(s)(\alpha_s-2B(\alpha_s,\alpha_t)\alpha_t)=\sigma(s)(\alpha_s+c\alpha_t) \\
&=-\alpha_s+c(\alpha_t+c\alpha_s)=(c^2-1)\alpha_s+c\alpha_t=(c^2-2)\alpha_s+(\alpha_s+c\alpha_t) \\
&=(c^2-2)\alpha_s+\sigma(t)(\alpha_s),
\end{align*}
thus, the element $(c^2-2)\alpha_s=\sigma(st)(\widetilde{\alpha})-\sigma(t)(\widetilde{\alpha})$ is in $\left<\sigma(W)(\widetilde{\alpha})\right>$, while $c^2-2\notin\Q$.
\end{itemize}
\end{proof}

\begin{rem}
We make the following observations:
\begin{itemize}[leftmargin=*]
\item Even in the non-crystallographic case, we may choose a highest root in the root system $\Phi$ and, as it may be checked case by case, if $W\ne H_3$, then the reflection associated to this highest root is indeed $r_W$. The extension of $H_3$ corresponding to the highest root has been considered in \cite{patera-twarock} and has the following Coxeter graph
\begin{center}
\begin{tikzpicture}[scale=0.8]
	\coordinate (a) at (-1,0);
	\coordinate (b) at (0,0);
	\coordinate (c) at (1,0);
	\coordinate (z) at (0,1);
	
	\draw (-1/2,0) node[above]{$5$};
	\draw (1/2,1/2) node[left]{$5$};
	
	\draw (a)--(b);
	\draw (b)--(c);
	\draw (b)--(z);
	
	\fill[fill=black] (c) circle (2.5pt);
	\fill[fill=black] (a) circle (2.5pt);
	\fill[fill=black] (b) circle (2.5pt);
	\fill[fill=black] (z) circle (2.5pt);
\end{tikzpicture}
\end{center}
The sub-graph 
\begin{tikzpicture}[scale=0.6]
	\coordinate (a) at (-1,0);
	\coordinate (b) at (0,0);
	\coordinate (c) at (1,0);
	
	\draw (-1/2,0) node[above]{\small{$5$}};
	\draw (1/2,0) node[above]{\small{$5$}};
	
	\draw (a)--(b);
	\draw (b)--(c);
	
	\fill[fill=black] (c) circle (2.5pt);
	\fill[fill=black] (a) circle (2.5pt);
	\fill[fill=black] (b) circle (2.5pt);
\end{tikzpicture}
is of negative type, so the extension is neither finite or affine.

\item The failure of faithfulness of the map $\mathbf{a}$ in the non-crystallographic case comes from the fact that in this case, we cannot relate the length function on the extension $\widehat{W}$ with separating reflecting hyperplanes in $V$.
\item A little bit more can be said about the subgroup $U:=\left<\sigma(W)(\widetilde{\alpha})\right>$ of $V$ in the dihedral case. Indeed, if $W=I_2(m)$ ($m=5$ or $m\ge7$), then $U$ is \emph{dense} in $V$. To see this, recall the notation from the previous proof and let $I_2(m)=\left<s,t\right>$ and assume that $r_W$ is conjugate to $s$, i.e. $\alpha_s\in\sigma(W)(\widetilde{\alpha})$. We have seen that $\alpha_s,c^2\alpha_s\in U$, where $c=2\cos(\pi/m)$. As $(c^2-2)\sigma(t)(\alpha_s)=\sigma(sts)(\alpha_s)-\alpha_s$, we also have $\sigma(t)(\alpha_s),c^2\sigma(t)(\alpha_s)\in U$, implying that $c\alpha_t,c^3\alpha_t\in U$. Because $\{\alpha_s,\alpha_t\}$ is a basis of $V\simeq\R^2$ and as $c^2\notin\Q$, the family $\{\alpha_s,c^2\alpha_s,c\alpha_t,c^3\alpha_t\}$ is free over $\Z$, generating a subgroup $A\simeq\Z^4$ of $U$. Now, since the subgroup $\left<1,c^2\right>$ is dense in $\R$, the subgroups $\left<\alpha_s,c^2\alpha_s\right>\le \R\alpha_s$ and $\left<c\alpha_t,c^3\alpha_t\right>\le\R\alpha_t$ are dense as well, so that $A\simeq \left<\alpha_s,c^2\alpha_s\right>\times\left<c\alpha_t,c^3\alpha_t\right>$ is dense in $\R\alpha_s\oplus\R\alpha_t=V$ and contained in $U$.\qed
\end{itemize}
\end{rem}

\subsection{A key property of the subgroup $Q$}
\hfill

\begin{lem}\label{Qinterpara=1}
The subgroup $Q$ trivially intersects any proper parabolic subgroup of $\widehat{W}$, i.e.
\[\forall I\subsetneq \widehat{S},~Q\cap\widehat{W}_I=1.\]
\end{lem}

\begin{proof}
First, if $W$ is a Weyl group, then by the Proposition \ref{Qisnotabelian}, the subgroup $Q$ is torsion-free, while every proper parabolic subgroup of $\widehat{W}\simeq W_{\rm a}$ is finite and the result follows. Assume now that $W$ is non-crystallographic and recall the projection $\pi : \widehat{W}\twoheadrightarrow W$. The statement may be rephrased as follows:
\[\forall s\in\widehat{S},~\ker\left(\widehat{W}_{\widehat{S}\setminus\{s\}} \stackrel{\tiny{\pi}}\longrightarrow W\right)=1.\]
For $s=\widehat{s}_0$, this is obvious since $\widehat{W}_{\widehat{S}\setminus\{\widehat{s}_0\}}\stackrel{\tiny{\pi}}\longrightarrow W$ is an isomorphism.

Let $s\in\widehat{S}\setminus\{\widehat{s}_0\}$. Since $\widehat{W}$ is compact hyperbolic, the parabolic subgroup $\widehat{W}_{\widehat{S}\setminus\{s\}}$ is finite and, to prove that the morphism $\widehat{W}_{\widehat{S}\setminus\{s\}}\stackrel{\tiny{\pi}}\longtwoheadrightarrow \pi\left(\widehat{W}_{\widehat{S}\setminus\{s\}}\right)$ is injective, it suffices to see that 
\begin{equation}\label{bydyer}
\left|\pi\left(\widehat{W}_{\widehat{S}\setminus\{s\}}\right)\right|=\left|\widehat{W}_{\widehat{S}\setminus\{s\}}\right|. \tag{$\star_s$}
\end{equation}
The right-hand side is easily computed using the Coxeter diagram of $\widehat{W}$ (see Table \ref{hypextdiags}). To compute the left-hand side, we proceed by a case-by-case analysis. Denote by $R:=\bigcup_wwSw^{-1}$ the set of reflections of $W$ and, for $w\in W$, let
\[N(w):=\{r\in R~;~\ell(rw)<\ell(w)\}.\]
If $H\le W$ is a reflection subgroup of $W$ (i.e. if $H=\left<H\cap R\right>$), then the set
\[D(H):=\{r\in R~;~N(r)\cap H=\{r\}\}\]
is a set of Coxeter generators of $H$ (\cite[Theorem 3.3]{dyer90}). We find the Coxeter generators $D(\pi(\widehat{W}_{\widehat{S}\setminus\{s\}}))$ and determine the resulting Coxeter diagram, giving the order of $\pi(\widehat{W}_{\widehat{S}\setminus\{s\}})$.
\begin{enumerate}[label=$\bullet$,leftmargin=*]
\item If $W=I_2(2k+1)$, we have $r_W=(s_1s_2)^ks_1$ and we readily compute $s_2={s_1}^{r_W}$ and $s_1={s_2}^{r_W}$ so that $\pi(\widehat{W}_{\widehat{s}_0,\widehat{s}_1})=\pi(\widehat{W}_{\widehat{s}_0,\widehat{s}_2})=W$. On the other hand, we get from the diagram $|\widehat{W}_{\widehat{s}_0,\widehat{s}_1}|=|\widehat{W}_{\widehat{s}_0,\widehat{s}_1}|=4k+2=|W|$. This proves (\ref{bydyer}) for $s=\widehat{s}_1,\widehat{s}_2$.
\item If $W=I_2(4k)$, then $r_W=(s_1s_2)^{2k-1}s_1$ and since $s_2=(s_1r_W)^{2k-1}s_1$, we also have $\left<r_W,s_1\right>=W$ and (\ref{bydyer}) is thus true for $s=\widehat{s}_2$ as $\widehat{W}_{\widehat{s}_0,\widehat{s}_1}\simeq W$. Because $s_2r_W=r_Ws_2$, we have $\left<s_2,r_W\right>=A_1\times A_1$ and $\widehat{W}_{\widehat{s}_0,\widehat{s}_2}\simeq A_1\times A_1$ so (\ref{bydyer}) also holds for $s=\widehat{s}_1$.
\item If $W=I_2(4k+2)$, then $r_W=(s_1s_2)^{2k}s_1$ and we compute $r_Ws_1r_W=(s_1s_2)^{4k}s_1=s_2s_1s_2=s_1^{s_2}$. In the same way, we get $(s_1(s_1^{s_2}))^ks_1=(s_1s_2)^{2k}s_1=r_W$. This implies $\left<s_1,r_W\right>=\left<s_1,s_1^{s_2}\right>\simeq I_2(2k+1)\simeq \widehat{W}_{\widehat{s}_0,\widehat{s}_1}$. In fact, we have $D(\left<s_1,r_W\right>)=\{s_1,s_1^{s_2}\}$. Now, as above we have $s_2r_W=r_Ws_2$ and $\widehat{W}_{\widehat{s}_0,\widehat{s}_2}\simeq A_1\times A_1\simeq\left<s_2,r_W\right>$.
\item For $W=H_3$, special relations among reflections occur, namely
\[r_W=s_3^{(s_2s_1)^2},~s_3=r_W^{(s_1s_2)^2},~s_2=s_3(r_Ws_3s_1)^2r_Ws_3,~s_1=(r_Ws_3s_2)^2r_Ws_3r_W.\]
Hence, for $s\in\widehat{S}$, we have $\pi\left(\widehat{W}_{\widehat{S}\setminus\{s\}}\right)=W\simeq\widehat{W}_{\widehat{S}\setminus\{s\}},$ this last isomorphism being given by the diagram of $\widehat{H_3}$. Therefore, all the relations (\ref{bydyer}) hold in this case.
\item Finally, for $W=H_4$, the additional reflection is
\[r_W=s_4^{(s_3s_2s_1s_2s_3(s_1s_2)^2s_3s_4)^2}.\]
We notice the following relation
\[s_1=s_2s_3(s_4r_W)^2(s_3s_4r_Ws_2(s_3s_4r_W)^2s_2)^3s_3s_4r_Ws_4s_3s_2.\]
This proves that $s_1\in\left<r_W,s_2,s_3,s_4\right>$ so $\pi(\widehat{W}_{\widehat{s}_0,\widehat{s}_2,\widehat{s}_3,\widehat{s}_4})=W\simeq\widehat{W}_{\widehat{s}_0,\widehat{s}_2,\widehat{s}_3,\widehat{s}_4}$. We treat the remaining cases by determining the Dyer generators; calculations can be done on the sixty reflections of $H_4$ (though easier using \cite{GAP4}). We obtain
\[\pi\left(\widehat{W}_{\widehat{s}_0,\widehat{s}_1,\widehat{s}_3,\widehat{s}_4}\right)=\left<r_W,s_1,s_3,s_4\right>=\left<s_1^{s_2s_3(s_1s_2)^2},s_1,s_3,s_4\right>\simeq A_1\times H_3\simeq\widehat{W}_{\widehat{s}_0,\widehat{s}_1,\widehat{s}_3,\widehat{s}_4},\]
\[\pi\left(\widehat{W}_{\widehat{s}_0,\widehat{s}_1,\widehat{s}_2,\widehat{s}_4}\right)=\left<r_W,s_1,s_2,s_4\right>=\left<s_3^{s_4s_2s_1s_2s_3(s_1s_2)^2s_3},s_1,s_2,s_4\right>\simeq I_2(5)^2\simeq\widehat{W}_{\widehat{s}_0,\widehat{s}_1,\widehat{s}_2,\widehat{s}_4}\]
and finally,
\[\pi\left(\widehat{W}_{\widehat{s}_0,\widehat{s}_1,\widehat{s}_2,\widehat{s}_3}\right)=\left<r_W,s_1,s_2,s_3\right>\simeq H_3\times A_1\simeq\widehat{W}_{\widehat{s}_0,\widehat{s}_1,\widehat{s}_2,\widehat{s}_3}.\]
This establishes the relations (\ref{bydyer}) for $W=H_4$, finishing the proof.
\end{enumerate}
\end{proof}

\begin{cor}\label{Qistorsionfree}
The group $Q$ is torsion-free.
\end{cor}

\begin{proof}
If $q\in Q$ has finite order, then by a theorem of Tits (\cite[\S V.4, Exercise 2.d]{bourbaki456}), there are $w\in\widehat{W}$ and $J\subsetneq\widehat{S}$ such that $q\in w\widehat{W}_Jw^{-1}$ and since $Q\unlhd\widehat{W}$, we get $q^w\in Q\cap\widehat{W}_J=1$.
\end{proof}

\subsection{The hyperbolic torus $\mathbf{T}(W)$ of $W$ and its first properties}
\hfill

Before defining the manifold $\mathbf{T}(W)$, we have to study the action of the subgroup $Q\unlhd\widehat{W}$ on the Tits cone of $\widehat{W}$. Recall some notation: define $\widehat{V}:=\mathrm{span}_\R(\alpha_s,~s\in\widehat{S})$ and the symmetric bilinear form $\widehat{B}$ by $\widehat{B}(\alpha_s,\alpha_t)=-\cos\left({\pi}/{\widehat{m}_{s,t}}\right)$, with $(\widehat{m}_{s,t})_{s,t\in\widehat{S}}$ the Coxeter matrix of $(\widehat{W},\widehat{S})$. If $\widehat{W}$ is hyperbolic, then the form $\widehat{B}$ has signature $(n-1,1)$, otherwise $\widehat{W}\simeq W_{\rm a}$ is affine and so $\widehat{B}$ has signature $(n-1,0)$. We also have the geometric representation $\widehat{\sigma} : \widehat{W} \longhookrightarrow O(\widehat{V},\widehat{B})$ and consider its contragredient representation 
\begin{equation}\label{contragredient}\tag{$\circledast$}
\widehat{\sigma}^* : \widehat{W}\longhookrightarrow GL(\widehat{V}^*)
\end{equation}
and define $(\alpha_s^\vee)_{s\in\widehat{S}}$ to be the basis of $\widehat{V}^*$, dual to $(\alpha_s)_{s\in\widehat{S}}$. We have $\widehat{\sigma}^*(w)={}^{\rm t\!}{(\widehat{\sigma}(w^{-1}))}$, i.e.
\[\forall s,t\in\widehat{S},~\widehat{\sigma}^*(s)(\alpha_t^\vee)=\alpha_t^\vee-2\delta_{s,t}\widehat{B}(-,\alpha_s).\]
The duality pairing of $\widehat{V}$ is denoted $\left<\cdot,\cdot\right>$ as usual. For $s\in\widehat{S}$, let moreover
\[H_s:=\{\lambda\in\widehat{V}^*~;~\left<\lambda,\alpha_s\right>=0\}~~\text{and}~~A_s:=\{\lambda\in\widehat{V}^*~;~\left<\lambda,\alpha_s\right>>0\}\]
and consider the respective \emph{fundamental chamber} and \emph{Tits cone}
\[C:=\{\lambda\in\widehat{V}^*~;~\left<\lambda,\alpha_{s}\right>>0,~\forall s\in\widehat{S}\}=\bigcap_{s\in\widehat{S}}A_s~~\text{and}~~X:=\bigcup_{w\in\widehat{W}}w(\overline{C}).\]
This is a convex cone and $\overline{C}$ is a fundamental domain for $\widehat{W}$ acting on $X$. For $I\subseteq\widehat{S}$ we let
\[C_I:=\left(\bigcap_{s\in I}H_s\right)\cap\left(\bigcap_{s\notin I}A_s\right)\subset\overline{C},\]
in particular $C_\emptyset=C$, $C_{\widehat{S}}=\{0\}$ and we have $\overline{C}=\bigsqcup_{I\subseteq\widehat{S}}C_I$. We have the \emph{Coxeter complex}
\[\widehat{\Sigma}:=\Sigma(\widehat{W},\widehat{S})=(X\setminus\{0\})/\R_+^*.\]
This is a $\widehat{W}$-pseudomanifold and we have a decomposition
\[\widehat{\Sigma}=\bigcup_{w\in\widehat{W},~I\subsetneq\widehat{S}}\R_+^*w(\overline{C_I})\]
which is in fact a $\widehat{W}$-triangulation since $\R_+^*w(\overline{C_I})$ may be identified with the standard $(n-|I|)$-simplex: $\R_+^*w(\overline{C_I})\simeq\Delta^{n-|I|}$. Moreover, since $\widehat{W}$ is infinite, $\widehat{\Sigma}$ is contractible and by \cite[III, \S 2, Corollary 3]{brown_buildings}, as every proper parabolic subgroup of $\widehat{W}$ is finite, the pseudomanifold $\widehat{\Sigma}$ is in fact a smooth $n$-manifold.

We can give a natural simplicial structure to the Coxeter complex (see \cite[Corollary 2.6]{babson-reiner}). Consider the set of parabolic cosets of $\widehat{W}$
\[P(\widehat{W},\widehat{S}):=\{w\widehat{W}_I~;~w\in\widehat{W},~I\subsetneq\widehat{S}\}.\]
We partially order this set as follows:
\[w\widehat{W}_I \preceq w'\widehat{W}_J~\stackrel{\tiny{\text{df}}}\Longleftrightarrow~w\widehat{W}_I\supseteq w'\widehat{W}_J.\]
Notice that $w\widehat{W}_I\preceq w'\widehat{W}_J$ implies $w\widehat{W}_I=w'\widehat{W}_I$ and $J\subset I$. We define the simplicial complex $\Delta(\widehat{W},\widehat{S})$ as the nerve of this poset:
\[\Delta(\widehat{W},\widehat{S}):=\mathcal{N}(P(\widehat{W},\widehat{S}),\preceq).\]

Let $F(\widehat{\Sigma})$ be the face lattice of $\widehat{\Sigma}$ with respect to the triangulation described above. Then we have an isomorphism of posets
\[\begin{array}{ccc}
(P(\widehat{W},\widehat{S}),\preceq) & \stackrel{\tiny{\sim}}\longrightarrow & (F(\widehat{\Sigma}),\subseteq) \\ w\widehat{W}_I & \longmapsto & \R_+^*w(\overline{C_I})\end{array}\]
and this yields a $\widehat{W}$-equivariant homeomorphism
\[|\Delta(\widehat{W},\widehat{S})|\stackrel{\tiny{\sim}}\longrightarrow \widehat{\Sigma}.\]

Now, recall that an action of a group $G$ on a space $Z$ is said to be \emph{properly discontinuous} (or a \emph{covering space action}, see \cite[\S 1.3]{hatcher}) if every point $z\in Z$ has an open neighbourhood $z\in U\subset Z$ such that if $g\in G$ is such that $gU\cap U\ne\emptyset$, then $g=1$, i.e. such that
\[O_G(U):=\{g\in G~;~g(U)\cap U\ne\emptyset\}=\{1\}.\]

\begin{lem}\label{Qactsasitshould}
Recall from (\ref{contragredient}) the representation $\widehat{\sigma}^*$. The action of the discrete subgroup $\widehat{\sigma}^*(Q)\le GL(\widehat{V}^*)$ on the Coxeter complex $\widehat{\Sigma}$ is free and properly discontinuous.
\end{lem}

\begin{proof}
Of course, we identify the group $\widehat{W}$ with $\widehat{\sigma}^*(\widehat{W})$. Let $\overline{x}\in\widehat{\Sigma}$ (with $x\in X\setminus\{0\}$). First, we prove that $q(\overline{x})\ne \overline{x}$ for $q\in Q\setminus\{1\}$. To say that $q(\overline{x})=\overline{x}$ amounts to say that $q(x)=ax$ for some $a\in\R_+^*$ and we may assume that $x\in\overline{C}\setminus\{0\}$ since $Q\unlhd\widehat{W}$. There is some $I\subsetneq\widehat{S}$ such that $x\in C_I$ and because $C_I$ is a cone, we have $ax\in C_I\cap q(C_I)\ne\emptyset$ and by \cite[V, \S 4, Proposition 5]{bourbaki456}, we obtain $q(C_I)=C_I$ so $q\in \widehat{W}_I\cap Q=1$ by Lemma \ref{Qinterpara=1}.

To prove that the action is properly discontinuous at $\overline{x}$, we have to find an open neighbourhood $U$ of $\overline{x}$ in $\widehat{\Sigma}$ such that $O_Q(U)=\{1\}$. By definition of the topology on the Coxeter complex, it suffices to prove the statement for $X\setminus\{0\}$. 

First, we show that the action of $\widehat{W}$ is \emph{wandering} at $x$, that is, we can find an open neighbourhood $A$ of $x$ such that $O_{\widehat{W}}(A)$ is finite. We may assume that $x\in\overline{C}\setminus\{0\}$, say $x\in C_I$ with $I\subsetneq\widehat{S}$. Define $A$ to be the interior in $X\setminus\{0\}$ of the subset $\bigcup_{v\in \widehat{W}_I}v(\overline{C})$. We prove that there are only finitely many $w\in\widehat{W}$ such that $A\cap w(A)\ne\emptyset$. Suppose that $w\in O_{\widehat{W}}(A)$ and choose $a\in A$ with $w(a)\in A$. Notice that we have
\[A\subseteq \bigcup_{u\in\widehat{W}_I}u(C)\cup\bigcup_{\substack{v\in\widehat{W}_I \\ s\in I}}v(H_s\cap\partial\overline{C}).\]
Thus, we distinguish four cases:
\begin{enumerate}[label=$\bullet$]
\item As $\widehat{W}$ acts on $X\setminus\bigcup_{s\in\widehat{S}}H_s$, we cannot have $a\in\bigcup_{v,s}v(H_s\cap\partial\overline{C})$ and $w(a)\in\bigcup_u u(C)$.
\item Similarly, we cannot have $a\in\bigcup_u u(C)$ and $w(a)\in\bigcup_{v,s}v(H_s\cap\partial\overline{C})$.
\item Suppose that $a\in\bigcup_v v(C)$ and $w(a)\in\bigcup_v v(C)$, say $a\in u(C)$ and $w(a)\in v(C)$. This implies $u^{-1}(a)\in C$ and $v^{-1}w(a)=v^{-1}wu(u^{-1}(a))\in C$, thus $uv^{-1}w(C)\cap C\ne\emptyset$ and so $w=vu^{-1}\in\widehat{W}_I$ by Tits' lemma.
\item Suppose now that we have $a\in\bigcup_{v,s}v(H_s\cap\partial\overline{C})$ and $w(a)\in\bigcup_{v,s}v(H_s\cap\partial\overline{C})$, say $a\in u(H_s\cap\partial\overline{C})$ and $w(a)\in v(H_t\cap\partial\overline{C})$. This implies $u^{-1}(a)\in\overline{C}$ and $v^{-1}wu(u^{-1}(a))=v^{-1}w(a)\in\overline{C}$ and by \cite[V, \S 4, Proposition 6]{bourbaki456} we get $v^{-1}w(a)=u^{-1}(a)$ and thus $uv^{-1}w\in (\widehat{W})_a=u\widehat{W}_Ju^{-1}$ for some $J\subsetneq\widehat{S}$ (in fact, $J$ is defined by the condition $\widehat{W}_J=(\widehat{W})_{u^{-1}(a)}$). Therefore, we have $w\in v\widehat{W}_Ju^{-1}$.
\end{enumerate}
In any case, we have
\[O_{\widehat{W}}(A)\stackrel{\tiny{\text{df}}}=\{w\in\widehat{W}~;~w(A)\cap A\ne\emptyset\}\subset\bigcup_{\substack{u,v\in\widehat{W}_I \\ J\subsetneq\widehat{S}}}u\widehat{W}_Jv.\]
However, as $\widehat{W}$ is compact, any proper parabolic subgroup is finite and so this last subset is finite and $O_{\widehat{W}}(A)$ is then finite as well.

The rest of the proof is very standard. For each $w\in O_{\widehat{W}}(A)\setminus\widehat{W}_I$ we have $w(x)\ne x$ and we may choose an open subset $A_w$ such that $x\in A_w\subset A$ and $w(A_w)\cap A_w=\emptyset$ and define
\[B:=\bigcap_{w\in O_{\widehat{W}}(A)\setminus\widehat{W}_I}A_w \subset A.\]
Because $O_{\widehat{W}}(A)$ is finite, $B$ is open and let $w'\in O_{\widehat{W}}(B)\subset O_{\widehat{W}}(A)$. We must have $w'\in\widehat{W}_I$ because otherwise, $\emptyset\ne B\cap w'(B)\subset A_{w'}\cap w'(A_{w'})=\emptyset$. Consider the open subset
\[U:=\bigcap_{w\in\widehat{W}_I}w(B)\subset B.\]
We have $O_{\widehat{W}}(U)\subset O_{\widehat{W}}(B)\subset\widehat{W}_I$ and $U$ is $\widehat{W}_I$-stable (i.e. $U$ is a $\widehat{W}$-\emph{slice} at $x$). In particular, if $q\in Q\setminus\{1\}$, then $q\notin\widehat{W}_I$ by Lemma \ref{Qinterpara=1} and thus $q\notin O_{\widehat{W}}(U)$.
\end{proof}

We arrive then to the main result of this section. Remark that the Tits form $\widehat{B}$ induces a Riemannian metric on the Coxeter complex $\widehat{\Sigma}$.

\begin{theo}\label{defsteinbergtorus}
Let $(W,S)$ be a finite irreducible Coxeter group of rank $n$ and $(\widehat{W},\widehat{S})$ be either the affine extension corresponding to a fixed root system $\Phi$ for $W$ if $W$ is crystallographic, or the extension constructed above otherwise, with $Q:=\ker(\widehat{W}\twoheadrightarrow W)$. If $\widehat{\sigma}^*$ denotes the contragredient geometric representation (as in (\ref{contragredient})), then the orbit space
\[\mathbf{T}(W):=\widehat{\Sigma}/\widehat{\sigma}^*(Q)\]
is a closed, connected, orientable, compact smooth $W$-manifold of dimension $n$.

If $W$ is a Weyl group, then we have a diffeomorphism $\widehat{\Sigma}\simeq\R^n$ and the manifold $\mathbf{T}(W)$ is $W$-diffeomorphic to a maximal torus of the simply-connected compact Lie group with root system $\Phi$. Otherwise, the Riemannian manifold $\widehat{\Sigma}$ is isometric to the hyperbolic $n$-space $\mathbb{H}^n$ and $\mathbf{T}(W)\simeq\mathbb{H}^n/Q$ is a hyperbolic $W$-manifold.

Furthermore, the canonical projection yields a covering space
\[Q\longhookrightarrow \widehat{\Sigma} \longtwoheadrightarrow \mathbf{T}(W)\]
and the quotient simplicial complex $\Delta(\widehat{W},\widehat{S})/Q$ is a regular $W$-triangulation of $\mathbf{T}(W)$.
\end{theo}

\begin{proof}
Since $\widehat{\Sigma}$ is a closed smooth manifold and the action $\widehat{\sigma}^*(Q)~\actleft~\widehat{\Sigma}$ is properly discontinous by Lemma \ref{Qactsasitshould}, the quotient manifold theorem ensures that $\mathbf{T}(W)$ is indeed a closed smooth manifold and by \cite[Proposition 1.40]{hatcher}, the projection $\widehat{\Sigma}\twoheadrightarrow\mathbf{T}(W)$ is a covering map. Moreover, $\mathbf{T}(W)$ is connected since the Coxeter complex is and, as $(\overline{C}\setminus\{0\})/\R_+^*\simeq\overline{C}\cap\Sph^n$ is a $\widehat{W}$-fundamental domain on the Coxeter complex, its projection onto $\mathbf{T}(W)$ is a $W$-fundamental domain, hence $\mathbf{T}(W)$ is compact ($W$ is finite). Since $Q$ is normally generated by $\widehat{s}_0\widehat{r_W}$ and because $\ell(\widehat{r_W})$ is odd, we have $\varepsilon(\widehat{s}_0\widehat{r_W})=1$ and so $Q\le\ker(\varepsilon)$. This proves that the action of $Q$ on $\widehat{\Sigma}$ preserves the orientation, ensuring the orientability of $\mathbf{T}(W)$.

Assume now that $W$ is a Weyl group. Then $\widehat{W}=W_{\rm a}$ is affine and the Tits form $\widehat{B}$ has signature $(n,0)$. By \cite[V, \S 4.9, Lemma 2]{bourbaki456}, the radical of $\widehat{B}$ is generated by a vector $u=\sum_{s\in\widehat{S}}u_s\alpha_s$ with $u_s>0$ for every $s$. Therefore, if $\lambda\in X\setminus\{0\}$ is written as $\lambda=\sum_s\lambda_s\alpha_s^\vee$ with $\lambda_s\ge0$, then we have $\lambda(u)=\sum_s\lambda_su_s>0$ since at least one of the $\lambda_s$'s is positive. The map $\lambda\mapsto \lambda/\lambda(u)$ thus induces a $\widehat{W}$-equivariant diffeomorphism
\[\widehat{\Sigma}\stackrel{\tiny{\sim}}\longto X\cap\mathcal{E},\]
where $\mathcal{E}:=\{\lambda\in\widehat{V}^*~;~\lambda(u)=1\}$. Moreover, by \cite[V, \S 5.9, Proposition 10]{bourbaki456}, we also have $X\cap\mathcal{E}=\mathcal{E}$. As explained in \cite[\S 6.5]{humphreys-reflectiongroups}, the space $\mathcal{E}$ identifies with the $n$-dimensional Euclidean space and the restricted action $\widehat{W}~\actleft~\mathcal{E}$ identifies with the standard action of the affine Weyl group $W_{\rm a}=\Z\Phi^\vee\rtimes W$ on $\Z\Phi^\vee\otimes\R$, where $\Phi$ is the root system of $W$. Therefore, we obtain a $W$-equivariant diffeomorphism
\[\mathbf{T}(W)=\widehat{\Sigma}/Q\simeq\mathcal{E}/Q\simeq(\Z\Phi^\vee\otimes\R)/\Z\Phi^\vee\]
and this last manifold indeed is $W$-diffeomorphic to a maximal torus in the simply-connected Lie group with root system $\Phi$.

In the non-crystallographic case, let $v^*\in \widehat{V}^*$ be a normalized eigenvector for the negative eigenvalue of $\widehat{B}$. Then the subset $\mathcal{H}:=\{\lambda\in \widehat{V}^*~;~\widehat{B}(\lambda,\lambda)=-1,~\widehat{B}(v^*,\lambda)<0\}$, together with the metric induced by the restriction of $\widehat{B}$ is a Riemannian manifold isometric to the hyperbolic space $\mathbb{H}^n$. We have $\mathbf{T}(W)=\widehat{\Sigma}/Q\simeq\mathcal{H}/Q\simeq\mathbb{H}^n/Q$ and since $Q$ preserves the form $\widehat{B}$, the manifold $\mathbf{T}(W)$ naturally inherits a hyperbolic Riemannian metric.
\end{proof}

\begin{rem}\label{alreadyseeninDavis}
After we did this work, we realized that the manifolds $\mathbf{T}(H_3)$ and $\mathbf{T}(H_4)$ have already been discovered in \cite{zimmermann_hyperbolic} and \cite{davis_hyperbolic}. Zimmermann and Davis construct them by taking the orbit under the action of $Q$ of hyperbolic polyhedra. However, our approach has the advantages of being more systematic and to work with any finite Coxeter group. The Zimmermann manifold $\mathbf{T}(H_3)$ has the particularity of being \emph{maximally symmetric} among hyperbolic 3-manifolds with Heegaard genus 11, in the sense of \cite{zimmermann_maximally_symmetric}. On the other hand, the Davis manifold $\mathbf{T}(H_4)$ has a \emph{spin structure} (equivalently, its second Stiefel-Whitney class $w_2$ vanishes) and is, to the knowledge of the author, the only closed 4-manifold for which the intersection form has been explicitly determined, see \cite{ratcliffe-tschantz} and \cite{martelli_hyperbolic}.
\end{rem}

Recall that, as $\widehat{W}$ is infinite, the Coxeter complex is contractible.

\begin{cor}\label{T(W)isBQ}
The covering map
\[Q\longhookrightarrow \widehat{\Sigma}\longtwoheadrightarrow\mathbf{T}(W)\]
is a universal principal $Q$-bundle. In particular, $\mathbf{T}(W)$ is a classifying space for $Q$:
\[\mathbf{T}(W)\simeq B_Q\simeq K(Q,1).\]
\end{cor}

\subsection{Presentation on the fundamental group of $\mathbf{T}(W)$}
\hfill

In this section, we use Poincar\'e's fundamental polyhedron theorem (\cite[Theorem 11.2.2]{ratcliffe_hyperbolic}) to derive a presentation of the group $\pi_1(\mathbf{T}(W))\simeq Q\unlhd\widehat{W}$, in the non-crystallographic case.

The tessellation $\Delta(\widehat{W},\widehat{S})$ of $\widehat{\Sigma}\simeq\mathbb{H}^n$ is \textit{exact} (meaning that each side of the fundamental polyhedron is the side of exactly two polyhedra of the tessellation, see \cite[\S 6.8]{ratcliffe_hyperbolic}) and yields a fundamental polyhedron for $Q$ acting on $\widehat{\Sigma}$. Indeed, choose $v^*\in V^*$ a normalized eigenvector of the Tits form $\widehat{B}$ for its unique negative eigenvalue and consider the subset
\[\mathcal{H}:=\{\lambda\in V^*~;~\widehat{B}(\lambda,\lambda)=-1,~\widehat{B}(v^*,\lambda)<0\}\subset V^*.\]
As already noted in the proof of Theorem \ref{defsteinbergtorus}, the form $\widehat{B}$ induces a Riemannian metric on $\mathcal{H}$ and we have an isometry $\mathcal{H}\simeq\mathbb{H}^n$, where $\mathbb{H}^n$ is the standard hyperbolic $n$-space. The fundamental chamber $C$ is included in the subset $\{\widehat{B}(\lambda,\lambda)<0\}$, hence we can project the punctured Tits cone $X\setminus\{0\}$ on the sheet $\mathcal{H}$ of the hyperbola $\{\widehat{B}(\lambda,\lambda)=-1\}$ and we get $\widehat{\Sigma}\simeq X\cap\mathcal{H}$. Consider the $n$-simplex
\[\Delta_0:=(\overline{C}\setminus\{0\})/\R_+^*\simeq\overline{C}\cap\mathcal{H}\subset\widehat{\Sigma}.\]
This is a fundamental polyhedron for $\widehat{W}$ acting on $\widehat{\Sigma}$.

Recall that we have denoted $H_s:=\{\left<\lambda,\alpha_s\right>=0\}$ for $s\in\widehat{S}$. As the subset $L_0:=\overline{C}\cap\bigcap_{s\ne\widehat{s}_0}H_s$ is a line, its intersection with $\mathcal{H}$ is a vertex of the tessellation $\Delta(\widehat{W},\widehat{S})$ and we may consider its star
\[\Delta:=\mathrm{St}\left(L_0\cap\mathcal{H}\right)\stackrel{\tiny{\text{df}}}=\bigcup_{\substack{\sigma\in F_n(\Delta(\widehat{W},\widehat{S})) \\ L_0\cap\mathcal{H}\subset\sigma}}\sigma=\bigcup_{w\in W}w(\Delta_0).\]
This is the required fundamental polyhedron for $Q$ acting on $\widehat{\Sigma}$.

We will describe generators and relations for $\pi_1(\mathbf{T}(W))$ in terms of side-pairing and cycle relations, as in \cite[\S 6.8]{ratcliffe_hyperbolic}. We see that the facets of $\Delta$ are the $W$-translates of the facet
\[\sigma_0:=H_{\widehat{s}_0}\cap\Delta\in F_{n-1}(\Delta),\]
in other words, $\partial\Delta=\bigcup_{w}w(\sigma_0)$. By \cite[Theorem 6.8.3]{ratcliffe_hyperbolic}, the group $Q=\pi_1(\mathbf{T}(W))$ is generated by the set
\[\Psi:=\{q\in Q~;~\Delta\cap q\Delta\in F_{n-1}(\Delta)\}.\]

\begin{lem}\label{gensforpi1}
The set $\Psi$ of generators of $Q$ is given by the $W$-conjugates of the normal generator of $Q$; i.e., if $q_0:=\widehat{s}_0r_W\in \widehat{W}$ then we have
\[\Psi=\{{}^{w\!}q_0,~w\in W\}=\{wq_0w^{-1},~\overline{w}\in W/C_W(q_0)\}.\]
\end{lem}

\begin{proof}
Let $q\in Q\setminus\{1\}$ such that $\Delta\cap q\Delta$ is a facet $w(\sigma_0)$ of $\Delta$, for some $w\in W$. We have
\[w(\sigma_0)=\Delta\cap q\Delta=\bigcup_{u,v\in W}u(\Delta_0)\cap qv(\Delta_0)=\bigcup_{u,v\in W}u(\Delta_0\cap u^{-1}qv(\Delta_0)).\]
Since any term of the last union is (empty or) a closed simplex, this means that one of them has to be the whole of $w(\sigma_0)$, so we can find $u,v\in W$ such that
\[u^{-1}w(\sigma_0)=\Delta_0\cap u^{-1}qv(\Delta_0).\]
In particular, we have $u^{-1}w(\sigma_0)\subset \Delta_0$ and since any $\widehat{W}$-orbit meets $\Delta_0$ in only one point, this implies that $u^{-1}w(\sigma_0)=\sigma_0$ and so $u^{-1}w\in \widehat{W}_{\sigma_0}=\left<\widehat{s}_0\right>$ but as $u^{-1}w\in W$, this is possible only when $u=w$. Hence we get
\[\sigma_0=\Delta_0\cap u^{-1}qv(\Delta_0).\]
This implies in turn that $u^{-1}qv\in \left<\widehat{s}_0\right>$ and since $q\ne1$, we must have $u^{-1}qv=\widehat{s}_0$, i.e. $q=u\widehat{s}_0v^{-1}$. Finally, because $q\in Q$, applying the projection $\pi : \widehat{W}\twoheadrightarrow W$ to this equality yields $1=ur_Wv^{-1}$, so $v=ur_W$ and $q=u\widehat{s}_0v^{-1}=uq_0u^{-1}$.
\end{proof}

Now, we need a technical lemma on the centralizer of $q_0$.

\begin{lem}\label{centralizer}
The centralizer of $q_0=\widehat{s}_0r_W$ in $W$ is given by
\[C_W(q_0)=C_W(\widehat{s}_0)=\left<s\in S~;~s\widehat{s}_0=\widehat{s}_0s\right>.\]
In particular, this is (standard) parabolic.
\end{lem}

\begin{proof}
First, we borrow an argument due to Sebastian Schoennenbeck\footnote{\tiny{\url{https://mathoverflow.net/questions/200433/centralizers-of-reflections-in-special-subgroups-of-coxeter-groups}}} to prove the second equality above. Let $w=s_{i_1}\cdots s_{i_r}$ be a reduced expression of an element $w\in C_W(\widehat{s}_0)$. To show that $w$ is in the parabolic subgroup of the statement, since the elements of $C_W(\widehat{s}_0)$ of length 1 are the simple reflections of $W$ commuting with $\widehat{s}_0$, by induction it is enough to show that $\widehat{s}_0s_{i_r}=s_{i_r}\widehat{s}_0$. We have $\ell(w\widehat{s}_0)=\ell(w)+1$ and $\ell(w\widehat{s}_0w^{-1})=\ell(\widehat{s}_0)=1$, so $\ell(w\widehat{s}_0s_{i_r})=\ell(w\widehat{s}_0w^{-1}ws_{i_r})\le1+\ell(ws_{i_r})=\ell(w)$ and thus $\ell(w\widehat{s}_0s_{i_r})=\ell(w)$. Thus, by the exchange condition, there is a reduced expression $w\widehat{s}_0=s_{j_1}\cdots s_{j_r}s_{i_r}$ for $w\widehat{s}_0$ and since $s_{i_1}\cdots s_{i_r}\widehat{s}_0$ is already a reduced expression, by Matsumoto's lemma, there is a finite series of braid-moves from the second to the first. The expression $s_{i_1}\cdots s_{i_r}\widehat{s}_0$ satisfies the property
\begin{equation}\label{redexp}\tag{$\ast$}
\begin{array}{ll}
\text{The expression contains only one occurrence of $\widehat{s}_0$ and there is no simple reflection} \\ 
\text{appearing on the right of $\widehat{s}_0$ that does not commute with it}.\end{array}
\end{equation}
Consider a braid relation $sts\cdots=tst\cdots$ connecting the two expressions of $w\widehat{s}_0$, with $m$ factors on each side and suppose that we apply it to a reduced expression of $w\widehat{s}_0$ verifying (\ref{redexp}). If $s,t\ne\widehat{s}_0$, then the resulting expression still satisfies (\ref{redexp}). Now, if $s=\widehat{s}_0$ say, then $t$ has to commute with $\widehat{s}_0$. Indeed, if not, then the left-hand side of the braid relation contains at least two occurrences of $\widehat{s}_0$ (one on each side of $t$) and, in the right-hand side there is at least one occurrence of $t$ on the right of $\widehat{s}_0$, but none of these occur in the considered reduced expression. Therefore, the reduced expression resulting from the application of the braid move still verifies (\ref{redexp}). In particular, the expression $s_{j_1}\cdots s_{j_r}s_{i_r}$ satisfies (\ref{redexp}) and thus, every simple reflection appearing on the right of $\widehat{s}_0$ must commute with it. In particular, this is the case of $s_{i_r}$, as required.

We now prove that $C_W(q_0)=C_W(\widehat{s}_0)$. Let $w=s_{i_1}\cdots s_{i_r}$ be a reduced expression of an element $w\in C_W(q_0)$. Since $wq_0=q_0w$, we get $\widehat{s}_0w\widehat{s}_0=r_Wwr_W\in W$. Let $\widehat{s}_0w\widehat{s}_0=s_{j_1}\cdots s_{j_k}$ be a reduced expression in $W$. Since $\ell(w\widehat{s}_0)=\ell(w)+1=\ell(\widehat{s}_0w)$, we have $\ell(\widehat{s}_0w\widehat{s}_0)\in\{\ell(w),\ell(w)+2\}$. But taking length in the equality $\widehat{s}_0s_{i_1}\cdots s_{i_r}=s_{j_1}\cdots s_{j_k}\widehat{s}_0$ gives $k=r$, that is $\ell(\widehat{s}_0w\widehat{s}_0)=\ell(w)$. In particular, $\ell(\widehat{s}_0w\widehat{s}_0)<\ell(w\widehat{s}_0)$ and by the exchange condition, there is a reduced expression $\widehat{s}_0w\widehat{s}_0=s_{i_1}\cdots \widehat{s_{i_l}}\cdots s_{i_r}\widehat{s}_0$ (the reflection $s_{i_l}$ is omitted) and since this last expression is in $W$, we must have $s_{i_l}=\widehat{s}_0$, thus $\widehat{s}_0w\widehat{s}_0=s_{i_1}\cdots s_{i_r}=w$ and $w\in C_W(\widehat{s}_0)$. The reverse inclusion can be directly checked case by case using the parabolic description of $C_W(\widehat{s}_0)$.
\end{proof}

\begin{rem}
From the diagrams of the hyperbolic extensions we get therefore
\[C_{I_2(2g+1)}(q_0)=1,~C_{I_2(4g+2)}(q_0)=\left<s_2\right>,~C_{I_2(4g)}(q_0)=\left<s_2\right>,~C_{H_3}(q_0)=\left<s_2\right>,~C_{H_4}(q_0)=\left<s_1,s_2,s_3\right>\simeq H_3.\]
\end{rem}

\begin{theo}\label{presentationQ}
Let $W$ be non-crystallographic and $U:=\{w\in W~;~\ell(ws)>\ell(w),~\forall s\in S~;~s\widehat{s_0}=\widehat{s}_0s\}\approx W/C_W(q_0)$ be the set of minimal length coset representatives modulo the parabolic subgroup $C_W(q_0)$ of $W$. The transitive action of $W$ on $W/C_W(q_0)$ induces an action of $W$ on $U$ and for $(w,u)\in W\times U$, we write $w(u)$ for the image of $u$ under the action of $w$. Then the fundamental group $\pi_1(\mathbf{T}(W))\simeq Q$ admits the following presentation
\[\pi_1(\mathbf{T}(W))=\left<q_u,~u\in U~|~R_{\text{side}}\cup R_{\text{cycle}}\right>,\]
where 
\[R_{\text{side}}=\{q_uq_v,~u,v\in U~;~u^{-1}vr_W\in C_W(\widehat{s}_0)\}\]
and
\[\begin{array}{c}
R_{\text{cycle}}=\bigg\{q_{w(u_1)}q_{w(u_2)}\cdots q_{w(u_r)},~w\in W,~u_0,u_1,\dotsc,u_r,u_{r+1}\in U~\text{such that}~u_0=u_{r+1}=1\\
\text{and, for}~i>0,~\left.\left<\widehat{s}_0,\widehat{s}_0^{u_{i+1}^{-1}u_ir_W}\right>~\text{and}~\left<\widehat{s}_0,\widehat{s}_0^{r_Wu_{i-1}^{-1}u_i}\right>~\text{are conjugate under}~C_W(\widehat{s}_0)\right\}.\end{array}\]
\end{theo}

\begin{proof}
Drop the presentation notation and, for $u\in U$, denote $q_u:={}^{u\!}q_0=uq_0u^{-1}$, $\sigma_u:=u(\sigma_0)=\Delta\cap q_u(\Delta)$  and $\sigma_v:=q_u^{-1}(\sigma_u)=uq_0^{-1}(\sigma_0)$. To say that for some $u,v\in U$ we have $q_uq_{v}=1$ amounts to say that ${}^{u\!}q_0={}^{v\!}q_0^{-1}={}^{vr_W\!}q_0$, i.e. $u^{-1}vr_W\in C_W(q_0)$.

For the cycle relations, we follow the method detailed in \cite[\S 6.8]{ratcliffe_hyperbolic}. First notice that each facet of $\sigma_0$ is of the form $\sigma_u$ for some $u\in U$ (see \cite[Theorem 6.7.5]{ratcliffe_hyperbolic}). Choose $\sigma\in F_{n-1}(\Delta)$ and $\tau\in F_{n-2}(\sigma)\subset F_{n-2}(\Delta)$. Recursively define a sequence of facets $\{\sigma_{u_j}\}_{j\in\N^*}$ as follows
\begin{enumerate}[label=$\bullet$]
\item let $\sigma_{u_1}:=\sigma$,
\item let $\sigma_{u_2}$ be the facet of $\sigma$ adjacent to $\sigma_{u_1}':=q_{u_1}^{-1}(\sigma)$ such that $q_{u_1}(\sigma_{u_1}'\cap\sigma_{u_2})=\tau$,
\item for $i>1$, let $\sigma_{u_{i+1}}\in F_{n-2}(\sigma)$ be the facet adjacent to $\sigma_{u_i}':=q_{u_i}^{-1}(\sigma_{u_i})$ such that $q_{u_i}(\sigma_{u_i}'\cap\sigma_{u_{i+1}})=\sigma_{u_{i-1}}'\cap\sigma_{u_i}$.
\end{enumerate}
By \cite[Theorem 6.8.7]{ratcliffe_hyperbolic}, there exists a least integer $k\in\N^*$ such that $\sigma_{u_{i+k}}=\sigma_{u_i}$ for all $i$ and we have $q_{u_1}\cdots q_{u_k}=1$. Moreover, by the Poincar\'e theorem \cite[Theorem 11.2.2]{ratcliffe_hyperbolic}, the set of all such relations (for $\sigma\in F_{n-1}(\Delta)$ and $\tau\in F_{n-2}(\sigma)$), together with the side-pairing relations described above, form a complete set of relations for $Q$.

Choose $\sigma\in F_{n-1}(\Delta)\subset W\cdot \sigma_0$ and $\tau\in F_{n-2}(\sigma)\subset F_{n-2}(\Delta)$ and let $\{\sigma_{u_j}\}_{j\in\N^*}$ denote the associated cycle of sides, with period $\ell$, say. We have the relation $q_{u_1}\cdots q_{u_\ell}=1$. Up to conjugation by an element of $W$, we may assume that $\sigma=\sigma_0$ and so $q_{u_1}=q_0$. Let $i>1$ be such that we have some relation
\[q_{u_i}(\sigma_{u_i}'\cap\sigma_{u_{i+1}})=\sigma_{u_{i-1}}'\cap\sigma_{u_i}\ne\emptyset.\]
We write
\begin{align*}
q_{u_i}(\sigma_{u_i}'\cap\sigma_{u_{i+1}})=\sigma_{u_i}\cap\sigma_{u_{i-1}}'&~\Longleftrightarrow~\sigma_{u_i}\cap q_{u_i}(\sigma_{u_{i+1}})=\sigma_{u_i}\cap q_{u_{i-1}}^{-1}(\sigma_{u_{i-1}}) \\
&~\Longleftrightarrow~u_i(\sigma_0)\cap q_{u_i}u_{i+1}(\sigma_0)=u_i(\sigma_0)\cap q_{u_{i-1}}^{-1}u_{i-1}(\sigma_0) \\
&~\Longleftrightarrow~u_i(\sigma_0\cap u_i^{-1}q_{u_i}u_{i+1}(\sigma_0))=u_i(\sigma_0\cap u_i^{-1}u_{i-1}q_0^{-1}(\sigma_0)) \\
&~\Longleftrightarrow~\sigma_0\cap q_0u_i^{-1}u_{i+1}(\sigma_0)=\sigma_0\cap u_i^{-1}u_{i-1}r_W(\sigma_0),
\end{align*}
and the two sides of the last equality are simplices of the tessellation $\Delta(\widehat{W},\widehat{S})$, whose face lattice is the lattice of standard parabolic subgroups of $\widehat{W}$, so these two coincide if and only if their stabilizers in $\widehat{W}$ are equal. Though this condition depends on the choice of elements of $U$, it is easy to check that different choices give conjugate stabilizers in $C_W(q_0)$.
\end{proof}

\begin{cor}\label{abelianizationofQ}
The group $\pi_1(\mathbf{T}(H_3))$ (resp. $\mathbf{T}(H_4)$) admits a presentation with 11 (resp. 24) generators , all of whose relations are products of commutators. In particular, we have
\[H_1(\mathbf{T}(H_3),\Z)=\pi_1(\mathbf{T}(H_3))^\mathrm{ab}\simeq\Z^{11}~~\text{and}~~H_1(\mathbf{T}(H_4),\Z)=\pi_1(\mathbf{T}(H_4))^\mathrm{ab}\simeq\Z^{24}.\]
\end{cor}

\noindent\textit{Sketch of proof.}
For $H_3$, beside the side-pairing relations (which we can immediately simplify by removing half of the $[H_3:C_{H_3}(q_0)]=[H_3:\left<s_2\right>]=60$ generators), we find only one primitive cycle relation (\emph{primitive} meaning starting by $q_0$) of length 3 and one of length 5. Taking the $H_3$-conjugates of these gives 120 relations of length 3 and 120 relations of length 5. But some of these $H_3$-conjugate relations are inverses of other relations, so we can simplify them. We can also remove any cyclic permutation of these relations, which finally yields a presentation for $\pi_1(\mathbf{T}(H_3))$ with 30 generators, 20 relations of length 3 and 12 relations of length 5. Using the relations, we can check that some of the generators are superfluous and that the simplified presentation has the stated number of generators (all among the original generators) and that the relations become trivial, once abelianized.

For $H_4$, there is only one primitive cycle relation of length 5, which gives a presentation for $\pi_1(\mathbf{T}(H_4))$ with $\tfrac{1}{2}[H_4:C_{H_3}(q_0)]=60$ generators and 144 relations of length 5.
\qed

\begin{rem}
The presentation of $\pi_1(\mathbf{T}(H_3))$ with $30$ generators and $32$ relations we found in the sketch of proof above, and the one of $\pi_1(\mathbf{T}(H_4))$ with $60$ generators and $144$ relations, are precisely (up to relabelling) the presentations given in \cite{zimmermann_hyperbolic} and \cite{ratcliffe-tschantz}.
\end{rem}

\subsection{The manifolds $\mathbf{T}(I_2(m))$ as Riemann surfaces}
\hfill

A little bit more can be said about the case of the surfaces $\mathbf{T}(I_2(m))$. Recall that by a theorem of Gauss (see \cite[Theorem 3.11.1]{jost-riemann}), any Riemannian metric on an oriented 2-manifold $M$ induces a complex structure on $M$ (making $M$ a Riemann surface), called the \emph{conformal structure} induced by the metric.

\begin{cor}\label{T(I2(m))hyperbolicsurface}
For any $g\in\N^*$ the surfaces $\mathbf{T}(I_2(2g+1))$, $\mathbf{T}(I_2(4g))$ or $\mathbf{T}(I_2(4g+2))$ are closed compact Riemann surfaces of genus $g$. In particular, we have homeomorphisms
\[\mathbf{T}(I_2(2g+1))\simeq\mathbf{T}(I_2(4g))\simeq\mathbf{T}(I_2(4g+2)).\]
\end{cor}

\begin{proof}
Since the surfaces are orientable, the Riemannian metric induced by the one on the Coxeter complex induces a conformal structure on them. To obtain the genus, we only have to compute the Euler characteristic. Let $m$ be either $2g+1$, $4g$ or $4g+2$ and
\[W:=I_2(m)=\left<s,t~|~s^2=t^2=(st)^m=1\right>.\]
The rational chain complex associated to the simplicial complex $\Delta(\widehat{W},\widehat{S})$ has the shape
\[\xymatrix{\Q[\widehat{W}] \ar[r] & \Q[\widehat{W}/\left<s\right>]\oplus\Q[\widehat{W}/\left<t\right>]\oplus\Q[\widehat{W}/\left<\widehat{s}_0\right>] \ar[r] & \Q[\widehat{W}/\left<s,t\right>]\oplus\Q[\widehat{W}/\left<s,\widehat{s}_0\right>]\oplus\Q[\widehat{W}/\left<t,\widehat{s}_0\right>].}\]
Now, by Lemma \ref{GCWandSDP}, the complex for the surface $\mathbf{T}(I_2(m))$ is the image of the previous one by the deflation functor $\mathrm{Def}^{\widehat{W}}_W$. Thus, it is of the form
\[\xymatrix{\Q[W] \ar[r] & \Q[W/\left<s\right>]\oplus\Q[W/\left<t\right>]\oplus\Q[W/\left<r\right>] \ar[r] & \Q\oplus\Q[W/\left<s,r\right>]\oplus\Q[W/\left<t,r\right>]},\]
where $r=r_W:=(st)^{\left\lfloor{(m-1)}/{2}\right\rfloor}s\in W$. Therefore the Euler characteristic is given by 
\[\chi(\mathbf{T}_g)=1+[W:\left<s,r\right>]+[W:\left<t,r\right>]-[W:\left<s\right>]-[W:\left<t\right>]-[W:\left<r\right>]+|W|\]
\[=1+[W:\left<s,r\right>]+[W:\left<t,r\right>]-3[W:\left<s\right>]+2m=1-m+[W:\left<s,r\right>]+[W:\left<t,r\right>].\]
It is routine to compute that
\[[I_2(m):\left<s,r\right>]=\left\{\begin{array}{cc}
2 & \text{if $m=4g+2$}, \\ 1 & \text{otherwise}\end{array}\right.~~\text{and}~~[I_2(m):\left<t,r\right>]=\left\{\begin{array}{cc}
1 & \text{if $m$ is odd}, \\ m/2 & \text{otherwise}\end{array}\right.\]
thus, the Euler characteristic is given by
\[\chi(\mathbf{T}(I_2(m)))=\left\{\begin{array}{ccc}
3-m & \text{if} & m=2g+1, \\
3-m/2 & \text{if} & m=4g+2, \\
2-m/2 & \text{if} & m=4g,\end{array}\right.\]
so $\chi(\mathbf{T}(I_2(m)))=2-2g$ and the genus of $\mathbf{T}(I_2(m))$ is indeed $g$ for $m\in\{2g+1,4g+2,4g\}$.
\end{proof}

As the fundamental group of a Riemann surface of genus $g\ge 1$ is well-known (see \cite[\S 1.2]{hatcher}), we obtain a presentation for the group $Q$ in the dihedral case.

\begin{rem}\label{QforI_2}
Let $g\in\N^*$ and $m$ be either $2g+1$, $4g$ or $4g+2$. Let also $Q$ be the subgroup of $\widehat{I_2(m)}$ constructed in the previous section (see Definition \ref{defofQ}). Then we have
\[Q\simeq\pi_1(\mathbf{T}(I_2(m)))\simeq\left<x_1,\dotsc,x_g,y_1,\dotsc,y_g~|~[x_1,y_1]\cdots [x_g,y_g]=1\right>.\]
\end{rem}

In the cases where $g=1$ that is, if $I_2(m)$ is one of the Weyl groups $I_2(3)=A_2$, $I_2(4)=B_2$ or $I_2(6)=G_2$, then $\mathbf{T}(I_2(m))$ is naturally an \emph{elliptic curve}. More precisely, recalling the notation of the previous section, we have a preferred point 
\[v_0:=\overline{C}\cap\mathcal{H}\cap\bigcap_{s\ne\widehat{s}_0}H_s\in\widehat{\Sigma},\]
and the pair $(\mathbf{T}(I_2(m)),[v_0])$ is a Riemann surface of genus 1 with a marked point, hence an elliptic curve. Under the diffeomorphism $\mathbf{T}(I_2(m))\simeq\R^2/\Z^2$ induced by quotienting the space $\widehat{V}$ by the radical of the Tits form of $I_2(m)_\mathrm{a}$, the point $[v_0]$ corresponds to the origin.

We can identify the elliptic curves $\mathbf{T}(I_2(m))$ (for $m=3,4,6$) in the moduli space $\mathcal{M}_{1,1}$ of complex elliptic curves. Recall that, if $\mathbb{H}^2=\{z\in\C~;~\Im(z)>0\}$ is the Poincar\'e half plane, then we have an isomorphism (see \cite[\S 2]{hain_lectures})
\[\begin{array}{ccc}
\mathbb{H}^2/PSL_2(\Z) & \stackrel{\tiny{\sim}}\longrightarrow & \mathcal{M}_{1,1} \\ \tau & \longmapsto & \C/(\Z+\tau\Z)\end{array}\]
Recall also from \cite[Chapitre VII, \S 1.2]{serre_ari} that $D:=\{z\in\mathbb{H}^2~;~|\Re(z)|\le 1/2,~|z|\ge1\}$ is a fundamental domain for $PSL_2(\Z)$ acting on $\mathbb{H}^2$. We just have to determine a corresponding element $\tau\in D$ for each case. We have the following proposition:

\begin{prop}
Let $m\in\{3,4,6\}$ and let $\{\alpha^\vee,\beta^\vee\}$ be the simple coroots of the root system of type $I_2(m)$ and $V^*:=\R\left<\alpha^\vee,\beta^\vee\right>$. We denote by $\phi : V^*\to\C$ the unique isometry sending $\alpha^\vee$ to $1$ and $\beta^\vee$ to an element of the upper-half plane $\mathbb{H}^2$. Then we have
\[\phi(\beta^\vee)=\left\{\begin{array}{ccc}
\exp\left(\tfrac{2i\pi}{3}\right) & \text{if} & m=3, \\[.5em]
\sqrt{2}\exp\left(\tfrac{3i\pi}{4}\right) & \text{if} & m=4, \\[.5em]
\sqrt{3}\exp\left(\tfrac{5i\pi}{6}\right) & \text{if} & m=6.\end{array}\right.\]
In particular, the corresponding lattice is $\Z\oplus\tau\Z$ where $\tau\in D$ equals $e^{\frac{2i\pi}{3}}$ for $A_2$ and $G_2$ and equals $i$ for $B_2$. Hence, the curves $\mathbf{T}(A_2)$, $\mathbf{T}(B_2)$ and $\mathbf{T}(G_2)$ are rational and correspond to the \emph{orbifold points} of $D$, i.e. the points in $D$ having a non-trivial stabilizer in $PSL_2(\R)$.
\end{prop}

\begin{proof}
We normalize the roots in such a way that the short simple roots have norm 2. For $I_2(3)=A_2$, we have $|\beta|=|\alpha|=2$ and $\left<\alpha^\vee,\beta\right>=\left<\beta^\vee,\alpha\right>=-1$. Therefore, since $\phi$ is an isometry we should have
\[-\frac{1}{2}=\frac{1}{2}\left<\alpha^\vee,\beta\right>=\left<\alpha^\vee,\beta^\vee\right>=\left<\phi(\alpha^\vee),\phi(\beta^\vee)\right>=\left<1,\phi(\beta^\vee)\right>\stackrel{\tiny{\text{df}}}=\Re(\overline{\phi(\beta^\vee)})=\Re(\phi(\beta^\vee)).\]
On the other hand, we have $1=\left<\beta^\vee,\beta^\vee\right>=|\phi(\beta^\vee)|^2$ and this implies $\phi(\beta^\vee)\in\{-1/2\pm i\sqrt{3}/2\}$ and if we impose that $\phi(\beta^\vee)\in\mathbf{H}^2$, then it should have a positive imaginary part and the only possibility is $\phi(\beta^\vee)=-1/2+i\sqrt{3}/2=\exp\left(\tfrac{2i\pi}{3}\right)$. The other cases are similar.
\end{proof}

\begin{prop}\label{isometrybetweenT(I2(m))}
If $g>1$, then we have an isometry (and in particular, an isomorphism of Riemann surfaces)
\[\mathbf{T}(I_2(4g+2))\simeq\mathbf{T}(I_2(2g+1))\]
and these two are not isometric to the surface $\mathbf{T}(I_2(4g))$.
\end{prop}

\begin{proof}
By \cite[Theorem 8.1.5]{ratcliffe_hyperbolic}, it suffices to show that the groups $Q_{2g+1}$ and $Q_{4g+2}$ are conjugate in the positive Lorentz group $PO(1,2)\simeq\mathrm{Isom}(\mathbf{H}^2)$ and are not conjugate to $Q_{4g}$.

Let $m:=2g+1$ and let us prove that $Q_m$ and $Q_{2m}$ are conjugate in $PO(1,2)$. Denote $I_2(2m)=\left<s,t~|~s^2=t^2=(st)^{2m}=1\right>$ and $\widehat{I_2(2m)}=\left<s,t,\widehat{s}_0\right>$ its hyperbolic extension. Let $s':=s$, $t':=tst=s^t$ and $\widehat{s}_0':=\widehat{s}_0$. Then $\left<s',t',\widehat{s}_0'\right>=\widehat{I_2(m)}$ and $\left<s',t'\right>=I_2(m)$. Recall moreover that we have the reflection $r_{2m}=(st)^{2g}s=((st)^2)^gs=(s't')^gs'=r_m$. Let $\alpha$, $\beta$ and $\gamma$ denote the simple roots of $\widehat{I_2(2m)}$ and $V_{2m}:=\mathrm{span}_\R(\alpha,\beta,\gamma)$. We have the representation
\[\widehat{I_2(2m)}\stackrel{\tiny{\sigma_{2m}}}\longhookrightarrow O(V_{2m},B_{2m}),\]
where
\[B_{2m}=\left(\begin{smallmatrix}1 & -\cos(\pi/2m) & -\cos(\pi/m) \\ -\cos(\pi/2m) & 1 & 0 \\ -\cos(\pi/m) & 0 & 1\end{smallmatrix}\right).\]
In the same way, denote $V_m:=\mathrm{span}_\R(\alpha',\beta',\gamma')$ and $\sigma_m : \widehat{I_2(m)}\longhookrightarrow O(V_m,B_m)$, where
\[B_m=\left(\begin{smallmatrix}1 & -\cos(\pi/m) & -\cos(\pi/m) \\ -\cos(\pi/m) & 1 & -\cos(\pi/m) \\ -\cos(\pi/m) & \cos(\pi/m) & 1\end{smallmatrix}\right).\]
Consider the linear map $P : V_{2m}\to V_m$ with matrix
\[P=\left(\begin{smallmatrix}1 & 1 & 0 \\ 0 & 2\cos(\pi/2m) & 0 \\ 0 & 0 & 1\end{smallmatrix}\right).\]
Then we have $B_m={}^{\rm t\!}{P}B_{2m}P$, so $P$ induces an isomorphism $O(V_{2m},B_{2m})\to O(V_m,B_m)$ fitting in a commutative diagram
\[\xymatrix{\widehat{I_2(m)} \ar@{^{(}->}[d] \ar@{^{(}->}^<<<<<{\sigma_m}[r] & O(V_m,B_m) \ar^<<<<<{\sim}[r] & PO(1,2) \\ \widehat{I_2(2m)} \ar@{^{(}->}_<<<<<{\sigma_{2m}}[r] & O(V_{2m},B_{2m}) \ar^{\sim}[u] \ar^<<<<<{\sim}[r] & PO(1,2 \ar^{\sim}[u])}\]
and thus the group $\sigma_m(\widehat{I_2(m)})$ is conjugate in $PO(1,2)$ to a subgroup of $\sigma_{2m}(\widehat{I_2(2m)})$. Therefore, identifying $\widehat{I_2(m)}$ with its image in $\widehat{I_2(2m)}$, it suffices to prove that $Q_{2m}=Q_m$. Recall that $q_{2m}\stackrel{\tiny{\text{df}}}=\widehat{s}_0r_{2m}=\widehat{s}_0r_m=q_m$, so $q_m^{\widehat{I_2(m)}}\subset q_{2m}^{\widehat{I_2(2m)}}$ and $Q_m\le Q_{2m}$. Since we have
\begin{align*}
2m[\widehat{I_2(2m)}:\widehat{I_2(m)}]&=[\widehat{I_2(2m)}:\widehat{I_2(m)}][\widehat{I_2(m)}:Q_m]=[\widehat{I_2(2m)}:Q_m] \\
&=[\widehat{I_2(2m)}:Q_{2m}][Q_{2m}:Q_m]=4m[Q_{2m}:Q_m],
\end{align*}
we are left to show that $[\widehat{I_2(2m)}:\widehat{I_2(m)}]=2$. Let $w\in \widehat{I_2(2m)}$. By induction on $\ell(w)$ and because $t$ and $\widehat{s}_0$ commute, we immediately see that $w\in \widehat{I_2(m)}$ if and only if the number of occurrences of $t$ in any reduced expression of $w$ is even. Hence we have $[\widehat{I_2(2m)}:\widehat{I_2(m)}]\le 2$ and since $st$ and $\widehat{s}_0t$ have even order, the map $\widehat{I_2(2m)}\to \Z/2\Z$ sending $s$ and $\widehat{s}_0$ to $0$ and $t$ to $1$ is a homomorphism whose kernel contains $\widehat{I_2(m)}$, hence the result.

We now prove that $Q_{2g+1}$ and $Q_{4g}$ are not conjugate in $PO(1,2)$. It is enough to prove that the elements $\sigma_{2g+1}(q_{2g+1})\in PO(1,2)$ and $\sigma_{4g}(q_{4g})$ have different traces. Write $\widehat{I_2(2g+1)}=\left<s,t,\widehat{s}_0\right>$ and $\widehat{I_2(4g)}=\left<s',t',\widehat{s}_0'\right>$. We have $q_{2g+1}=\widehat{s}_0(st)^gs$ and $q_{4g}=\widehat{s}_0'(st)^{2g}s'$ and we can write explicitly the matrices of the simple reflections in the geometric representation, we diagonalize $st$ and after calculations, we find
\[\tr(q_{2g+1})=8\left(1+\cos\left(\frac{\pi}{2g+1}\right)\right)\cot^2\left(\frac{\pi}{2g+1}\right)-1\ne4\cot^2\left(\frac{\pi}{4g}\right)-1=\tr(q_{4g}).\]
\end{proof}

Recall that a \emph{Belyi function} on a Riemann surface $X$ is a holomorphic map $\beta : X \to \widehat{\C}$ which is ramified only over three points of $\widehat{\C}$. Since $\widehat{I_2(m)}$ is a compact triangle group and $Q_m\unlhd \widehat{I_2(m)}$ is torsion-free and of finite index, by \cite[Theorem 3.10]{jones-wolfart}, the projection
\[\beta : \mathbf{T}(I_2(m))=\mathbb{H}^2/Q_m \longtwoheadrightarrow \mathbb{H}^2/\widehat{I_2(m)}\simeq\widehat{\C}\]
is a Belyi function on $\mathbf{T}(I_2(m))$ of degree $[\widehat{I_2(m)}:Q_m]=2m$. Using \cite[Theorem 1.3]{jones-wolfart}, this implies the following result:

\begin{prop}\label{definableoverQbar}
For any $m\ge 3$, the Riemann surface $\mathbf{T}(I_2(m))$ may be defined over a number field (or equivalently, may be defined over $\overline{\Q}$). Moreover, if $m=5$ or $m\ge 7$, then the 1-skeleton of the tessellation $\Delta(\widehat{I_2(m)})/Q_m$ defines a dessin d'enfant on $\mathbf{T}(I_2(m))$.
\end{prop}

\begin{rem}
It is a reasonable to expect that $\mathbf{T}(I_2(m))$ is definable over $\Q(\cos(2\pi/m))$.
\end{rem}

\begin{exemple}\label{I2(5,7)}
The triangulation $\Delta(\widehat{I_2(5)})$ is the classical tessellation $\{3,10\}$ of the Poincar\'e disk. More precisely, the Tits form $\widehat{B}$ is given by
\[\widehat{B}=\left(\begin{smallmatrix}1 & -c & -c \\ -c & 1 & -c \\ -c & -c & 1\end{smallmatrix}\right)~~\text{with}~~c=\cos(\pi/5)\]
and, if $v^*\in V^*$ is a normalized eigenvector for the unique negative eigenvalue of $\widehat{B}$, then we have an identification with the hyperbolic plane
\[\mathcal{H}:=\{\lambda\in V^*~;~\widehat{B}(\lambda,\lambda)=-1,~\widehat{B}(v^*,\lambda)<0\}\simeq\mathbb{H}^2\]
and the stereographic projection on the hyperplane $\widehat{B}(v^*,-)=0$ with pole $\lambda_0$ gives the Poincar\'e disk model for $\mathbb{H}^2$. Under this projection, we represent the tessellation $\Delta(\widehat{I_2(5)})=\{3,10\}$ of $\mathcal{H}$ as in Figure \ref{3,10tessellation}, where the black triangles are the images of the fundamental triangle $\overline{C}/\R^*_+\simeq \overline{C}\cap\mathcal{H}$ under elements of odd length. In this tessellation, we can identify the triangles that are in the $Q$-orbit of $\overline{C}\cap\mathcal{H}$. These are displayed in green in Figure \ref{patreon}. Collapsing these triangles in one gives the surface $\mathbf{T}(I_2(5))$. The computations were done using \cite{maple}.

\begin{center}
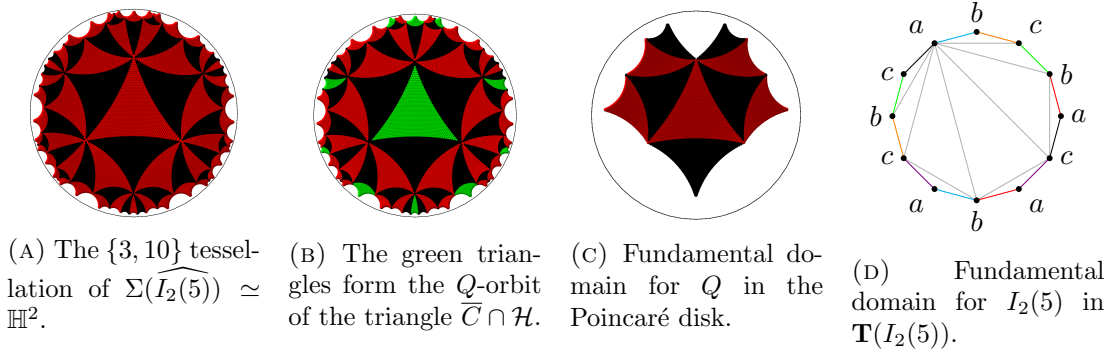
\begin{figure}[h!]
\begin{subfigure}[h]{0.22\textwidth}
\centering
\includegraphics[scale=0.11]{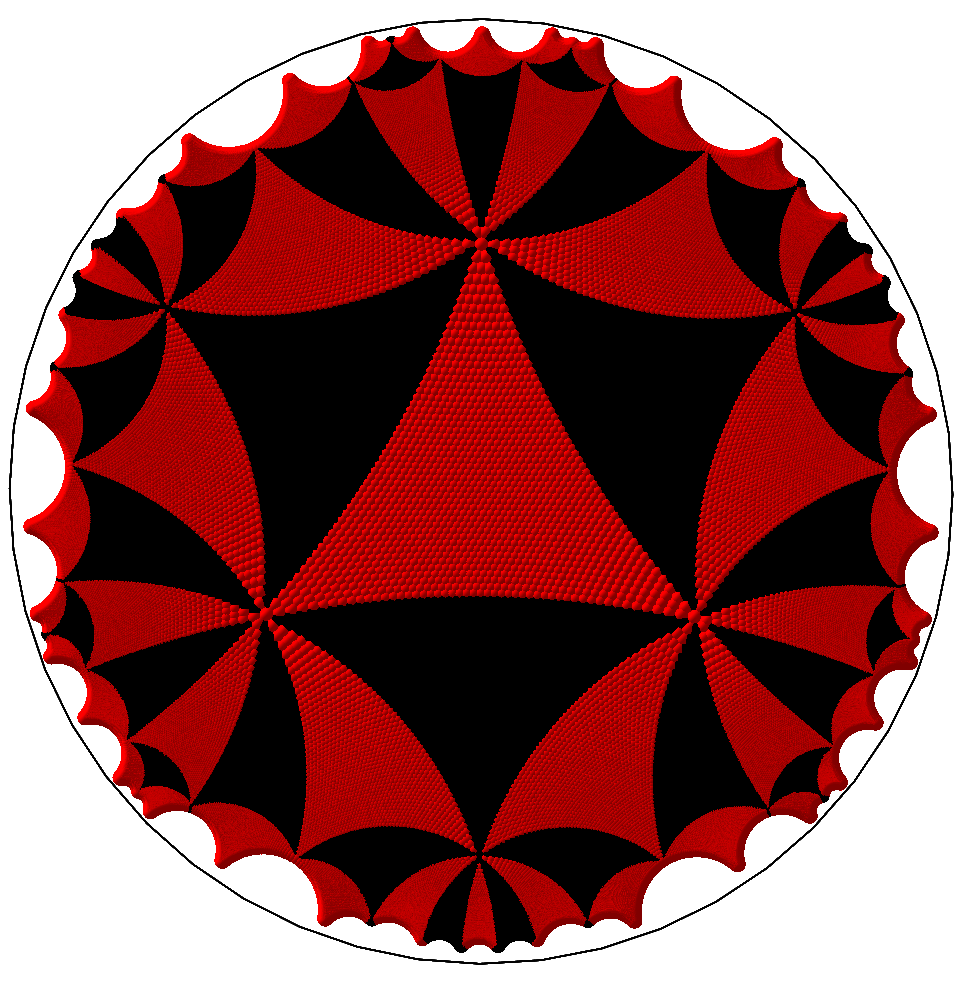}
\caption{The $\{3,10\}$ tessellation of $\Sigma(\widehat{I_2(5)})\simeq\mathbb{H}^2$.}
\label{3,10tessellation}
\end{subfigure}
\hspace{1mm}
\begin{subfigure}[h]{0.22\textwidth}
\centering
\includegraphics[scale=0.08]{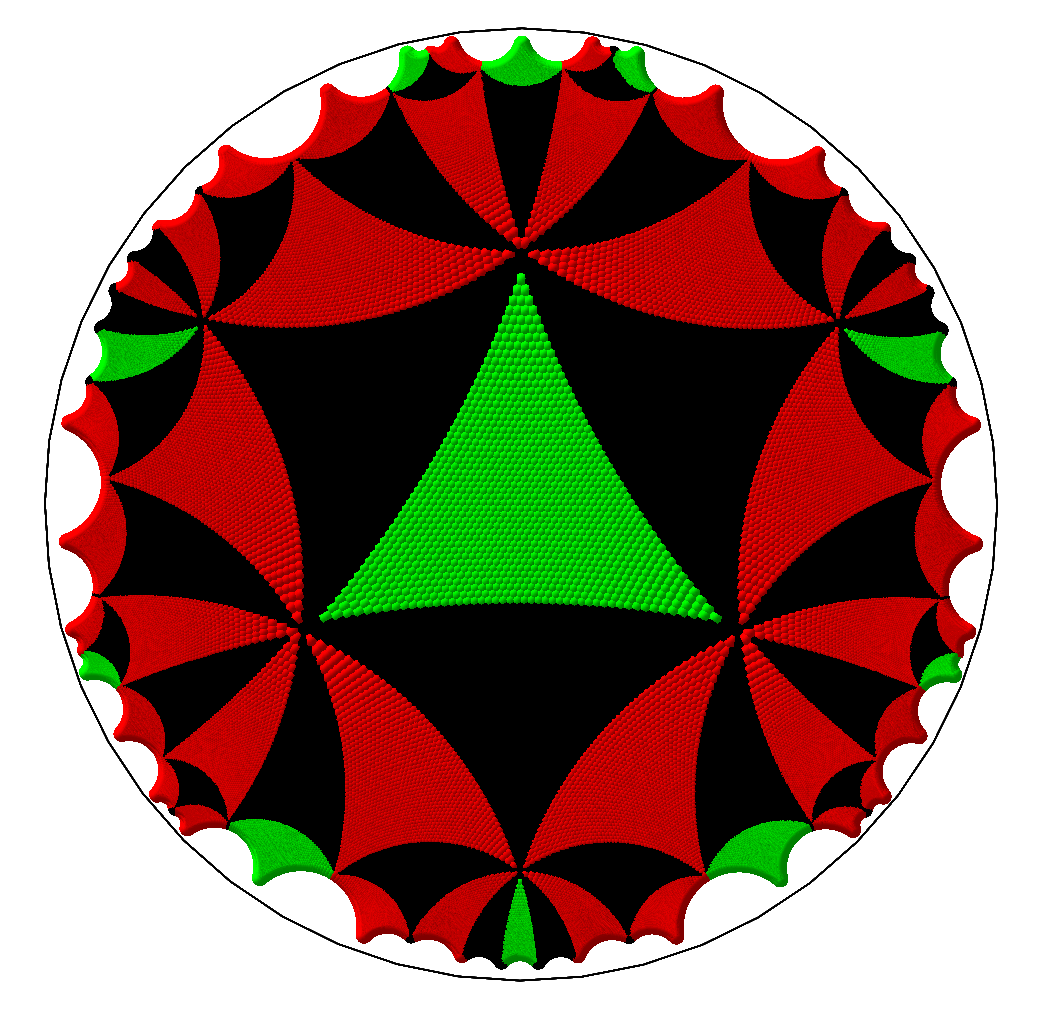}
\caption{The green triangles form the $Q$-orbit of the triangle $\overline{C}\cap\mathcal{H}$.}
\label{patreon}
\end{subfigure}
\hspace{1mm}
\begin{subfigure}[h]{0.22\textwidth}
\centering
\includegraphics[scale=0.14]{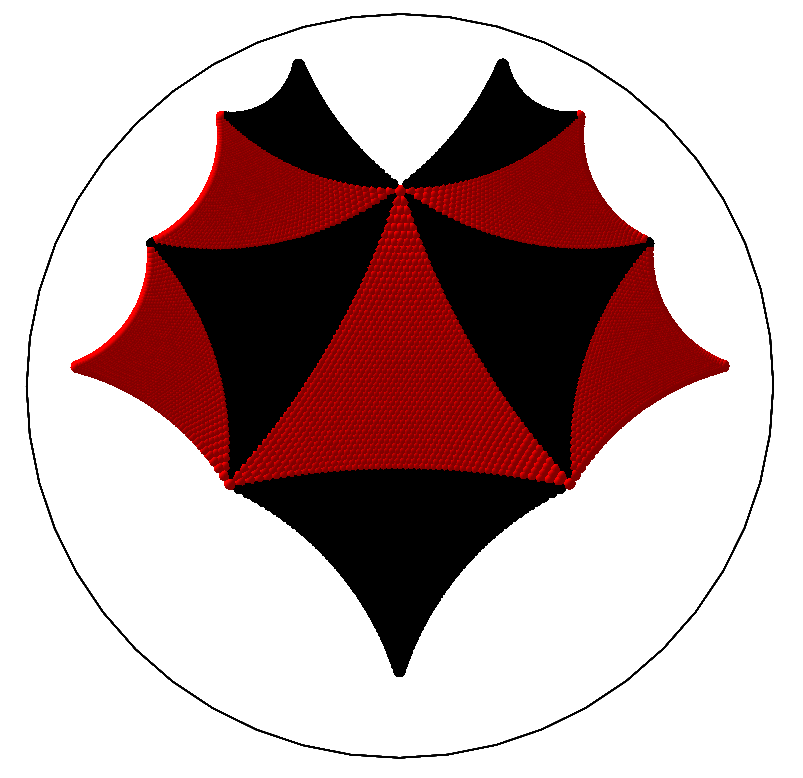}
\caption{Fundamental domain for $Q$ in the Poincar\'e disk.}
\label{fundT(I2(5))}
\end{subfigure}
\hspace{1mm}
\begin{subfigure}[h]{0.22\textwidth}
\centering
\begin{tikzpicture}[scale=1.12]
	\coordinate (a) at (1,0);
	\coordinate (b) at (0.866,0.5);
	\coordinate (c) at (0.5,0.866);
	\coordinate (d) at (0,1);
	\coordinate (e) at (-0.5,0.866);
	\coordinate (f) at (-0.866,0.5);
	\coordinate (g) at (-1,0);
	\coordinate (h) at (-0.866,-0.5);
	\coordinate (i) at (-0.5,-0.866);
	\coordinate (j) at (0,-1);
	\coordinate (k) at (0.5,-0.866);
	\coordinate (l) at (0.866,-0.5);
	
	\draw (a) node[right]{$a$};
	\draw (b) node[right]{$b$};
	\draw (c) node[above right]{$c$};
	\draw (d) node[above]{$b$};
	\draw (e) node[above left]{$a$};
	\draw (f) node[left]{$c$};
	\draw (g) node[left]{$b$};
	\draw (h) node[left]{$c$};
	\draw (i) node[below left]{$a$};
	\draw (j) node[below]{$b$};
	\draw (k) node[below right]{$a$};
	\draw (l) node[right]{$c$};
		
	\draw[red] (a)--(b)
	(j)--(k);
	\draw[green] (b)--(c)
	(f)--(g);
	\draw[orange] (c)--(d)
	(g)--(h);
	\draw[cyan] (d)--(e)
	(i)--(j);
	\draw (e)--(f)
	(l)--(a);
	\draw[violet] (h)--(i)
	(k)--(l);
	
	\draw[opacity=0.3] (e)--(b)
	(e)--(c)
	(e)--(g)
	(e)--(h)
	(e)--(j)
	(e)--(l)
	(l)--(b)
	(j)--(h)
	(j)--(l);
	
	\fill[fill=black] (a) circle (1pt);
	\fill[fill=black] (b) circle (1pt);
	\fill[fill=black] (c) circle (1pt);
	\fill[fill=black] (d) circle (1pt);
	\fill[fill=black] (e) circle (1pt);
	\fill[fill=black] (f) circle (1pt);
	\fill[fill=black] (g) circle (1pt);
	\fill[fill=black] (h) circle (1pt);
	\fill[fill=black] (i) circle (1pt);
	\fill[fill=black] (j) circle (1pt);
	\fill[fill=black] (k) circle (1pt);
	\fill[fill=black] (l) circle (1pt);
\end{tikzpicture}
\caption{Fundamental domain for $I_2(5)$ in $\mathbf{T}(I_2(5))$.}
\label{2toruspatreon}
\end{subfigure}
\caption{$I_2(5)$-tessellation of the Poincar\'e disk, $Q$-orbit of the fundamental triangle, fundamental domain for $Q$ and its image in $\mathbf{T}(I_2(5))$. The first three images were created with \cite{maple}.}
\end{figure}
\end{center}
We remark that we can extract a fundamental domain for $Q$ on $\mathbf{T}(I_2(5))$ as the projection of the domain displayed in Figure \ref{fundT(I2(5))}. Rearranging the figure we obtain the triangulation displayed in the Figure \ref{2toruspatreon}, where the points with the same name (resp. the edges with the same color) are identified. The resulting space is indeed a closed surface of genus 2.
\end{exemple}

\section{Equivariant chain complex of $\mathbf{T}(W)$ and computation of homology}

\subsection{The $W$-dg-ring of $\mathbf{T}(W)$}
\hfill

The combinatorics of the complex $C_\mathrm{cell}^*(\mathbf{T}(W),W;\Z)$ is fairly similar to the one of the complex $C_\mathrm{cell}^*(T,W;\Z)$ we constructed in the first part and the proofs given above can be applied \emph{verbatim} to this new situation.

\begin{theo}\label{cochaincomplexforT(W)}
The $W$-dg-ring $C^*_{\mathrm{cell}}(\mathbf{T}(W),W;\Z)$ associated to the $W$-triangulation $\Delta(\widehat{W},\widehat{S})/Q$ of $\mathbf{T}(W)$ has homogeneous components
\[C^k_\mathrm{cell}(\mathbf{T}(W),W;\Z)=\bigoplus_{\substack{I\subset \widehat{S} \\ |I|=n-k}}\Z[\pi({}^I{\widehat{W}})]\simeq\bigoplus_{\substack{I\subset \widehat{S} \\|I|=n-k}}\Z[\pi(\widehat{W}_I)\backslash \widehat{W}],\]
differentials given, for any $I\subset\widehat{S}$ and $w\in\widehat{W}$, by
\[d^k(\pi({}^I{w}))=\sum_{\substack{0\le u \le k+1 \\ j_{u-1}<j<j_u}}(-1)^u\pi\left(\epsilon^{I\setminus\{j\}}_Iw\right),~\epsilon_I^J=\sum_{x\in{}^J_I{\widehat{W}}}x,\]
where $\{j_0<\cdots<j_k\}:=\widehat{S}\setminus I$. Its product 
\[C^p_\mathrm{cell}(\mathbf{T}(W),W;\Z)\otimes_\Z C^q_\mathrm{cell}(\mathbf{T}(W),W;\Z) \stackrel{\tiny{\cup}}\longrightarrow C^{p+q}_\mathrm{cell}(\mathbf{T}(W),W;\Z)\]
is induced by the deflation from $\widehat{W}$ to $W$ of the unique map
\[\Z[[{}^I\widehat{W}]]\otimes_\Z\Z[[{}^J\widehat{W}]]\longrightarrow \Z[[{}^{I\cap J}\widehat{W}]]\]
satisfying the formula
\[\forall x,y\in\widehat{W},~{}^I{x}\cup{}^J{y}=\delta_{\max(I^\complement),\min(J^\complement)}\times\left\{\begin{array}{cc}
{}^{I\cap J}((xy^{-1})_Jy) & \text{if }xy^{-1}\in\widehat{W}_I\widehat{W}_J \\ 0 & \text{otherwise}.\end{array}\right.\]
\end{theo}

\begin{rem}
As explained in \cite[\S 2.3]{babson-reiner}, a quotient simplicial complex of the form $\Delta(\widehat{W},\widehat{S})/H$ (with $H\le\widehat{W}$) has a an interpretation in terms of double cosets. In our case, we have an isomorphism of posets
\[\begin{array}{ccc}
\left(F(\Delta(\widehat{W},\widehat{S})/Q),\subseteq\right) & \stackrel{\tiny{\sim}}\longrightarrow & \left(\{(I,Qw\widehat{W}_I)\}_{I\subsetneq\widehat{S},~w\in\widehat{W}},\preceq\right) \\ \pi(\widehat{W}^I) \ni \pi(w^I) & \longmapsto & (I,Qw\widehat{W}_I) \end{array}\]
where the order $\preceq$ on the second factor is defined by
\[(I,Qw\widehat{W}_I)\preceq(J,Qw'\widehat{W}_J)~\stackrel{\tiny{\text{df}}}\Longleftrightarrow~\left\{\begin{array}{c}I\supseteq J \\ Qw\widehat{W}_I\supseteq Qw'\widehat{W}_J\end{array}\right.\]
and we may rephrase the above results using this poset.
\end{rem}

\subsection{The homology $W$-representation of $\mathbf{T}(W)$}
\hfill

We can now determine the action of $W$ on $H_*(\mathbf{T}(W);\overline{\Q})$. Recall from \cite[Theorem 5.3.8]{geck-pfeiffer} that a splitting field for $W$ is given by
\[\Q(W)=\Q(\cos(2\pi/m_{s,t}),~s,t\in S).\]
If $W$ is a Weyl group, then $\Q(W)=\Q$ and we have $\Q(I_2(m))=\Q(\cos(2\pi/m))$ and $\Q(H_3)=\Q(H_4)=\Q(\sqrt{5})$. Recall that we have $n=\mathrm{rk}(W)=\dim\mathbf{T}(W)$.

\begin{prop}\label{H0andHn}
Let $\mathds{1}$ and $\varepsilon$ be the trivial and signature modules over $\Z[W]$, respectively. We have isomorphisms of $\Z[W]$-modules
\[H_0(\mathbf{T}(W);\Z)\simeq\mathds{1}~~\text{and}~~H_n(\mathbf{T}(W);\Z)\simeq\varepsilon.\]
\end{prop}

\begin{proof}
Since $\widehat{\Sigma}$ is path-connected, its quotient $\mathbf{T}(W)$ is path-connected too and is orientable by Theorem \ref{defsteinbergtorus}. Thus, we have an isomorphism of abelian groups
\[H_0(\mathbf{T}(W);\Z)\simeq\Z\simeq H_n(\mathbf{T}(W);\Z).\]
It is clear that $H_0(\mathbf{T}(W);\Z)$ is the trivial module and we have $H_n(\mathbf{T}(W);\Z)=\ker(\partial_n)$ with
\[\begin{array}{ccccc}
\partial_n & : & \Z[W] & \longrightarrow & \bigoplus_{i=0}^n\Z[W/\left<s_i\right>] \\ & & w & \longmapsto & \sum_{i=0}^n(-1)^iw\!\left<s_i\right>\end{array}\]
where $s_i=\pi(\widehat{s}_i)$ is a simple reflection of $W$ for $i\ge 1$ and $s_0:=r_W=\pi(\widehat{s}_0)$. It is routine to check that $\ker\partial_n=\Z e$ where $e=\sum_w\varepsilon(w)w=\sum_w(-1)^{\ell(w)}w\in\Z[W]$.
\end{proof}

\begin{prop}\label{PD}
The homology $H_*(\mathbf{T}(W);\Z)$ is torsion-free and the Poincar\'e duality on $\mathbf{T}(W)$ induces isomorphisms of $\Z[W]$-modules
\[H_{n-i}(\mathbf{T}(W);\Z)^\vee\simeq H_i(\mathbf{T}(W);\Z)^\vee\otimes_\Z\varepsilon.\]
\end{prop}

\begin{proof}
The first statement, together with the Poincar\'e map $H^i(\mathbf{T}(W)) \otimes H_{n-i}(\mathbf{T}(W)) \to \varepsilon$ and the universal coefficients theorem, imply the second one. If $W$ is a Weyl group, then $\mathbf{T}(W)\simeq T$ is a torus and the statement follows from the isomorphism $\Lambda^*(X(T))\stackrel{\tiny{\sim}}\to H^*(T;\Z)$ proven in the Remark \ref{lambda_homo}. In the non-crystallographic case, since $\dim\mathbf{T}(W)\le4$, we only have to see that $H_1(\mathbf{T}(W);\Z)$ is torsion-free, which is true by Corollaries \ref{abelianizationofQ} and \ref{QforI_2}.
\end{proof}

The above Proposition, combined with the \emph{Hopf trace formula} (see \cite[Chap. 4, \S 7, Theorem 6]{spanier} or \cite[Lemma 2.4]{lin_lefschetz}), provides enough information to determine the homology representation of $\mathbf{T}(W)$. Let $G$ be a (discrete) group, $H\le G$ a subgroup and $M$ be an $H$-module. We denote by $M{\uparrow}_H^G$ the \emph{induced module} of $M$; it is a $G$-module. Similarly, the \emph{restricted module} of a $G$-module $N$ is denoted $N{\downarrow}_H^G$. Observe that we have a canonical isomorphism of $\Q[G]$-modules $\Q[G/H]\simeq\mathds{1}{\uparrow}_H^G$.

In our context, we have isomorphisms of $\Q[\widehat{W}]$-modules
\[C_k^\mathrm{cell}(\widehat{\Sigma},\widehat{W};\Q)=\bigoplus_{I\subset\widehat{S}~;~|I|=n-k}\mathds{1}{\uparrow}_{\widehat{W}_I}^{\widehat{W}}.\]
Thus 
\[C_k^\mathrm{cell}(\mathbf{T}(W),W;\Q)=\mathrm{Def}^{\widehat{W}}_W(C_k^\mathrm{cell}(\widehat{\Sigma},\widehat{W};\Q))=\bigoplus_{I\subset\widehat{S}~;~|I|=n-k}\mathds{1}{\uparrow}_{\pi(\widehat{W}_I)}^W,\]
and applying Hopf's formula yields the following result:
\begin{lem}\label{hopfforW}
We have the following equality of virtual rational characters of $W$
\[\sum_{I\subsetneq\widehat{S}}(-1)^{|I|}\mathds{1}{\uparrow}_{\pi(\widehat{W}_I)}^W=(-1)^n\sum_{0\le i \le n}(-1)^iH_i(\mathbf{T}(W);\Q).\]
\end{lem}

We use the conventions of \cite{geck-pfeiffer} to denote the irreducible characters of $W$.

\begin{theo}\label{homologyI2}
Let $m\ge 3$. Following \cite[\S 5.3.4]{geck-pfeiffer}, for $1\le j \le \lfloor (m-1)/2\rfloor$, we consider the following representation of $I_2(m)=\left<s,t~|~s^2=t^2=(st)^m=1\right>$
\[\widetilde{\rho_j} : I_2(m) \to GL_2(\R)~~\text{defined by}~~\widetilde{\rho_j}(s):=\begin{pmatrix}0 & 1 \\ 1 & 0\end{pmatrix}~~\text{and}~~\widetilde{\rho_j}(st):=\begin{pmatrix} \cos(j\theta_m) & -\sin(j\theta_m) \\ \sin(j\theta_m) & \cos(j\theta_m)\end{pmatrix},\]
where $\theta_m:=2\pi/m$. Let $\rho_j$ be a realization of $\widetilde{\rho_j}$ on the splitting field $\Q(\theta_m)$ of $I_2(m)$.

Then, the first homology representation of $\mathbf{T}(I_2(m))$ is given by
\[H_1(\mathbf{T}(I_2(m));\Q(\theta_m))=\left\{\begin{array}{cc}
\displaystyle{\bigoplus_{1\le j \le (m-1)/2}\rho_j} & \text{if $m$ is odd}, \\[.5em]
\displaystyle{\bigoplus_{\substack{1\le j \le m/2-1 \\ j~\text{odd}}}\rho_{j}} & \text{if $m$ is even}.\end{array}\right.\]

In particular, this representation is irreducible exactly when $m\in\{3,4,6\}$.
\end{theo}

\begin{proof}
Let $\chi_j:=\mathrm{tr}(\rho_j)$ and, denoting by $\mathrm{Reg}_{\Q(\theta_m)}=\Q(\theta_m)[I_2(m)]$ the regular module, the Lemma \ref{hopfforW} yields the following equality of virtual characters of $I_2(m)$
\[H_1(\mathbf{T}(I_2(m));\Q(\theta_m))=\mathds{1}+\varepsilon-\mathrm{Reg}_{\Q(\theta_m)}-\sum_{\emptyset\ne I\subsetneq\{\widehat{s}_0,s,t\}}(-1)^{|I|}\mathds{1}{\uparrow}_{\pi(\widehat{I_2(m)}_I)}^{I_2(m)}.\]
We deal with each case separately. Recall the images in $I_2(m)$ of the parabolic subgroups of $\widehat{I_2(m)}$ from the proof of the Corollary \ref{T(I2(m))hyperbolicsurface}. Denote $a:=st$ and $r:=r_W=a^{\lfloor(m-1)/2\rfloor}s$.
\begin{itemize}[leftmargin=*]
\item If $m=2k+1$, then $r=a^ks$ and  $s^{r}=t$, $t^{r}=s$, so $I_2(m)=\left<s,r\right>=\left<t,r\right>$. We also have (cf \cite[\S 5.3.4]{geck-pfeiffer}) $\mathrm{Reg}_{\Q(\theta_m)}=\mathds{1}+\varepsilon+\sum_j 2\chi_j$ and $\mathds{1}{\uparrow}_{\left<s\right>}^{I_2(m)}=\mathds{1}+\sum_j\chi_j$ (\cite[\S 6.3.5]{geck-pfeiffer}), so the formula reduces to
\[H_1(\mathbf{T}(I_2(m));\Q(\theta_m))=3\cdot\mathds{1}{\uparrow}_{\left<s\right>}^{I_2(m)}-3\cdot\mathds{1}-\sum_j2\chi_j=\sum_j\chi_j.\]
\item If $m=4k$, then $r=a^{2k-1}s$ and the conjugacy classes of $I_2(m)$ are as follows
\begin{center}
\begin{tabular}{|c||c|c|c|c|c|c|c|c|}
\hline
Representative & $1$ & $a$ & $a^2$ & $\cdots$ & $a^{m/2-1}$ & $a^{m/2}$ & $s$ & $t$ \\
\hline
Cardinality & $1$ & $2$ & $2$ & $\cdots$ & $2$ & $1$ & $m/2$ & $m/2$ \\
\hline
\end{tabular}
\end{center}

First, we determine the characters $\mathds{1}{\uparrow}_{\left<x,r\right>}^{I_2(m)}$ for $x=s,t$. In the proof of \ref{T(I2(m))hyperbolicsurface} we have seen that $t=(sr)^{2k-1}s$, so $\left<s,r\right>=I_2(m)$. Next, as detailed in \cite[\S 5.3.4]{geck-pfeiffer}, the character $\chi_j$ is given by $\chi_j(a^i)=2\cos(2ij\pi/m)$ and $\chi_j(sa^i)=0$. We have $\left<t,r\right>=\{1,t,r,a^{2k}\}\simeq C_2\times C_2$ and by Frobenius reciprocity
\[\forall j,~\left(\mathds{1}{\uparrow}_{\left<t,r\right>}^{I_2(m)},\chi_j\right)_W=\left(\mathds{1},\chi_j{\downarrow}_{\left<t,r\right>}^{I_2(m)}\right)_{\left<t,r\right>}=\frac{\chi_j(1)+\chi_j(t)+\chi_j(r)+\chi_j(a^{2k})}{4}\]
\[=\frac{\chi_j(1)+2\chi_j(t)+\chi_j(a^{2k})}{4}=\frac{1+\cos(\pi j)}{2}=\left\{\begin{array}{cc}
1 & \text{if $j$ is even} \\ 0 & \text{otherwise}.\end{array}\right.\]
The 1-dimensional irreducible representations of $I_2(m)$ other that $\mathds{1}$ and $\varepsilon$ are given by $\varepsilon_s(s)=\varepsilon_t(t)=1$ and $\varepsilon_s(t)=\varepsilon_t(s)=-1$. Therefore, $\mathrm{Reg}_{\Q(\theta_m)}=\mathds{1}+\varepsilon+\varepsilon_s+\varepsilon_t+\sum_j2\chi_j$ and by Frobenius reciprocity,
\[\left(\mathds{1}{\uparrow}_{\left<t,r\right>}^{I_2(m)},\varepsilon_s\right)_{I_2(m)}=\left(\mathds{1}{\uparrow}_{\left<t,r\right>}^{I_2(m)},\varepsilon\right)_{I_2(m)}=0~~\text{and}~~\left(\mathds{1}{\uparrow}_{\left<t,r\right>}^{I_2(m)},\varepsilon_t\right)_{I_2(m)}=\left(\mathds{1}{\uparrow}_{\left<t,r\right>}^{I_2(m)},\mathds{1}\right)_{I_2(m)}=1\]
and hence $\mathds{1}{\uparrow}_{\left<t,r\right>}^{I_2(m)}=\mathds{1}+\varepsilon_t+\sum_{j~\text{even}}\chi_j$. On the other hand, by \cite[\S 6.3.5]{geck-pfeiffer}, we have $\mathds{1}{\uparrow}_{\left<s\right>}^{I_2(m)}=\mathds{1}+\varepsilon_s+\sum_j\chi_j$ and $\mathds{1}{\uparrow}_{\left<t\right>}^{I_2(m)}=\mathds{1}+\varepsilon_t+\sum_j\chi_j$. We finally get
\[H_1(\mathbf{T}(I_2(m)),\Q(\theta_m))=\varepsilon-\mathrm{Reg}_{\Q(\theta_m)}+\mathds{1}{\uparrow}_{\left<s\right>}^{I_2(m)}+2\cdot\mathds{1}{\uparrow}_{\left<t\right>}^{I_2(m)}-\mathds{1}{\uparrow}_{\left<t,r\right>}^{I_2(m)}-\mathds{1}=\sum_{j~\text{odd}}\chi_j.\]
\item If $m=4k+2$, then $r=a^{2k}s$, we find $\mathds{1}{\uparrow}_{\left<s\right>}^{I_2(m)}$ and $\mathds{1}{\uparrow}_{\left<t\right>}^{I_2(m)}$ as above and compute
\[\left(\mathds{1}{\uparrow}_{\left<s,r\right>}^{I_2(m)},\varepsilon_s\right)_{I_2(m)}=\left(\mathds{1},\varepsilon_s{\downarrow}_{\left<s,r\right>}^{I_2(m)}\right)_{\left<s,r\right>}=1=\left(\mathds{1}{\uparrow}_{\left<s,r\right>}^{I_2(m)},\mathds{1}\right)_{I_2(m)}\]
but since $\deg(\mathds{1}{\uparrow}_{\left<s,r\right>}^{I_2(m)})=[I_2(m):\left<s,r\right>]=2$ this gives $\mathds{1}{\uparrow}_{\left<s,r\right>}^{I_2(m)}=\mathds{1}+\varepsilon_s$. Now, we have $\left<t,r\right>=\{1,t,r,a^{2k+1}\}\simeq C_2\times C_2$ and using again the Frobenius reciprocity,
\[\left(\mathds{1}{\uparrow}_{\left<t,r\right>}^{I_2(m)},\chi_j\right)_{I_2(m)}=\left(\mathds{1},\chi_j{\downarrow}_{\left<t,r\right>}^{I_2(m)}\right)_{\left<t,r\right>}=\frac{\chi_j(1)+\chi_j(t)+\chi_j(r)+\chi_j(a^{2k+1})}{4}=\frac{1+\cos(\pi j)}{2}.\]
Since ${\varepsilon_s}{\downarrow}_{\left<t,r\right>}^{I_2(m)}\ne{\varepsilon_t}\ne\mathds{1}{\downarrow}_{\left<t,r\right>}^{I_2(m)}$ we also get
\[\left(\mathds{1}{\uparrow}_{\left<t,r\right>}^{I_2(m)},\varepsilon_s\right)_{I_2(m)}=\left(\mathds{1}{\uparrow}_{\left<t,r\right>}^{I_2(m)},\varepsilon_t\right)_{I_2(m)}=\left(\mathds{1}{\uparrow}_{\left<t,r\right>}^{I_2(m)},\varepsilon\right)_{I_2(m)}=0,\]
so $\mathds{1}{\uparrow}_{\left<t,r\right>}^{I_2(m)}=\mathds{1}+\sum_{j~\text{even}}\chi_j$ and we conclude that
\[H_1(\mathbf{T}(I_2(m));\Q(\theta_m))=\varepsilon-\mathrm{Reg}_{\Q(\theta_m)}+2\cdot\mathds{1}{\uparrow}_{\left<s\right>}^{I_2(m)}+\mathds{1}{\uparrow}_{\left<t\right>}^{I_2(m)}-\mathds{1}{\uparrow}_{\left<s,r\right>}^{I_2(m)}-\mathds{1}{\uparrow}_{\left<t,r\right>}^{I_2(m)}=\sum_{j~\text{odd}}\chi_j.\]
\end{itemize}
\end{proof}

\begin{theo}\label{homologyH3}
With the notation of \cite[Appendix C, Table C.1]{geck-pfeiffer}, we have
\[\forall 0\le i\le 3,~H_i(\mathbf{T}(H_3);\Q(\sqrt{5}))=\left\{\begin{array}{ccc}
\mathds{1} & \text{if} & i=0, \\[.5em]
3_s'\oplus\overline{3'_s}\oplus5_r & \text{if} & i=1, \\[.5em]
3_s\oplus\overline{3_s}\oplus5'_r & \text{if} & i=2, \\[.5em]
\varepsilon & \text{if} & i=3. \end{array}\right.\]
\end{theo}

\begin{proof}
Consider the virtual character $\chi_H:=\sum_{I\subsetneq\widehat{S}}(-1)^{|I|+1}\mathds{1}{\uparrow}_{\pi(\widehat{H_3}_I)}^{H_3}$. For $\chi\in\mathrm{Irr}(H_3)$, we compute $(\chi_H,\chi)_{H_3}$ using Frobenius reciprocity and we obtain
\[\chi_H=\varepsilon-\mathds{1}-3_s-\overline{3_s}+3_s'+\overline{3'_s}+5_r-5_r'.\]
Therefore, using the Lemma \ref{hopfforW}
\[H_2(\mathbf{T}(H_3))-H_1(\mathbf{T}(H_3))=\chi_H+\mathds{1}-\varepsilon=-3_s-\overline{3_s}+3_s'+\overline{3_s'}+5_r-5'_r.\]
By the Lemma \ref{abelianizationofQ}, we have $\dim(H_1(\mathbf{T}(H_3)))=\dim(H_2(\mathbf{T}(H_3)))=11$.
\end{proof}

\begin{theo}\label{homologyH4}
With the notation of \cite[Appendix C, Table C.2]{geck-pfeiffer}, we have
\[\forall 0\le i \le 4,~H_i(\mathbf{T}(H_4);\Q(\sqrt{5}))=\left\{\begin{array}{ccc}
\mathds{1} & \text{if} & i=0, \\[.5em]
4_t\oplus\overline{4_t}\oplus16_r' & \text{if} & i=1, \\[.5em]
6_s\oplus\overline{6_s}\oplus30_s\oplus\overline{30_s} & \text{if} & i=2, \\[.5em]
4_t'\oplus\overline{4_t'}\oplus16_r & \text{if} & i=3, \\[.5em]
\varepsilon & \text{if} & i=4. \end{array}\right.\]
\end{theo}

\begin{proof}
As for the previous proof, we let $\chi_H:=\sum_{I\subsetneq\widehat{S}}(-1)^{|I|}\mathds{1}{\uparrow}_{\pi(\widehat{H_4}_I)}^{H_4}$ and show that
\[\chi_H=\mathds{1}+\varepsilon-4_t-\overline{4_t}-4_t'-\overline{4_t'}+6_s+\overline{6_s}-16_r-16'_r+30_s+\overline{30_s}.\]
Since $\dim(H_1(\mathbf{T}(H_4)))=\dim(H_3(\mathbf{T}(H_4)))=24$ we obtain
\[H_2(\mathbf{T}(H_4))=30_s+\overline{30_s}+6_s+\overline{6_s}~~\text{and}~~H_1(\mathbf{T}(H_4))+H_3(\mathbf{T}(H_4))=4_t+4_t'+\overline{4_t}+\overline{4_t'}+16_r+16_r'.\]
Since the $\Q[H_4]$-module $H_1(\mathbf{T}(H_4),\Q)$ is a sub-quotient of the module
\[C_1^\mathrm{cell}(\mathbf{T}(H_4),H_4;\Q)=\sum_{I\subset\widehat{S}~;~|I|=3}\mathds{1}{\uparrow}_{\pi(\widehat{H_4}_I)}^{H_4}\]
and since
\[\left(C_1^\mathrm{cell}(\mathbf{T}(H_4)),16_r\right)_{H_4}=\left(C_1^\mathrm{cell}(\mathbf{T}(H_4)),4_t'\right)_{H_4}=0,\]
we get $(H_1(\mathbf{T}(H_4)),16_r)_{H_4}=(H_1(\mathbf{T}(H_4)),4_t')_{H_4}=0$. Because $H_1(\mathbf{T}(H_4))$ is rational, this implies that $H_1(\mathbf{T}(H_4))=4_t+\overline{4_t}+16_r'$.
\end{proof}

\begin{rem}
In \cite[\S 3]{ratcliffe-tschantz} and \cite[\S 2.2]{martelli_hyperbolic}, the homology of $\mathbf{T}(H_4)$ is also described, but only as a $\Z$-module. 
\end{rem}

Finally, we exhibit another algebraic meaning of the Euler characteristic of $\mathbf{T}(W)$. The Poincar\'e series of $\widehat{I_2(m)}$, $\widehat{H_3}$ and $\widehat{H_4}$ can be found in \cite[\S 3.1, Table 7.4 and Table 7.5]{chapovalov}. Using these expressions, we immediately obtain the following corollary:

\begin{cor}\label{euler+poincare=love}
Let $W$ be a finite irreducible Coxeter group. If $W(q)$ (resp. $\widehat{W}(q)$) denotes the Poincar\'e series of $W$ (resp. of its extension $\widehat{W}$), then we have
\[\chi(\mathbf{T}(W))=\left.\frac{W(q)}{\widehat{W}(q)}\right|_{q=1}\]

Moreover, $H_1(\mathbf{T}(W),\Q(W))$ contains the geometric representation as a direct summand and is irreducible if and only if $W$ is crystallographic.
\end{cor}

\begin{rem}
The irreducibility of $H_1(\mathbf{T}(W);\Q(W))$ when $W$ is crystallographic just comes from the fact that in this case, if $T$ denotes a maximal torus in the simply-connected group of the same type as $W$ and if $Q^\vee$ is the associated coroot lattice, then 
\[H_1(\mathbf{T}(W);\Q(W))=H_1(T;\Q)\simeq Q^\vee\otimes_\Z\Q=\Q\Phi^\vee.\]
This is the (dual of the) geometric representation of the Weyl group $W$, hence is irreducible.

With \cite{chapovalov} it can be seen that the quotient $W(q)/\widehat{W}(q)$ is a polynomial in $q$, but we cannot hope for a generalization of the Bott factorization theorem \cite[Theorem 6.3]{hiller_coxeter}, as $H_4(q)/\widehat{H_4}(q)$ is irreducible of degree 60.
\end{rem}

\subsection*{Acknowledgments}
We are much grateful to the referees for their careful reading and very useful and accurate suggestions, which improved the paper by a lot. We also warmly thank Owen Garnier, Daniel Juteau and David Chataur for their very appreciated help and the exciting discussions.

\printbibliography

@PROCEEDINGS{REU2019,
  INSTITUTION = {Univ. of Chicago},
  BOOKTITLE = {Proceedings of the REU},
  DATE = {2019},
}

@ARTICLE{babson-reiner,
  AUTHOR = {Babson, E. and Reiner, V.},
  DATE = {2004},
  DOI = {10.46298/dmtcs.312},
  JOURNALTITLE = {DMTCS},
  PAGES = {223--252},
  TITLE = {Coxeter-like complexes},
  VOLUME = {6},
}

@ARTICLE{billey2016,
  AUTHOR = {Billey, S. C. and Konvalinka, M. and Petersen, T. K. and Slofstra, W. and Tenner, B. E.},
  DATE = {2018},
  DOI = {10.37236/6741},
  ISSUE = {1},
  JOURNALTITLE = {The Electronic Journal of Combinatorics},
  TITLE = {Parabolic double cosets in {C}oxeter groups},
  VOLUME = {25},
  SHORTHAND = {BKP+16},
}

@BOOK{bourbaki456,
  AUTHOR = {Bourbaki, N.},
  PUBLISHER = {Springer},
  DATE = {2002},
  SERIES = {Elements of Mathematics},
  TITLE = {Lie groups and {L}ie algebras, chapters 4, 5, 6},
}

@BOOK{brown_buildings,
  AUTHOR = {Brown, K. S.},
  PUBLISHER = {Springer},
  DATE = {1989},
  TITLE = {Buildings},
}

@ARTICLE{chapovalov,
  AUTHOR = {Chapovalov, M. and Leites, D. and Stekolshchik, R.},
  PUBLISHER = {Atlantis Press},
  DATE = {2010},
  DOI = {10.1142/S1402925110000842},
  JOURNALTITLE = {J. of nonlinear Math. Physics},
  PAGES = {169--215},
  TITLE = {The {P}oincar\'{e} series of the hyperbolic {C}oxeter groups with finite volume of fundamental domains},
  VOLUME = {17},
}

@ARTICLE{chein,
  AUTHOR = {Chein, M.},
  DATE = {1969},
  JOURNALTITLE = {R. I. R. O},
  PAGES = {3--16},
  SERIES = {S\'{e}rie rouge},
  TITLE = {Recherche des graphes de matrices de {C}oxeter hyperboliques d'ordre au plus 10},
  VOLUME = {3},
}

@ARTICLE{chirivi-garnier-spreafico,
  AUTHOR = {Chiriv\`{i}, R. and Garnier, A. and Spreafico, M.},
  DATE = {2022},
  DOI = {https://doi.org/10.1016/j.exmath.2022.02.001},
  JOURNALTITLE = {Expositiones Mathematicae},
  NUMBER = {3},
  PAGES = {572--604},
  TITLE = {Cellularization for exceptional spherical space forms and the flag manifold of ${SL}_3(\mathbb{R})$},
  VOLUME = {40},
}

@ARTICLE{salvetti,
  AUTHOR = {Concini, C. De and Salvetti, M.},
  DATE = {2000},
  DOI = {10.4310/MRL.2000.v7.n2.a7},
  JOURNALTITLE = {Math. Research Letters},
  PAGES = {213--232},
  TITLE = {Cohomology of {C}oxeter groups and {A}rtin groups},
  VOLUME = {7},
}

@ARTICLE{davis_hyperbolic,
  AUTHOR = {Davis, M. W.},
  DATE = {1985},
  JOURNALTITLE = {Proc. of the AMS},
  TITLE = {A hyperbolic 4-manifold},
  VOLUME = {93},
}

@BOOK{tomdieck_transgroups,
  AUTHOR = {tom Dieck, T.},
  PUBLISHER = {Walter de Gruyter},
  DATE = {1987},
  SERIES = {De Gruyter studies in Mathematics},
  TITLE = {Transformation groups},
}

@ARTICLE{dyer90,
  AUTHOR = {Dyer, M.},
  DATE = {1990},
  DOI = {10.1016/0021-8693(90)90149-I},
  JOURNALTITLE = {Journal of Algebra},
  PAGES = {57--73},
  TITLE = {Reflection subgroups of {C}oxeter groups},
  VOLUME = {135},
}

@MANUAL{GAP4,
  ORGANIZATION = {The GAP~Group},
  URL = {https://www.gap-system.org},
  DATE = {2021},
  TITLE = {{GAP -- Groups, Algorithms, and Programming, Version 4.11.1}},
  SHORTHAND = {GAP21},
}

@BOOK{geck-pfeiffer,
  AUTHOR = {Geck, M. and Pfeiffer, G.},
  PUBLISHER = {Clarendon Press},
  DATE = {2000},
  SERIES = {London Math. Society monographs},
  TITLE = {Characters of finite {C}oxeter groups and {I}wahori--{H}ecke algebras},
  VOLUME = {21},
}

@MISC{hain_lectures,
  AUTHOR = {Hain, R.},
  DATE = {2014},
  EPRINT = {0812.1803},
  EPRINTCLASS = {math.AG},
  EPRINTTYPE = {arXiv},
  TITLE = {Lectures on Moduli Spaces of Elliptic Curves},
}

@BOOK{hatcher,
  AUTHOR = {Hatcher, A.},
  PUBLISHER = {Cambridge University Press},
  DATE = {2002},
  TITLE = {Algebraic Topology},
}

@BOOK{hiller_coxeter,
  AUTHOR = {Hiller, H.},
  PUBLISHER = {Pitman Advanced Publishing Program},
  DATE = {1982},
  TITLE = {Geometry of {C}oxeter groups},
}

@BOOK{humphreys-reflectiongroups,
  AUTHOR = {Humphreys, J. E.},
  PUBLISHER = {Cambridge University Press},
  DATE = {1992},
  SERIES = {Cambridge studies in advanced Mathematics},
  TITLE = {Reflection groups and {C}oxeter groups},
}

@BOOK{jones-wolfart,
  AUTHOR = {Jones, G. A. and Wolfart, J.},
  PUBLISHER = {Springer},
  DATE = {2016},
  SERIES = {Springer Monographs in Mathematics},
  TITLE = {Dessins d'Enfants on {R}iemann surfaces},
}

@BOOK{jost-riemann,
  AUTHOR = {Jost, J.},
  PUBLISHER = {Springer},
  DATE = {2002},
  EDITION = {Third Edition},
  SERIES = {Universitext},
  TITLE = {Compact {R}iemann surfaces},
}

@MISC{kirillovandjunior,
  AUTHOR = {Kirillov, A. and Jr, A. Kirillov},
  DATE = {2005},
  EPRINT = {math/0506118},
  EPRINTCLASS = {math.RT},
  EPRINTTYPE = {arXiv},
  TITLE = {Compact groups and their representations},
}

@INPROCEEDINGS{lin_lefschetz,
  AUTHOR = {Lin, E.},
  URL = {https://math.uchicago.edu/~may/REU2019/REUPapers/Lin,Edgar.pdf},
  CROSSREF = {REU2019},
  TITLE = {An overview and proof of the {L}efschetz fixed-point theorem},
}

@SOFTWARE{maple,
  AUTHOR = {Maplesoft},
  LOCATION = {Waterloo, Ontario},
  URL = {https://www.maplesoft.com/},
  DATE = {2019-10-16},
  TITLE = {Maple},
  VERSION = {17.00},
}

@MISC{martelli_hyperbolic,
  AUTHOR = {Martelli, B.},
  DATE = {2015},
  EPRINT = {1512.03661},
  EPRINTCLASS = {math.GT},
  EPRINTTYPE = {arXiv},
  TITLE = {Hyperbolic four-manifolds},
}

@ARTICLE{patera-twarock,
  AUTHOR = {Patera, J. and Twarock, R.},
  PUBLISHER = {IOP Publishing},
  DATE = {2002},
  DOI = {10.1088/0305-4470/35/7/306},
  JOURNALTITLE = {Journal of Physics A: Mathematical and General},
  NUMBER = {7},
  PAGES = {1551--1574},
  TITLE = {Affine extension of noncrystallographic {C}oxeter groups and quasicrystals},
  VOLUME = {35},
}

@ARTICLE{qi_thesis,
  AUTHOR = {Qi, D.},
  DATE = {2007},
  DOI = {10.4064/fm193-1-5},
  JOURNALTITLE = {Fundamenta Mathematicae},
  PAGES = {79--93},
  TITLE = {On irreducible, infinite, non-affine {C}oxeter groups},
  VOLUME = {193},
}

@BOOK{ratcliffe_hyperbolic,
  AUTHOR = {Ratcliffe, J. G.},
  PUBLISHER = {Springer},
  DATE = {2006},
  EDITION = {Second Edition},
  SERIES = {Graduate Texts in Mathematics},
  TITLE = {Foundations of hyperbolic manifolds},
  VOLUME = {149},
}

@ARTICLE{ratcliffe-tschantz,
  AUTHOR = {Ratcliffe, J. G. and Tschantz, S. T.},
  DATE = {2001},
  DOI = {10.1016/S0166-8641(99)00221-7},
  JOURNALTITLE = {Topology and its Applications},
  PAGES = {327--342},
  TITLE = {On the {D}avis hyperbolic 4-manifold},
  VOLUME = {111},
}

@BOOK{serre_ari,
  AUTHOR = {Serre, J. P.},
  PUBLISHER = {Presses Universitaires de France},
  DATE = {1970},
  TITLE = {Cours d'arithmétique},
}

@BOOK{spanier,
  AUTHOR = {Spanier, E. H.},
  PUBLISHER = {McGraw-Hill},
  DATE = {1981},
  TITLE = {Algebraic topology},
}

@BOOK{zelobenko,
  AUTHOR = {\v{Z}elobenko, D. P.},
  PUBLISHER = {AMS},
  DATE = {1973},
  SERIES = {Translations of Mathematical Monographs},
  TITLE = {Compact {L}ie groups and their representations},
  VOLUME = {40},
}

@ARTICLE{zimmermann_maximally_symmetric,
  AUTHOR = {Zimmermann, B.},
  DATE = {1992},
  DOI = {10.1016/0166-8641(92)90161-R},
  JOURNALTITLE = {Topology and its Applications},
  PAGES = {263--274},
  TITLE = {Finite group actions on handlebodies and equivariant {H}eegaard genus for 3-manifolds},
  VOLUME = {43},
}

@ARTICLE{zimmermann_hyperbolic,
  AUTHOR = {Zimmermann, B.},
  DATE = {1993},
  DOI = {10.1017/S0305004100075782},
  JOURNALTITLE = {Math. Proc. Camb. Phil. Soc.},
  PAGES = {87--90},
  TITLE = {On a hyperbolic 3-manifold with some special properties},
  VOLUME = {113},
}

@ARTICLE{davis-okun,
  AUTHOR = {Davis, M. W. and Dymara, J. and Januszkiewicz, T. and Okun, B.},
  TITLE = {Weighted $L^2$-cohomology of {C}oxeter groups},
  JOURNALTITLE = {Geometry and Topology},
  VOLUME = {11},
  DATE = {2007},
  PAGES = {47--138},
  DOI = {10.2140/gt.2007.11.47},
  SHORTHAND = {DDJO07},
}

\end{document}